\documentclass{article}

\usepackage{amssymb,amsmath,amsthm,amsfonts}
\usepackage{mathrsfs}
\usepackage{algorithm, algpseudocode}
\usepackage{multirow}
\usepackage{threeparttable}
\usepackage{booktabs}
\usepackage{makecell}
\usepackage{graphicx}
\usepackage[nohead,margin=1.0in]{geometry}
\usepackage[colorlinks=true, allcolors=blue]{hyperref}

\newtheorem{theorem}{Theorem}[section]
\newtheorem{lemma}{Lemma}[section]
\newtheorem{proposition}{Proposition}[section]

\numberwithin{equation}{section}

\newtheorem{definition}{Definition}[section]
\newtheorem{remark}{Remark}[section]
\newtheorem{assumption}[theorem]{Assumption}
\numberwithin{equation}{section}
\allowdisplaybreaks

\usepackage{color}
\usepackage{graphicx,epsfig}
\usepackage[tight]{subfigure}
\usepackage{cite}
\usepackage{times}
\usepackage{bm}
\usepackage{float}
\usepackage{epstopdf}
\usepackage{mathrsfs}
\usepackage{threeparttable}
\usepackage{verbatim}
\graphicspath{{./pic/}}

\renewcommand{\d}{\partial}

\newcommand{\dist}{\text{dist}}

\title{Numerical Analysis of Unsupervised Learning Approaches for Parameter Identification in PDEs}

\author{Siyu Cen\thanks{Department of Applied Mathematics, The Hong Kong Polytechnic University, Kowloon, Hong Kong, P.R. China. (\texttt{siyu2021.cen@connect.polyu.hk; zhizhou@polyu.edu.hk})}
\and Bangti Jin\thanks{Department of Mathematics, The Chinese University of Hong Kong, Shatin, New Territories, Hong Kong, P.R. China (\texttt{bangti.jin@gmail.com, b.jin@cuhk.edu.hk}).}
\and Qimeng Quan\thanks{School of Mathematics and Statistics, Wuhan University, Wuhan 430072, P. R. China (\texttt{quanqm@whu.edu.cn})}
\and Zhi Zhou\footnotemark[2]}

\date{}

\begin{document}

\maketitle

\begin{abstract}
Identifying parameters in partial differential equations (PDEs) represents a very broad class of applied inverse problems. In recent years, several unsupervised learning approaches using (deep) neural networks have been developed to solve PDE parameter identifications. These approaches employ neural networks as ansatz functions to approximate the parameters and / or the states, and have demonstrated impressive empirical performance. In this paper, we provide a comprehensive survey on these unsupervised learning techniques on one model problem, diffusion coefficient identification, from the classical numerical analysis perspective, and outline a general framework for deriving rigorous error bounds on the discrete approximations obtained using the Galerkin finite element method, hybrid method and deep neural networks. Throughout we highlight the crucial role of conditional stability estimates in the error analysis. 
\end{abstract}

\section{Introduction}

Estimating various physical parameters in partial differential equations (PDEs) represents a very large class of applied inverse problems and is of great importance
across many applied fields, e.g., medical imaging (including electrical impedance tomography,
diffuse optical tomography and photoacoustic tomography etc.), geophysical prospection, nano-optics, and non-destructive testing. Mathematically, these problems are frequently ill-posed in the sense that the solutions can be highly susceptible to the presence of the perturbation of the data, and in practice, the  measurement error is inevitable due to the limited precision of data acquisition techniques.

Numerically, specialized solution techniques known as regularization are needed for their stable and accurate resolution. This can be achieved via either variational regularization or iterative regularization. In variational regularization, one formulates the reconstruction task as solving a suitable optimization problem that involves a data-fitting term and a regularization term. The former measures the discrepancy between the model output with the
measured data, whereas the latter combats the inherent ill-posedness. Common penalty terms include Sobolev smoothness, sparsity promoting penalty (e.g., $\ell^1$ and elastic net) and total variation. The two terms are then properly balanced via a scalar known as the regularization parameter. See the monographs \cite{TikhonovArsenin:1977,EnglHankeNeubauer:1996,ScherzerGrasmair:2009,SchusterKaltenbacher:2012,ItoJin:2015} for in-depth mathematical treatments. Numerically, variational regularization leads to an infinite-dimensional PDE constrained optimization problem.  In practical computation, one has to discretize the governing equation and the objective functional. This can be achieved using the finite difference method, finite element method and more recently deep neural networks. The resulting finite-dimensional optimization problem can then be solved using standard optimization methods, yielding regularized reconstructions. This procedure has been routinely employed in practice with great success, and has been established as standard inversion techniques.

From the perspective of numerical analysis, it is also desirable to incorporate discretization parameters in the analysis. Indeed, a fine mesh leads to a more  accurate approximation of the forward map, but at the expense of increased computational expense, whereas a coarse mesh may compromise the accuracy of the reconstruction (due to the large discretization error). Therefore it is important to derive quantitative bounds on the discrete approximations to guide the choice of various algorithmic parameters. Unfortunately, the rigorous numerical analysis of discrete approximations lags far behind. This is due to the nonlinearity of the forward map for many PDE parameter identifications as well as the inherent ill-posedness of inverse problems. Consequently, so far in the literature, there were only a few studies on the error analysis of PDE parameter identifications prior to 2010s, prominently \cite{Falk:1983,Richter:1981numer,Kohn:1988}.

Meanwhile, for many PDE inverse problems, there are conditional stability estimates, derived using deep mathematical tools, e.g., Carleman estimates \cite{KlibanovTimonov:2004,Yamamoto:2009ip}. The estimates are conditional in the sense that they are valid only on a suitable subset of admissible parameters. It is a very natural idea to combine stability estimates with numerical procedures in the hope of developing numerical schemes with error analysis. This idea was first explored in the pioneering work of Cheng and Yamamoto \cite{ChengYamamoto:2000}, who analyzed Tikhonov regularization using conditional stability, and moreover proposed a new rule for choosing the regularization parameter based on the stability estimate; see Section \ref{sec:cond} for details. Since then the approach has been further studied in several works for both variational regularization \cite{BotHofmann:2016,EggerHofmann:2018,WernerHofmann:2020} and iterative regularization \cite{deHoopQiu:2012,QiuScherzer:2015}.

In view of these observations, it is of great interest to push the frontier further into the numerical analysis of discrete reconstruction schemes for PDE parameter identifications. Following the folklore theorem of Peter Lax in classical numerical analysis, error analysis usually combines (conditional) stability estimates with consistency bounds. This paradigm underpins existing works on the numerical analysis of PDE parameter identifications and uncertainty quantification of statistical inverse problems etc.

Over the last few years,  deep learning has achieved astounding successes in many diverse areas, e.g., computer vision, natural language processing and autonomous driving, and has completely transformed the landscape of these areas. It has also permeated into scientific computing, forming the burgeoning area of scientific machine learning. Unsurprisingly, the idea of using deep learning  to solve PDE inverse problems has also received a lot of recent attention (see, e.g., \cite{BarSochen:2021,Jin:2022CDII,jin:2017deep,Raissi:2019,Mitusch:2021,Cen:2024,Jin:2024conductivity}). Roughly, these approaches can be classified into two groups, one is operator learning in which one approximates the map from the observational data to the unknown parameters directly, and the other is function approximator, in which one uses DNNs as the ansatz space to approximate the unknown parameter and / or the state. Both approaches have progressed significantly and have demonstrated impressive empirical performance in a myriad of important applied inverse problems in terms of either the reconstruction quality or computational efficiency or both. One can find an excellent overview of diverse practical aspects of these approaches in the recent survey \cite{Tanyu:2023}. In particular, unsupervised schemes are very attractive due to their potential in addressing high-dimensional or highly noisy data, for which more traditional methods often face significant computational challenges, and appear very highly promising in several challenging scenarios.

The theoretical underpinnings of these approaches are largely still in their infancy and still await systematic investigations. In particular, the mechanism behind the empirical successes remains mysterious.  This has greatly hindered the wider applications of these approaches in several areas of critical safety. Therefore, it is of great importance to develop mathematical guarantees for relevant DNN-based methods. In this review, we systematically survey recent developments for a class of unsupervised learning approaches to solve PDE parameter identifications from the perspective of numerical analysis. The main strategy is to integrate conditional stability into the design and analysis of unsupervised numerical schemes to provide relevant mathematical guarantees.

The rest of the paper is organized as follows. In Section \ref{sec:cond}, we recall the concept of conditional stability in theoretical inverse problems, which is the main technical tool that will be used in the error analysis. In Section \ref{sec:FEM}, we discuss existing works using the Galerkin finite element method \cite{Brenner:2002,Cialet:2002,Thomee:2006}. Then in Section \ref{sec:Hybrid}, we discuss hybrid approaches that combine classical approaches with DNNs. In Section \ref{sec:NN}, we survey relevant approaches that are purely based on DNNs. We conclude the discussion in Section \ref{sec:concl}, including several potential research questions. Throughout, we use the standard notation for Sobolev spaces \cite{AdamsFournier:2003}; see also Appendix \ref{sec:app}. For a given set $D$, the notation $(\cdot,\cdot)_{L^2(D)}$ denotes the $L^2(D)$ inner product, and it is abbreviated to $(\cdot,\cdot)$ for $D=\Omega$. The notation $c$ denotes a generic constant which may change at each occurrence, but it is always independent of the discr
etization parameter, noise level and regularization parameter.

\section{Conditional stability}\label{sec:cond}

Establishing stability estimates represents one central theme in the area of theoretical inverse problems. These stability estimates are generally conditional in the sense that they are valid only on a suitable subset of the admissible parameters. The estimates quantify the degree of ill-posedness of concrete inverse problems, and their derivation often relies on deep mathematical tools. Stability estimates are also indispensable in the numerical analysis of practical reconstruction schemes, following the folklore theorem of Lax. Below we describe the abstract  framework.

\subsection{Conditional stability in abstract spaces}\label{subsec:cond_abst}

Let  $\mathcal{X}$ and $\mathcal{Y}$ be two Banach spaces equipped with the norms $\|\cdot\|_{\mathcal{X}}$ and $\|\cdot\|_{\mathcal{Y}}$, and $K: \mathcal{X} \to \mathcal{Y}$ be a densely defined, bounded and injective mapping. Consider the following operator equation:
\begin{equation}\label{eqn:abs-gov}
    Kx=y,
\end{equation}
with $x\in \mathcal{X}$ and $y\in \mathcal{Y}$.
Parameter identifications for PDEs are concerned with recovering the unknown coefficient $x^\dag\in \mathcal{X}$ from the noisy measurement data $y^\delta$ on the exact state $y^\dag = Kx^\dag$, with the accuracy measured by the noise level
\begin{equation*}
    \delta = \|y^\delta-y^\dag\|_{\mathcal{Y}}.
\end{equation*}
Throughout, we do not assume the linearity of the operator $K$, commonly referred to as the forward map. The operator $K$ is often implicitly defined by the underlying PDE. This problem is typically unstable in the sense of Hadamard: a small data perturbation of the data $y^\delta$ can lead to significant deviation of the corresponding solution $x$.  The relevant stability estimates for parameter identifications are typically conditional. Cheng and Yamamoto \cite{ChengYamamoto:2000} formulated the following definition on conditional stability.
\begin{definition}\label{def:cond_stab}
Let $\mathcal{Z} \subset \mathcal{X}$  be a Banach space $($equipped with the norm $\|\cdot\|_{\mathcal{Z}})$ with the continuous embedding relation $\mathcal{Z} \hookrightarrow \mathcal{X}$. Fix some $M>0$, and define the admissible set $\mathcal{A}_M$ by
\begin{equation*}
    \mathcal{A}_M=\left\{x \in \mathcal{Z}:\|x\|_{\mathcal{Z}} \leq M\right\},
\end{equation*}
and choose $\mathcal{Q} \subset \mathcal{Z}$ suitably.
Let the function $\omega\equiv\omega(\eta):[0,\infty)\mapsto[0,\infty)$ be monotonically increasing, satisfying $\lim_{\eta \to 0^+} \omega(\eta)=0 $.
    The conditional stability estimate is said to hold for the operator equation $K x=y$, if for a given $M>0$, there exists $c=c(M)>0$ such that
\begin{equation}\label{eqn:cond-stab-def}
  \left\|x_1-x_2\right\|_{\mathcal{X}} \leq c(M) \omega\left(\left\|Kx_1-Kx_2\right\|_{\mathcal{Y}}\right),\quad \forall x_1, x_2 \in \mathcal{A}_M \cap \mathcal{Q}.
\end{equation}
The function $\omega$ is called the modulus of the conditional stability.
\end{definition}

Note that the stability estimate \eqref{eqn:cond-stab-def} holds with respect to the $\|\cdot\|_{\mathcal{X}}$-norm only on the set $\mathcal{A}_M\cap \mathcal{Q}$, and hence the term \textit{conditional} stability. The gap between $\mathcal{Z}$ and $\mathcal{X}$ indicates the degree of ill-posedness of the inverse problem. Due to the ill-posedness nature of the inverse problem, one has to employ the idea of regularization for stable recovery. The most prominent method is Tikhonov regularization. It amounts to minimizing the following functional
\begin{equation}\label{eqn:abs-obj}
    J_\gamma(x) :=\|Kx-y\|_{\mathcal{Y}}^2 + \gamma\|x\|_{\mathcal{Z}}^2.
\end{equation}
The regularizing parameter
$\gamma>0$ balances the data discrepancy and regularization term. Note that the penalty $\|\cdot\|_{\mathcal{Z}}^2$ is chosen according to the stability estimate \eqref{eqn:cond-stab-def}, and conditional stability is directly incorporated into the construction of the numerical scheme. In practice, one can only access noisy data $y^\dag\approx y^\delta$, for which the regularized model \eqref{eqn:abs-obj} admits a minimizer  $x_\gamma^\delta\equiv x_\gamma(y^\delta)$ under minor conditions, also known as the regularized solution. The \textit{a priori} strategy to select the scalar $\gamma$ is crucial in obtaining the convergence of the regularized solutions as the noise level $\delta$ tends to zero. Generally one requires the condition $\lim _{\delta \rightarrow 0^{+}} \gamma(\delta)^{-1} \delta^2=\lim _{\delta \rightarrow 0^{+}} \gamma(\delta)=0$, and to derive  convergence rates, one has to additionally assume suitable source conditions. In a Hilbert space setting (i.e., both $\mathcal{Y}$ and $\mathcal{Z}$ are Hilbert spaces), it often takes the form of the range condition that the exact solution $x^\dag$ belongs to the range of fractional powers $(K^*K)^\nu$, $\nu>0$. However, characterizing the range condition is challenging for most parameter identification problems. Condition stability provides an \textit{a priori} parameter choice \cite[Theorem 2.1]{ChengYamamoto:2000}.

\begin{theorem}\label{thm:cond-stab in abstract form}
Let the conditional stability in Definition \ref{def:cond_stab} hold, and
\begin{equation*}
  Kx^\dag=y^\dag, \quad\mbox{with } x^\dag\in \mathcal{Q} \cap \mathcal{Z}.
\end{equation*}
Let $x_\gamma^\delta \in \mathcal{Q} \cap \mathcal{Z}$ satisfy
\begin{equation*}
 J_\gamma(x_\gamma^\delta) \leqslant \inf _{x \in \mathcal{Q} \cap \mathcal{Z}} J_\gamma(x) + c_0\delta^2.
\end{equation*}
Then for the choice $\gamma = O(\delta^2)$, there holds
\begin{equation*}
 \left\|x^\dag - x_\gamma^\delta\right\|_{\mathcal{X}}=\mathrm{O}(\omega(\delta)).
\end{equation*}
\end{theorem}

\begin{remark}
Theorem \ref{thm:cond-stab in abstract form} gives an a priori choice of the regularization parameter $\gamma$, i.e., $\gamma=O(\delta^2)$.  This choice differs markedly from the conventional requirement $\lim_{\delta\to 0^+}\gamma(\delta)^{-1}\delta^2 =\lim_{\delta\to0^+}\gamma(\delta)= 0$ in the standard theory of Tikhonov regularization.
\end{remark}

The modulus of continuity  $\omega$ in Theorem \ref{thm:cond-stab in abstract form} determines the convergence rate of the continuous regularized solution $x_\gamma^\delta$. It is emphasizing that
the overall approach in \eqref{eqn:abs-obj} is specifically adapted to the conditional stability. However, for many PDE inverse problems, the stability estimates often impose very strong regularity assumptions on the concerned class of parameters or the data, due to the proof strategy via Carleman estimates. Consequently, many existing stability estimates are not directly amenable with numerical treatments, e.g., $\|x\|_\mathcal{Z}=\|x\|_{C^2(\overline{\Omega})}$. There is a need to develop stability estimates that are more amenable with numerical treatment.

Below we illustrate conditional stability on diffusion coefficient identification. These results will serve as the benchmark for numerical analysis. Further results on conditional stability can be found in Section \ref{sec:NN}.

\subsection{Diffusion coefficient identification}\label{subsec:cond_IDP}

The stability of  diffusion coefficient identification has been extensively studied in the literature \cite{Richter:1981,Richter:1981numer,Falk:1983,Alessandrini:1986,Alessandrini:2017,Bonito:2017}. One of the earliest works, due to  Falk \cite{Falk:1983}, studied the second-order elliptic equation with a Neumann boundary condition over a smooth domain $\Omega\subset \mathbb{R}^2$:
\begin{equation}\label{eqn:gov_Falk}
    \left\{\begin{aligned}
      -\nabla\cdot\big(a\nabla u\big) &= f,\quad \mbox{in}\ \Omega,\\
      a\partial_{\nu} u &= g, \quad \mbox{on}\ \partial\Omega,
      \end{aligned}\right.
\end{equation}
where $\nu$ is the unit outward normal to the boundary $\partial\Omega$. The given functions $f$ and $g$ are the known source and boundary flux, respectively, satisfying the compatibility condition $ (f,1) + ( g,1)_{L^2(\partial\Omega)} = 0$. The solution $u\equiv u(a)$ satisfies the normalization condition $(u,1) = 0$. The diffusion coefficient $a$ belongs to the following admissible set
\begin{equation}\label{eqn:box_ad}
    \mathcal{A}:=\{a\in H^1(\Omega): 0<\underline{c}_a \leq a \leq \overline{c}_a<\infty\ \mbox{a.e. in}\ \Omega\}.
\end{equation}
Falk \cite{Falk:1983} derived the following stability estimate of \eqref{eqn:gov_Falk} using an energy argument under a structural assumption.
\begin{theorem}\label{thm:Falk_stab}
Let $a,a^\dag\in \mathcal{A}$ satisfy
    \begin{equation*}
       \max(\|a^\dagger\|_{H^1(\Omega)}),\|a\|_{H^1(\Omega)})\le M.
    \end{equation*}
Let $u\equiv u(a)$ and $u^\dag\equiv u(a^\dag)$ be the solutions to problem \eqref{eqn:gov_Falk} corresponding to $a$ and $a^\dagger$, respectively. Suppose that $u\in H^1(\Omega)$, $u^\dagger\in H^1(\Omega)\cap W^{2,\infty}(\Omega)$ and the following non-vanishing gradient condition holds    \begin{equation}\label{eqn:Falk_cond}
        \begin{aligned}
            &\exists \bm\nu\in\mathbb{S}^1 \mbox{ and }c_{\bm\nu}>0 \text{ s.t. }\nabla u^\dag(x)\cdot\bm\nu>c_{\bm\nu},\quad \forall  x\in\Omega.
        \end{aligned}
    \end{equation}
    Then there exists $c$ depending on $\overline{c}_a,M,c_{\bm\nu},\bm\nu,\Omega$ and $\|u^\dagger\|_{W^{2,\infty}(\Omega)}$ such that
\begin{equation}\label{eqn:Falk_stab}
        \|a-a^\dagger\|_{L^2(\Omega)}\le c\|\nabla (u-u^\dagger) \|_{L^2(\Omega)}^{\frac{1}{2}},
    \end{equation}
\end{theorem}

The proof relies on the weak formulations of $u^\dagger$ and $u(a)$ with the test function $\varphi= e^{-2k \bm x \cdot \bm\nu }(a^\dag-a)$, with the constant $k>\|\Delta u^\dagger\|_{L^{\infty}(\Omega)}/(2c_{\bm\nu})$. Indeed, direct computation leads to the following weighted stability
\begin{align*}
   \int_\Omega(a^\dag - a)^2\big( k\nabla u^\dag\!\cdot\!\bm\nu + \tfrac12\Delta u^\dag \big)e^{-2k\bm x\cdot\bm \nu}\ {\rm d}x \leq  c\|\nabla(u - u^\dag)\|_{L^2(\Omega)},
\end{align*}
where $c$ depends on $\overline{c}_a,k,M,\bm\nu$ and $\Omega$.  By the choice of the constant $k$, we can remove the weight function and get the desired $L^2(\Omega)$ estimate.
\begin{remark}\label{rmk:IDP_abst}
Theorem \ref{thm:Falk_stab} is an example of conditional stability: $\mathcal{X}=L^2(\Omega)$, $\mathcal{Y}=H^1(\Omega)$ and $K:\mathcal{X}\to \mathcal{Y}$ with $K(a)=u$ denotes the coefficient to solution map. The problem data lies in $\mathcal{Z}=H^1(\Omega)$, and $\mathcal{U}_M=\left\{a \in H^1(\Omega):\|a\|_{H^1(\Omega)} \leq M\right\}$. The constant $c$ in the estimate \eqref{eqn:Falk_stab}
depends on both the upper bound $M$ and the quantity $\|u^\dagger\|_{W^{2,\infty}(\Omega)}$. Indeed, by the elliptic regularity theory, further regularity conditions on $a^\dagger$ should be imposed in order to ensure $u^\dagger\in W^{2,\infty}(\Omega)$.
\end{remark}

The stability result in Theorem \ref{thm:Falk_stab} relies on the condition \eqref{eqn:Falk_cond}, i.e.,  $\nabla u^\dagger$ does not vanish along a given direction $\nu$. This type of non-vanishing gradient assumption was investigated in \cite{Alessandrini:1986} with a zero source and a nonzero Dirichlet boundary condition in the two-dimensional case:
\begin{equation}\label{eqn:gov_Alessandrini}
    \left\{\begin{aligned}-\nabla\!\cdot\!\big(a\nabla u\big) &= 0,\quad \mbox{in}\ \Omega\subset \mathbb{R}^2,\\
    u &= g, \quad \mbox{on}\ \partial\Omega,
    \end{aligned}\right.
\end{equation}
and the following weighted stability result holds \cite[Lemma 2.1]{Alessandrini:1986}.
\begin{lemma}\label{lem:Alessandrini_weight_stab}
    Let the domain $\Omega\subset \mathbb{R}^2$ be bounded, $C^2$-smooth and simply connected.  Let $g\in C^2(\partial\Omega)$ and  $a^\dagger, a\in W^{1,\infty}(\Omega)\cap\mathcal{A}$ satisfy
    \begin{equation*}
       \max(\|a^\dagger\|_{W^{1,\infty}(\Omega)},\|a\|_{W^{1,\infty}(\Omega)})\le M.
    \end{equation*}
    Suppose $a=a^\dagger$ on $\partial\Omega$. Then with $u^\dag\equiv u(a^\dag)$ and $u\equiv u(a)$, there holds
    \begin{equation}\label{eqn:Alessandrini_weight_stab}
        \int_\Omega|a^\dagger-a||\nabla u^\dagger|^2\d x\le c\|u^\dagger-u\|_{L^2(\Omega)}^\frac14,
    \end{equation}
    where $c$ only depends on $\underline{c}_a$, $\overline{c}_a$, $M$, $\Omega$, $\theta$ and $\|g\|_{C^2(\partial\Omega)}$.
\end{lemma}

The proof of Lemma \ref{lem:Alessandrini_weight_stab} is based on an energy argument using test functions $\pm \varphi$ with $\varphi=\lambda^{-1}\min\{ (a^\dagger-a)_+,\lambda\} u$, with $\lambda>0$ to be chosen and $(a^\dagger-a)_+ $ being the positive part. Then $\varphi=0$ if $a^\dagger-a\le 0 $, $\varphi= \lambda^{-1}  ( a^\dagger-a ) u $ if $0< a^\dagger-a\le \lambda $ and $\varphi=   u $ if $ a^\dagger-a > \lambda $. Using the weak formulation of \eqref{eqn:gov_Alessandrini}, we have
\begin{equation*}
    ((a^\dagger-a)\nabla u^\dag, \nabla\varphi )= (a\nabla (u- u^\dagger), \nabla\varphi).
\end{equation*}
The choice of the test function $\varphi$ enables controlling $\int_{|a^\dagger-a|\le \lambda}|a^\dagger-a||\nabla u^\dagger|^2\d x$ by $\lambda$ and controlling  $\int_{|a^\dagger-a|> \lambda}|a^\dagger-a||\nabla u^\dagger|^2\d x$ by $\|\nabla(u^\dagger-u)\|_{L^2(\Omega)}$ and $\lambda^{-1}$. By the interpolation inequality and taking $\lambda=\|u^\dagger-u\|_{L^2(\Omega)}^{\frac14}$, we obtain the desired result.

To remove the weighted function $|\nabla u^\dagger|^2$, Alesandrini \cite{Alessandrini:1986} proposed the following structural condition on the boundary data $g$ (when $f\equiv0$):
\begin{equation}\label{eqn:Alessandrini_cond}
    \text{$g$ has $N$ relative maxima and minima on $\partial\Omega$,}
\end{equation}
with $N$ being the total number of maxima and minima of the Dirichlet boundary condition $g$ over the boundary $\partial\Omega$.
Under condition \eqref{eqn:Alessandrini_cond}, $|\nabla u^\dagger|$ vanishes only at a finite number of interior points and only with a finite multiplicity. Moreover, the number of critical points and their multiplicities can be controlled in terms of $N$, which yields a lower bound on the weight $|\nabla u^\dag|$. The following stability result is direct from Lemma \ref{lem:Alessandrini_weight_stab} and the estimate of the lower bound on $|\nabla u^\dagger|$.
\begin{theorem}\label{thm:Alessandrini_stab}
    Let the hypothesis of Lemma \ref{lem:Alessandrini_weight_stab} and condition \eqref{eqn:Alessandrini_cond} hold. Then for any $\rho>0$, there holds
    \begin{equation}\label{eqn:Alessandrini_stab}
        \|a^\dagger-a\|_{L^\infty(\Omega_{\rho})}\le c\|u^\dagger-u\|_{L^2(\Omega)}^{\frac{1}{4(2N+1)}},
    \end{equation}
    where $\Omega_{\rho}=\{x\in \Omega: \mathrm{dist}(x,\partial\Omega)\ge \rho\}$ and $c$ depends only on $\underline{c}_a$, $ \overline{c}_a$, $M$, $\Omega$, $\rho$, $\|g\|_{C^2(\partial\Omega)}$ and $\max_{\partial\Omega} g-\min_{\partial\Omega} g$.
\end{theorem}

\begin{remark}\label{rmk:Alessandrini}
The stability estimate \eqref{eqn:Alessandrini_stab} is valid on the interior domain $\Omega_{\rho}$, since the lower bound estimate for $|\nabla u^\dagger|$ holds only within $\Omega_{\rho}$. Alessandrini et al \cite{Alessandrini:2017} extended the result to higher dimensions $(d\ge 3)$, and derived the  stability estimate on the interior domain $\Omega_{\rho}$ without condition \eqref{eqn:Alessandrini_cond}. They also proposed a structural condition similar to \eqref{eqn:Alessandrini_cond} to derive  the stability on the whole domain $\Omega$.
The argument can also be applied to other elliptic inverse problems, e.g., recovering Lam\'{e} parameters in linear elasticity \cite{Fazio:2017} and Winkler plate \cite{alessandrini:2024stable}.
\end{remark}

More recently Bonito et al \cite{Bonito:2017} established a novel stability estimate under mild assumptions on the problem data (with a zero Dirichlet boundary):
\begin{equation}\label{eqn:inv_cond_gov}
    \left\{\begin{aligned}-\nabla\cdot \big(a\nabla u\big) &= f,\quad \mbox{in}\ \Omega,\\
    u &= 0, \quad \mbox{on}\ \partial\Omega.
    \end{aligned}\right.
\end{equation}
Based on an energy argument and a new test function $\varphi = \frac{a^\dag-a}{a^\dag}u^\dag\in H_0^1(\Omega)$, Bonito et al \cite{Bonito:2017} established the following H{\"o}lder stability estimate.
\begin{theorem}\label{thm:Bonito_stab}
Let $a,a^\dag \in H^1(\Omega)\cap\mathcal{A}$ satisfy $\max(\|a^\dag\|_{H^1(\Omega)}, \|a\|_{H^1(\Omega)})\!\leq\! M$, and $f\in L^\infty(\Omega)$.
Then with $u^\dag\equiv u(a^\dag)$ and $u\equiv u(a)$, there holds
\begin{equation*}    \int_{\Omega}\Big|\frac{a^\dag - a}{a^\dag}\Big|^2\Big(fu^\dag + a^\dag|\nabla u^\dag|^2\Big)\ {\rm d}x \leq c\|\nabla( u^\dag - u)\|_{L^2(\Omega)},
\end{equation*}
where $c$ depends on $\underline{c}_a$,  $\overline{c}_a$, $M$, $d$, $\Omega$ and $\|f\|_{L^\infty(\Omega)}$. Moreover, if in addition, there exists $\beta\ge 0$ such that the following positivity condition holds \begin{equation}\label{eqn:Bonito_cond}
  \mbox{\rm PC}(\beta):\quad fu^\dag + a^\dag|\nabla u^\dag|^2 \geq c_0\dist(x,\partial\Omega)^\beta,
\end{equation}
then there holds
\begin{equation}\label{eqn:Bonito_stab}
        \|a^\dagger-a\|_{L^2(\Omega)} \leq
         c\|\nabla( u^\dagger - u)\|_{L^2(\Omega)}^{\frac{1}{2(1+\beta)}},
    \end{equation}
where $c$ depends on $c_0$, $\underline{c}_a$, $\overline{c}_a$, $M$, $d$, $\Omega$ and $\|f\|_{L^\infty(\Omega)}$.
\end{theorem}
\begin{proof}
    By the weak formulations of $u^\dag$ and $u(a)$, there holds \begin{equation*}
        ((a^\dag-a)\nabla u^\dag,\nabla \varphi) = (a\nabla(u(a) -u^\dag ),\nabla \varphi),\quad \forall \varphi\in H_0^1(\Omega).
    \end{equation*}
    Let $\varphi = \frac{a^\dag-a}{a^\dag}u^\dag$. Then clearly,
    \begin{align*}
        \nabla \varphi = \frac{\nabla(a^\dag-a)}{a^\dag}u^\dag - \frac{(a^\dag-a)\nabla a^\dag}{(a^\dag)^2}u^\dag + \frac{a^\dag-a}{a^\dag}\nabla u^\dag.
    \end{align*}
    Using the regularity estimate $\|u^\dagger\|_{L^\infty(\Omega)}\le c\|f\|_{L^\infty(\Omega)}$ \cite[Theorem 8.15]{Gilbarg:1977}, where $c$ only depends on $\underline{c}_a$, $\overline{c}_a$, $\Omega$, and the \textit{a priori} estimate
    $$\|\nabla u^\dagger\|_{L^2(\Omega)}\le c\|f\|_{H^{-1}(\Omega)}\le c\|f\|_{L^\infty(\Omega)},$$
we deduce
    \begin{align*}
        \|\nabla \varphi\|_{L^2(\Omega)} &\leq \Big\|\nabla\Big(\frac{a^\dag-a}{a^\dag}\Big)\Big\|_{L^2(\Omega)} \big\|u^\dag\big\|_{L^\infty(\Omega)} + \Big\|\frac{a^\dag-a}{a^\dag}\Big\|_{L^\infty(\Omega)}\big\|\nabla u^\dag\big\|_{L^2(\Omega)} \\
        & \leq c\big(\|\nabla a^\dag\|_{L^2(\Omega)} + \|\nabla a\|_{L^2(\Omega)}+1\big)\|f\|_{L^\infty(\Omega)}\leq c,
    \end{align*}
    with $c$ depending on $\underline{c}_a$, $\overline{c}_a$, $M$, $d$, $\Omega$ and $\|f\|_{L^\infty(\Omega)}$. Hence $\varphi\in H^1_0(\Omega)$. By H\"{o}lder's inequality, we arrive at
    \begin{equation*}
        ((a^\dag-a)\nabla u^\dag,\nabla \varphi) \leq c\|\nabla( u^\dag - u)\|_{L^2(\Omega)}.
    \end{equation*}
For the term $((a^\dag-a)\nabla u^\dag,\nabla \varphi)$, by integration by parts, we get
    \begin{equation*}
  ((a^\dag-a)\nabla u^\dag,\nabla \varphi) = -\Big(\nabla\Big(\frac{a^\dag-a}{a^\dag}\Big),a^\dag \varphi \nabla u^\dag\Big)-\Big(\frac{a^\dag-a}{a^\dag} \varphi,\nabla\cdot(a^\dag\nabla u^\dag)\Big).
    \end{equation*}
Using the relation $f=-\nabla\cdot(a^\dag\nabla u^\dag)$ in $\Omega$, we deduce
\begin{align*}
  ((a^\dag-a)\nabla u^\dag,\nabla \varphi) & = \frac{1}{2}((a^\dag-a)\nabla u^\dag,\nabla \varphi) \\&\quad -\frac12\bigg(\nabla\Big(\frac{a^\dag-a}{a^\dag}\Big),a^\dag \varphi \nabla u^\dag\bigg) +\frac12\Big(\frac{a^\dag-a}{a^\dag} \varphi, f\Big).
\end{align*}
This and the choice $\varphi=\frac{a^\dag-a}{a^\dag}u^\dag$ yield the identity
\begin{equation*}
     ((a^\dag-a)\nabla u^\dag,\nabla \varphi) = \frac12\int_{\Omega}\Big|\frac{a^\dag - a}{a^\dag}\Big|^2\Big(fu^\dag + a^\dag|\nabla u^\dag|^2\Big)\ {\rm d}x.
\end{equation*}
This proves the first assertion. Now we partition the domain $\Omega$ into two disjoint sets $\Omega=\Omega_\rho\cup\Omega_\rho^c$:
\begin{equation*}
 \Omega_\rho = \{ x\in\Omega: ~ \dist(x,\partial\Omega) \ge \rho \} \quad \text{and}\quad  \Omega_\rho^c = \Omega\backslash \Omega_\rho,
\end{equation*}
with $\rho>0$ to be chosen. On the subdomain $\Omega_\rho$, we have
\begin{align*}
 \int_{\Omega_\rho} (a^\dag-a)^2 \,\d x &=  \rho^{-\beta} \int_{\Omega_\rho} (a^\dag-a)^2 \rho^{\beta} \,\d x \\
 &\le  \rho^{-\beta} \int_{\Omega_\rho} (a^\dag-a )^2 \mathrm{dist}(x,\partial\Omega)^\beta \,\d x\\
 &\le  c_0\overline{c}_a^2 \rho^{-\beta} \int_{\Omega_\rho} \Big(\frac{a^\dag-a }{a^\dag}\Big)^2  \big(  a^\dag | \nabla u^\dag  |^2 + fu^\dag \big) \,\d x\\
&\le c\rho^{-\beta}\|\nabla (u^\dagger-u)\|_{L^2(\Omega)}.
\end{align*}
On the subdomain $\Omega_\rho^c$, by the box constraint of $\mathcal{A}$, we have
\begin{align*}
 \int_{\Omega_\rho^c} (a^\dag-a )^2 \,\d x \le   c |\Omega_\rho^c|  \le c\rho.
\end{align*}
By balancing these two terms with the choice $\rho=\|\nabla (u^\dagger-u)\|_{L^2(\Omega)}^{\frac{1}{\beta+1}}$, we arrive at the desired estimate.
\end{proof}
\begin{remark}\label{rmk:Bonito_positive_cond}
To derive the standard $L^2(\Omega)$ estimate, Bonito et al \cite{Bonito:2017} proposed the  positivity condition \eqref{eqn:Bonito_cond}. It holds under mild regularity assumptions on the problem data. For example, PC(2) holds if $\Omega$ is a Lipschitz domain, $a^\dagger \in \mathcal{A}$, and $f \in L^2(\Omega )$ with $f\geq c_f>0$ in $\Omega$. Further, if $\Omega$ is a $C^{2,\mu}$ domain with $\mu\in(0,1)$, $a^\dag\in C^{1,\mu}(\Omega)\cap \mathcal{A}$, $f\in C^\mu(\Omega)$ with $f\geq c_f>0$, then PC(0) holds. These results follow from the strong maximum principle and Schauder estimate for second-order elliptic equations, the decay rate of Green's function near the boundary $\partial\Omega$; see, e.g., \cite[Corollaries 3.4 and 3.8]{Bonito:2017} for details.
\end{remark}
\begin{remark}\label{rmk:Bonito_const_dep}
Note that the stability result in Theorem \ref{thm:Falk_stab} relies on the test function $\varphi=(a^\dagger-a)e^{-2k\bm{x}\cdot \bm{\nu}}$. The exponential type suggests that the constant $c$ in the estimate \eqref{eqn:Falk_stab}  increases exponentially with respect to the quantity $\|u^\dagger\|_{W^{2,\infty}(\Omega)}$. In contrast, Theorem \ref{thm:Bonito_stab} utilizes the test function $\varphi=\frac{(a^\dagger-a)}{a^\dagger}u^\dagger$, and the constant $c$ in \eqref{eqn:Bonito_stab} depends linearly on $\|f\|_{L^\infty(\Omega)}$ and $\max(\|a^\dagger\|_{L^2(\Omega)},\|a\|_{L^2(\Omega)} )$ \cite[Theorem 2.2]{Bonito:2017}.
\end{remark}


More recently, Zhang and Liu \cite{LiuZhang:2024} derived conditional H\"older stability for diffusion coefficient identification with measurements on a subdomain, under the assumption the diffusion coefficient $a$ is analytic.

\section{Galerkin FEM}\label{sec:FEM}
In this section we describe the classical Galerkin finite element method (FEM) for solving inverse problems, and discuss relevant error analysis. The discussions focus on diffusion coefficient identification. These results serve as benchmarks for neural solvers for PDE inverse problems.

First we recall briefly the Galerkin FEM approximation \cite{Cialet:2002,Brenner:2002,Thomee:2006}.
Let $\mathcal{T}_h$ be a shape regular quasi-uniform triangulation of the
domain $\Omega $ into $d$-simplexes with a mesh size $h>0$. Over the triangulation $\mathcal{T}_h$,
we define an $H^1(\Omega)$ conforming piecewise polynomial finite element space $V^r_h$ of degree $r$ by
\begin{equation*}
  V^r_h= \left\{v_h\in H^1(\Omega):\ v_h|_T \in P_r(T),\ \forall\, T \in \mathcal{T}_h\right\},
\end{equation*}
and $X^r_h = V^r_h\cap H_0^1(\Omega)$, with $P_r(T)$ being the space of polynomials of degree at most $r$ over the element $T$. For the case $r=1$, we suppress the superscript $r$ and write $V_h$ and $X_h$ instead. The discussions below mostly focus on the case $r=1$: the concerned coefficient / parameter often has low regularity and thus using high-order elements does not bring much benefit, and it is convenient to enforce the box constraint on the parameter, which is physically very important.

We extensively use the following operators on $X_h$ and $V_h$. We define the $L^2(\Omega)$ projection $P_h:L^2(\Omega)\mapsto X_h$ by
\begin{equation*}
     (P_h v,\varphi_h) =(v,\varphi_h) , \quad \forall v\in L^2(\Omega),\ \varphi_h\in X_h.
\end{equation*}
Then for $2\leq p\leq \infty$ and $s=0,1,2$, $k=0,1$ with $k\leq s$, there holds \cite[p. 32]{Thomee:2006}:
\begin{equation}\label{ineq:P_h-approx}
	\big\|v-P_hv\big\|_{W^{k,p}(\Omega)}\leq  ch^{s-k}\big\|v\big\|_{W^{s,p}(\Omega)}, \quad \forall v\in W^{s,p}(\Omega) \cap H_0^1(\Omega).
\end{equation}
Let $\Pi_h: C(\overline{\Omega})\mapsto V_h$ be the standard Lagrange interpolation operator. Then it satisfies the following error estimate for $s=1,2$ and $1 \le p\le \infty$ with $sp>d$ \cite[Corollary 4.4.24]{Brenner:2002}:
\begin{equation}\label{ineq:Pi_h-approx}
	\|v-\Pi_hv\|_{L^p(\Omega)} + h\|v-\Pi_hv\|_{W^{1,p}(\Omega)} \leq ch^s \|v\|_{W^{s,p}(\Omega)}, \,\forall v\in W^{s,p}(\Omega).
\end{equation}

\subsection{Error analysis for diffusion coefficient identification}\label{subsec:FEM_err}
Numerically identifying the diffusion coefficient $a$ using the Galerkin FEM has a long history \cite{yeh1986review}. The rigorous error analysis dates at least back to the seminal work of Falk \cite{Falk:1983}. Falk \cite{Falk:1983} investigated problem
\eqref{eqn:gov_Falk}
and derived error bounds based on the conditional stability in Theorem \ref{thm:Falk_stab}. In practice, the measurement data $z^\delta$ is corrupted by noise:
\begin{equation}\label{eqn:internal_data}
    z^\delta(x) = u^\dag(x) + \xi(x), \quad\forall x\in\Omega,
\end{equation}
with the pointwise noise $\xi\in L^2(\Omega)$ and the noise level $\delta$ given by
$$\delta:=\|u^\dag - z^\delta\|_{L^2(\Omega)}.$$
To recover the unknown parameter $a^\dag$ from the noisy data $z^\delta$, Falk employed the standard output least-squares formulation:
\begin{equation}\label{eqn:FEM_loss_Falk}
    \min_{a\in\mathcal{A}}J(a):= \|u(a) - z^\delta\|^2_{L^2(\Omega)}.
\end{equation}
Note that the formulation does not involve a penalty term, and the optimization problem is generally ill-posed, i.e., the existence of a minimizer is not ensured \cite{KohnStrang:1986}.
Falk discretized the objective $J(a)$ using the Galerkin FEM scheme in two steps: (i) Fix a mesh size $h>0$, and define a conforming  $P_r$ element space $V^r_h$ over a quasi-uniform triangulation $\mathcal{T}_h$ of $\Omega$; (ii) Use $V^r_h$ and $\mathcal{A}^{r+1}_h:= \mathcal{A}\cap V^{r+1}_h$ to discretize the state $u$ and the diffusion coefficient $a$, respectively.
Then the FEM approximation problem reads
\begin{equation}\label{eqn:FEM_loss_Falk_dis}
     \min_{a_h\in\mathcal{A}^{r+1}_h}J_h(a_h):= \|u_h(a_h) - z^\delta\|^2_{L^2(\Omega)},
\end{equation}
where the discrete state $u_h\equiv u_h(a_h)\in V_h^r$ with $(u_h,1)= 0$ satisfies
\begin{equation}\label{eqn:FEM_vari_Falk_dis}
    (a_h\nabla u_h,\nabla \varphi_h) = (f,\varphi_h) + ( g,\varphi_h)_{L^2(\partial\Omega)},\quad \forall \varphi_h\in V_h.
\end{equation}
Following the standard argument for finite-dimensional optimization problems, the discrete problem \eqref{eqn:FEM_loss_Falk_dis}-\eqref{eqn:FEM_vari_Falk_dis} admits at least one global minimizer $a_h^*\in \mathcal{A}_h^{r+1}$. Moreover, it depends continuously on the data. Falk \cite[Theorem 1]{Falk:1983} provided a first error bound on the discrete approximation $a_h^*$. The main challenge in the error analysis lies in the nonlinear dependence of the parameter $a_h^*$ on the given observation $z^\delta$.
\begin{theorem}\label{thm:FEM_Falk}
Let $a^\dag\in H^{r+1}(\Omega)\cap\mathcal{A}$, $u^\dag\in W^{r+3,\infty}(\Omega)$ and $\Gamma:=\{x\in\partial\Omega: \partial_\nu u^\dag>0\}\in C^{r+2}$. Let $a_h^*\in\mathcal{A}^{r+1}_h$ be a discrete minimizer to problem \eqref{eqn:FEM_loss_Falk_dis}-\eqref{eqn:FEM_vari_Falk_dis}.
Then under condition \eqref{eqn:Falk_cond},
there holds    \begin{equation}\label{eqn:FEM_Falk_err}
        \|a^\dag-a^*_h\|_{L^2(\Omega)} \leq c\big(h^r + h^{-2}\delta\big).
    \end{equation}
\end{theorem}
\begin{remark}\label{rmk:FEM_Falk_proof}
The conditional stability in Theorem \ref{thm:Falk_stab} motivates the error analysis of $a^\dag$. Taking $\varphi = (P_h a^\dag-a_h^*)e^{-2k\bm x\cdot\bm \nu}\in H^1(\Omega)$ yields
  \begin{align*}
    &\int_\Omega(P_h a^\dag - a_h^*)^2\big( k\nabla u^\dag\!\cdot\!\bm\nu + \tfrac12\Delta u^\dag \big)e^{-2k\bm x\cdot\bm \nu}\ {\rm d}x\\
    =&  -\int_\Omega (P_h a^\dag - a_h^* ) \nabla u^\dagger \cdot \nabla\varphi\ \d x \\
    =& \int_\Omega (P_h a^\dag - a_h^* ) \nabla u^\dagger \cdot \nabla (\Pi_h\varphi -\varphi)\ \d x-\int_\Omega (P_h a^\dag - a^\dagger ) \nabla u^\dagger \cdot \nabla \Pi_h\varphi\  \d x \\
    &+\int_\Omega  a_h^* \nabla( u^\dagger-u_h(a_h^*) )\cdot \nabla \Pi_h\varphi \ \d x,
    \end{align*}
where $P_h$ is the $L^2(\Omega)$ projection operator into $V_h^r$ and $\Pi_h$ is the Lagrange interpolation onto $V_h^{r+1}$. Using the approximation property of the FEM space $V_h^r$, one can derive the error of the state approximation $u_h(a_h^*)$ in the $L^2(\Omega)$ norm. Using the FEM inverse estimate and the stability result, we get the desired error bound.
\end{remark}
\begin{remark}\label{rmk:FEM_Falk_order}
The $L^2(\Omega)$ error bound in Theorem \ref{thm:FEM_Falk} depends explicitly on the noise level $\delta$ and the mesh size $h$, which in practice gives a useful guideline on an \textit{a priori} choice of the mesh size $h$. For example, by choosing $h\sim\delta^{\frac{1}{r+2}}$, one can achieve an convergence rate $O(\delta^{\frac{r}{r+2}})$. This rate is consistent with the conditional stability estimate in Theorem \ref{thm:Falk_stab}. Indeed, Theorem \ref{thm:Falk_stab} implies
    \begin{equation*}
        \|a^\dag-a \|_{L^2(\Omega)}^2\leq  c \|u^\dag - u(a) \|_{H^1(\Omega)}\|\varphi \|_{H^1(\Omega)},
    \end{equation*}
    where $\varphi= e^{-2k \bm x \cdot \bm\nu }(a^\dag-a)$ is the test function with the vector $\bm\nu$ satisfying condition \eqref{eqn:Falk_cond}. Then with $e^{-2k \bm x \cdot \bm\nu }\in W^{1,\infty}(\Omega)$, we obtain
    \begin{equation*}
        \|a^\dag-a \|_{L^2(\Omega)}^2\leq  c \|u^\dag - u(a) \|_{H^1(\Omega)}\|a^\dag-a\|_{H^1(\Omega)}.
    \end{equation*}
By Gagliardo-Nirenberg interpolation inequality \cite{Brezis:2018} for all $s>1$,
\begin{equation}\label{eqn:Gagliardo-Nirenberg}
\|v\|_{H^1(\Omega)}\leq c\|v\|^{1-\frac{1}{s} }_{L^2(\Omega)}\|v\|^{\frac{1}{s}}_{H^s(\Omega)},
    \end{equation}
    there hold for $a^\dag,a\in H^r(\Omega)$,
    \begin{align*}
        \|a^\dagger-a \|_{H^1(\Omega)} \leq  c\|a^\dagger-a \|_{L^2(\Omega)}^{1-\frac{1}{r}} \|a^\dagger-a\|_{H^{r}(\Omega)}^{\frac{1}{r}}\leq  c\|a^\dagger-a \|_{L^2(\Omega)}^{1-\frac{1}{r}},
    \end{align*}
    and under the a priori assumption $u^\dag, u(a)\in H^{r+2}(\Omega)$, there holds
    \begin{equation*}
    \begin{split}
        \|u^\dagger - u(a) \|_{H^1(\Omega)} &\leq  \|u^\dagger - u(a) \|_{L^2(\Omega)}^{1-\frac{1}{r+2}}  \|u^\dagger - u(a) \|_{H^{r+2}(\Omega)}^{ \frac{1}{r+2}}\\& \leq c\|u^\dagger - u(a) \|_{L^2(\Omega)}^{\frac{r+1}{r+2}}.
    \end{split}
    \end{equation*}
    Combining the last two estimates yields
    \begin{equation*}
        \|a^\dagger-a\|^2_{L^2(\Omega)}\leq c\|u^\dagger - u(a) \|_{L^2(\Omega)}^{\frac{r+1}{r+2}}\|a^\dagger-a \|_{L^2(\Omega)}^{1-\frac{1}{r}}.
    \end{equation*}
    Direct computation leads to the desired estimate
    \begin{equation*}
        \|a^\dagger-a\|_{L^2(\Omega)}\leq c\|u^\dagger - u(a) \|_{L^2(\Omega)}^{\frac{r}{r+2}}.
    \end{equation*}
Note that the error bound \eqref{eqn:FEM_Falk_err} involves the factor $h^{-2}$, due to the absence of any penalty on $a$. The regularizing effect is achieved solely with the FEM discretization, with the parameter $h$ controlling the amount of regularization.
\end{remark}

In 2010, Wang and Zou \cite{Wang:2010} improved the error analysis of diffusion coefficient identification with a zero Neumann boundary condition:
\begin{equation}\label{eqn:gov_Wang}
    \left\{\begin{aligned}
       -\nabla\cdot\big(a\nabla u\big) &= f,\quad \mbox{in}\ \Omega, \\  a\partial_{\nu} u &= 0, \quad \mbox{on}\ \partial\Omega.
    \end{aligned}\right.
\end{equation}
Following the well established variational regularization, one standard approach to obtain a numerically stable reconstruction is
the following regularized output least-squares method with an $H^1(\Omega)$ seminorm penalty
\begin{equation}\label{eqn:FEM_loss_Wang}
    \min_{a \in \mathcal{A}} {J_\gamma(a)}=\frac12 \|u(a) - z^\delta\|_{L^2(\Omega)}^2  + \frac\gamma2\|\nabla a \|_{L^2(\Omega) }^2,
\end{equation}
with {the admissible set $\mathcal{A}$ given by \eqref{eqn:box_ad}}, where $u\equiv u(a)\in H^1(\Omega)$ satisfies $(u, 1)=0$ and  the variational problem
\begin{equation}\label{eqn:FEM_vari_Wang}
 (a\nabla u,\nabla \varphi) =  (f,\varphi),\quad \forall  \varphi\in H^1(\Omega).
\end{equation}
The $H^1(\Omega)$ seminorm penalty is suitable for recovering a smooth diffusion coefficient. The penalty parameter $\gamma>0$  controls the strength of the penalty \cite{EnglKunischNeubauer:1989,ItoJin:2015}. Using the direct method in calculus of variation, it can be verified that for every $\gamma>0$, problem \eqref{eqn:FEM_loss_Wang}-\eqref{eqn:FEM_vari_Wang} has at least one global minimizer $a^*$, and further the sequence of minimizers converges subsequentially in the $H^1(\Omega)$ norm to a minimum seminorm solution as the noise level $\delta$ tends to zero, provided that $\gamma$ is chosen appropriately in accordance with $\delta$, i.e., $\lim_{\delta\to0^+}\gamma(\delta)^{-1}\delta^2= \lim_{\delta\to0^+}\gamma(\delta)=0$ (see, e.g., \cite{EnglKunischNeubauer:1989,ItoJin:2015}).

For the numerical implementation, Wang and Zou discretize both the diffusion coefficient $a$ and the state $u$ using the conforming piecewise linear FEM space $V_h$:
\begin{equation}\label{eqn:FEM_loss_Wang_dis}
     \min_{a_h\in\mathcal{A}_h}J_h(a_h):= \frac{1}{2}\|u_h(a_h) - z^\delta\|^2_{L^2(\Omega)}+\frac{\gamma}{2}\|\nabla a_h\|_{L^2(\Omega)}^2,
\end{equation}
where the discrete admissible set $\mathcal{A}_h=\mathcal{A}\cap V_h$ and the discrete state $u_h\equiv u_h(a_h)\in V_h$ with $( u_h,1 ) = 0$ satisfies
\begin{equation}\label{eqn:FEM_vari_Wang_dis}
    ( a_h\nabla u_h,\nabla \varphi_h) = ( f,\varphi_h), \quad \forall \varphi_h\in V_h.
\end{equation}
The well-posedness and convergence analysis of the discrete problem \eqref{eqn:FEM_loss_Wang_dis}-\eqref{eqn:FEM_vari_Wang_dis} are well established  \cite{Gutman:1990,Keung:1998,Zou:1998,HinzeKaltenbacher:2018}. For any fixed $h>0$, there
exists at least one discrete minimizer $a_h^*\in\mathcal{A}_h$ to the discrete problem \eqref{eqn:FEM_loss_Wang_dis}-\eqref{eqn:FEM_vari_Wang_dis}. Further, the sequence of discrete minimizers $\{a_h^*\}_{h>0}$ contains a subsequence that converges in the $H^1(\Omega)$ norm to a minimizer to problem \eqref{eqn:FEM_loss_Wang}-\eqref{eqn:FEM_vari_Wang}. The proof follows by a standard compactness argument in calculus of variation and the density of the space $V_h$ in $H^1(\Omega)$. Wang and Zou \cite{Wang:2010} derived the following error estimate on the approximation $a_h^*$ in terms of all algorithmic parameters, i.e., the mesh size $h$, regularization parameter $\gamma$ and noise level $\delta$.
\begin{theorem}\label{thm:FEM_Wang}
Let $a^\dag\in H^{2}(\Omega)\cap W^{1,\infty}(\Omega) \cap \mathcal{A}$, $u^\dag\in H^{2}(\Omega)\cap W^{1,\infty}(\Omega)$. Let $a_h^*\in\mathcal{A}_h$ be a discrete minimizer to problem \eqref{eqn:FEM_loss_Wang_dis}-\eqref{eqn:FEM_vari_Wang_dis}, and the following assumption be fulfilled: \begin{equation}\label{eqn:FEM_cond_Wang}
     \exists c_0>0 \text{ s.t. } c_0|\nabla u^\dag (x)|^2\ge f(x) \quad \text{ a.e. }x\in\Omega.
    \end{equation}
Then there holds
\begin{equation}\label{eqn:FEM_Wang_err}
        \|(a^\dag-a^*_h)\nabla u^\dagger\|_{L^2(\Omega)} \leq c\big(h^{\frac{1}{2}}\gamma^{-\frac{1}{2}} + h^{-\frac{1}{2}}\gamma^{-\frac{1}{4}}  \big)\big(h^2 + \delta + \gamma^{ \frac{1}{2}}  \big).
    \end{equation}
\end{theorem}
The proof of Theorem \ref{thm:FEM_Wang} is similar to Theorem \ref{thm:FEM_Falk}. First, using the test function $\varphi= \frac{a^\dag-a}{a^\dag}e^{-2c_0 \underline{c}_a^{-1} u^\dag }$, the following weighted stability estimate holds:
    \begin{equation*}
        \int_\Omega \frac{(a^\dag - a)^2}{(a^\dag)^2}\big(  2c_0\underline{c}_a^{-1} a^\dag|\nabla u^\dag|^2 - f \big)e^{-2c_0 \underline{c}_a^{-1} u^\dag }\d x \leq c\|\nabla(u(a) - u^\dag)\|_{L^2(\Omega)}.
    \end{equation*}
    The box constraint \eqref{eqn:box_ad} and the choice of $c_0$ imply
    \begin{equation*}
         \|(a^\dag-a )\nabla u^\dagger\|_{L^2(\Omega)}   \leq c\|\nabla(u(a) - u^\dag)\|^{\frac12}_{L^2(\Omega)},
    \end{equation*}
    This estimate motivates the test function $\varphi= \frac{a^\dag-a_h^*}{a^\dag}e^{-2c_0 \underline{c}_a^{-1} u^\dag }\in H^1(\Omega)$ in \eqref{eqn:gov_Wang} and $\varphi_h= \Pi_h \varphi\in V_h$ in \eqref{eqn:FEM_vari_Wang_dis}. Together with the \textit{a priori} estimate on $\|u_h(\Pi_h a^\dag)- u(a^\dag)\|_{L^2(\Omega)}$ (cf. Lemma \ref{lem:FEM_uq-uh} below), the desired  estimate \eqref{eqn:FEM_Wang_err} follows.
\begin{remark}\label{rmk:FEM_Wang_order}
Theorem \ref{thm:FEM_Wang} indicates that we can provide an error bound on the discrete approximation $a_h^*$ in the subregion in which the gradient $\nabla u^\dag$ is non-vanishing. The data regularity assumption in Theorem \ref{thm:FEM_Wang} is  weaker than that in Theorem \ref{thm:FEM_Falk}. The positivity condition \eqref{eqn:FEM_cond_Wang} holds if $f\in L^\infty(\Omega)$ and $u(a^\dag)\in C(\overline{\Omega})$ with nonvanishing flux. In contrast,  Theorem \ref{thm:FEM_Falk} requires nonvanishing flux in one specific direction, which is more restrictive.
The error estimate \eqref{eqn:FEM_Wang_err} provides a guidance for choosing the algorithmic parameters $h$ and $\gamma$. By choosing $h\sim \delta^{\frac{1}{2}}$ and $\gamma\sim \delta^2$, we arrive at
    \begin{equation*}
        \|(a^\dag-a^*_h)\nabla u^\dagger\|_{L^2(\Omega)} \leq c \delta^{\frac{1}{4}}.
    \end{equation*}
\end{remark}

Jin and Zhou \cite{Jin:2021Error} adapted the stability result in Theorem \ref{thm:Bonito_stab} to the numerical analysis of  problem \eqref{eqn:inv_cond_gov}. To recover
the diffusion coefficient $a^\dag$ from the given internal noisy data $z^\delta$ in \eqref{eqn:internal_data}, they employ the standard
output least-squares formulation with an $H^1(\Omega)$ seminorm penalty and then discretize the regularized problem by continuous piecewise linear elements:
\begin{equation}\label{eqn:FEM_loss_Jin_dis}
    \min_{a_h \in \mathcal{A}_h} J_{\gamma,h}(a_h)=\frac12  \|u_h(a_h) - z ^\delta\|_{L^2(\Omega)}^2  + \frac\gamma2\|\nabla a_h \|_{L^2(\Omega) }^2,
\end{equation}
subject to $a_h\in\mathcal{A}_h:= V_h\cap \mathcal{A}$ and $u_h\equiv u_h(a_h)\in X_h = V_h\cap H_0^1(\Omega)$ satisfying
\begin{align}\label{eqn:FEM_vari_Jin_dis}
(a_h \nabla u_h, \nabla \varphi_h)=(f,\varphi_h),\quad \forall \varphi_h \in X_h.
\end{align}
Jin and Zhou derived an error bound on the FEM approximation $a_h^*$ under the following regularity assumption on the problem data.

\begin{assumption}\label{Jin2021Error:ass-ellip}
$a^\dag \in H^2(\Omega) \cap W^{1,\infty}(\Omega) \cap \mathcal A$ and $f \in L^\infty(\Omega)$.
\end{assumption}
The elliptic regularity theory implies that the exact state $u^\dag$ satisfies
(see \cite[Lemma 2.1]{LiSun:2017} and \cite[Theorems 3.3 and 3.4]{GruterWidman:1982})
\begin{equation*}
   u^\dag \in H^2(\Omega) \cap W^{1,\infty}(\Omega)\cap H_0^1(\Omega).
\end{equation*}
Note that the regularity result requires only $a^\dag\in W^{1,\infty}(\Omega)\cap \mathcal{A}$.

The overall analysis strategy follows the folklore Lax theorem by combining conditional stability in Theorem \ref{thm:Bonito_stab} with  consistency bounds. The next lemma gives an approximation result on the intermediate state $u_h(\Pi_ha^\dag)$, which plays a crucial role in the derivation of the error bound of the state $u^\dag$ in Lemma \ref{lem:FEM_priori} below. The key ingredients include the approximation property $\|a^\dag - \Pi_ha^\dag\|_{L^2(\Omega)} \leq ch^2$ and the theory of the Galerkin FEM.
\begin{lemma}\label{lem:FEM_uq-uh}
    Let Assumption \ref{Jin2021Error:ass-ellip} be fulfilled. Then there holds
    \begin{equation*}
        \|u^\dag - u_h(\Pi_ha^\dag)\|_{L^2(\Omega)} \leq ch^2.
    \end{equation*}
\end{lemma}
\begin{proof}
    By a standard duality argument, we have
    \begin{equation*}
        \|u^\dag - u_h(a^\dag)\|_{L^2(\Omega)} + h\|\nabla(u^\dag - u_h(a^\dag))\|_{L^2(\Omega)}\leq ch^2.
    \end{equation*}
     Let $w_h= u_h(\Pi_ha^\dag) - u_h(a^\dag)$. Using the weak formulations of $u_h(\Pi_ha^\dag)$ and $u_h(a^\dag)$, we get for any $\varphi_h\in X_h$,
    \begin{align*}
        & (\Pi_ha^\dag\nabla w_h,\nabla \varphi_h) =  ((a^\dag-\Pi_ha^\dag)\nabla u_h(a^\dag),\nabla \varphi_h) \\
         =&  ((a^\dag-\Pi_ha^\dag)\nabla (u_h(a^\dag) - u^\dag),\nabla \varphi_h) + ((a^\dag-\Pi_ha^\dag)\nabla u^\dag,\nabla \varphi_h).
    \end{align*}
    Let $\varphi_h = w_h$. Then H\"{o}lder's inequality gives
    \begin{align*}
        \|\nabla w_h\|_{L^2(\Omega)}& \leq c\|a^\dag-\Pi_ha^\dag\|_{L^\infty(\Omega)}\|\nabla (u_h(a^\dag) - u^\dag)\|_{L^2(\Omega)} \\& \quad + c\|a^\dag-\Pi_ha^\dag\|_{L^2(\Omega)}\|\nabla u^\dag\|_{L^\infty(\Omega)} \\& \leq
         ch^2\big(\|a^\dag\|_{W^{1,\infty}(\Omega)}\|u^\dag\|_{H^2(\Omega)} + \|a^\dag\|_{H^2(\Omega)} \|\nabla u^\dag\|_{L^\infty(\Omega)}\big),
    \end{align*}
and by Poinc{a}r\'{e} inequality, $$\|w_h\|_{L^2(\Omega)} \leq
     c\|\nabla w_h\|_{L^2(\Omega)}\leq
      ch^2.$$
This and the triangle inequality complete the proof of the lemma.
\end{proof}

The next result gives the state approximation $ u_h(a_h^*) - u^\dag$ and an \textit{a priori} $L^2(\Omega)$ bound on $\nabla a_h^*$. It plays the role of consistency bounds in Lax theorem.
\begin{lemma}\label{lem:FEM_priori}
 Let $(a_h^*,u_h(a_h^*)) \in \mathcal{A}_h
\times X_h$ be a minimizing pair of problem \eqref{eqn:FEM_loss_Jin_dis}--\eqref{eqn:FEM_vari_Jin_dis}. Then there holds
\begin{equation*}
    \| u_h(a_h^*) - u^\dag  \|_{L^2(\Omega)} + \gamma^\frac12 \| \nabla a_h^* \|_{L^2(\Omega)} \le c(h^2 +\delta+\gamma^{\frac12}).
\end{equation*}
\end{lemma}
\begin{proof}
    By the minimizing property of $a_h^*$, since $\Pi_h a^\dag\in\mathcal{A}_h$, it follows that
    \begin{equation*}
      \| u_h(a_h^*) - z^\delta  \|^2_{L^2(\Omega)} + \gamma \| \nabla a_h^* \|^2_{L^2(\Omega)} \le \| u_h(\Pi_h a^\dag) - z^\delta  \|^2_{L^2(\Omega)} + \gamma \| \nabla \Pi_h a^\dag \|^2_{L^2(\Omega)}.
    \end{equation*}
    Since the interpolation operator $\Pi_h$ is $H^1(\Omega)$ stable, by Lemma \ref{lem:FEM_uq-uh} and the triangle inequality, we obtain
    \begin{align*}
       & \| u_h(a_h^*) - z^\delta  \|^2_{L^2(\Omega)} + \gamma \| \nabla a_h^* \|^2_{L^2(\Omega)} \\
        \leq &c\big( \| u_h(\Pi_h a^\dag) -u^\dag  \|^2_{L^2(\Omega)} +  \| u^\dag  - z^\delta  \|^2_{L^2(\Omega)} + \gamma  \big) \\
        \leq &c(h^4 + \delta^2 + \gamma).
    \end{align*}
    Then the triangle inequality completes the proof of the lemma.
\end{proof}

Now we can state an error bound on the discrete approximation $a_h^*$ under Assumption \ref{Jin2021Error:ass-ellip}. The proof proceeds in two steps: we first derive a weighed error estimate by taking $\varphi = \frac{a^\dag-a_h^*}{a^\dag}u^\dag\in H_0^1(\Omega)$ in the weak formulation and using Lemma \ref{lem:FEM_priori}; then we remove the weight function $a^\dag|\nabla u^\dag|^2 + fu^\dag $ and obtain the standard $L^2(\Omega)$ error estimate under the positivity condition PC($\beta$).
\begin{theorem}\label{thm:FEM_error}
Let Assumption \ref{Jin2021Error:ass-ellip} be fulfilled. Let $a^\dag$ be the exact diffusion coefficient and $a_h^* \in \mathcal{A}_h$ a minimizer of
problem \eqref{eqn:FEM_loss_Jin_dis}-\eqref{eqn:FEM_vari_Jin_dis}. Then with $\eta=h^2 +\delta+\gamma^{\frac12}$, there holds
\begin{equation*}
  \int_\Omega \Big(\frac{a^\dag-a_h^*}{a^\dag}\Big)^2 \Big(  a^\dag | \nabla u^\dag  |^2 + fu^\dag \Big) \,{\rm d} x\leq  c(h \gamma^{-\frac12}\eta + \min(h + h^{-1} \eta,1)) \gamma^{-\frac12}\eta.
\end{equation*}
Moreover, if condition \eqref{eqn:Bonito_cond} holds, then
\begin{equation*}
    \|a^\dag-a_h^*\|_{L^2(\Omega)}  \le c((h \gamma^{-\frac12}\eta + \min( h^{-1} \eta,1)) \gamma^{-\frac12}\eta)^{\frac1{2(1+\beta)}}.
\end{equation*}
\end{theorem}
\begin{proof}
Let $u^\dag=u(a^\dag)$. By integration by parts and the weak formulations of $u^\dag$ and $u_h(a_h^*)$, we have the following splitting for any $\varphi\in H_0^1(\Omega)$,
\begin{equation}\label{eqn:sp-00}
    \begin{split}
        &((a^\dag-a_h^*)\nabla u^\dag,\nabla \varphi) \\
        =&((a^\dag-a_h^*)\nabla u^\dag,\nabla(\varphi-P_h \varphi))+ (a^\dag\nabla u^\dag-a_h^*\nabla u^\dag,\nabla P_h \varphi) \\=&-(\nabla\cdot((a^\dag-a_h^*)\nabla u^\dag),  \varphi-P_h \varphi ) + (a_h^*\nabla (u_h(a_h^*) - u^\dag),\nabla P_h \varphi)\nonumber =: {\rm I}_1+{\rm I}_2.
    \end{split}
\end{equation}
It remains to bound the two terms ${\rm I}_1$ and ${\rm I}_2$ separately. By the triangle inequality, we have
\begin{align*}
 \| \nabla\cdot((a^\dag-a_h^*)\nabla u^\dag)\|_{L^2(\Omega)} & \le \|a^\dag-a_h^*\|_{L^\infty(\Omega)}\|\Delta u^\dag\|_{L^2(\Omega)} \\& \quad +  \| \nabla (a^\dag-a_h^*)\|_{L^2(\Omega)}  \| \nabla u^\dag \|_{L^\infty(\Omega)}.
\end{align*}
In view of Assumption \ref{Jin2021Error:ass-ellip} and the box constraint of $\mathcal{A}$, we derive
\begin{equation*}
 \| \nabla\cdot ((a^\dag-a_h^*)\nabla u)\|_{L^2(\Omega)}
 \le c (1+  \| \nabla a_h^* \|_{L^2(\Omega)}).
\end{equation*}
This estimate and the Cauchy-Schwarz inequality imply
\begin{equation*}
 |{\rm I}_1 | \le  c(1+\| \nabla a_h^* \|_{L^2(\Omega)} ) \|   \varphi-P_h \varphi \|_{L^2(\Omega)}.
\end{equation*}
Let $\varphi\equiv \frac{a^\dag-a_h^*}{a^\dag} u^\dag$. Then
\begin{equation*}
\nabla  \varphi = \bigg(\frac{\nabla(a^\dag-a_h^*)}{a^\dag} - \frac{(a^\dag-a_h^*)\nabla a^\dag}{(a^\dag)^2}\bigg) u^\dag + \frac{a^\dag-a_h^*}{a^\dag}\nabla u^\dag,
\end{equation*}
By  Lemma \ref{lem:FEM_priori}, the box constraint of the admissible set $\mathcal{A}$ and Assumption \ref{Jin2021Error:ass-ellip}, we have
\begin{align*}
  \|\nabla \varphi\|_{L^2(\Omega)}&\le c\big[(1+\|\nabla a_h^*\|_{L^2(\Omega)})\|u^\dag \|_{L^\infty(\Omega)}
  + \| \nabla u^\dag \|_{L^2(\Omega)}\big]\\
     &\le c(1+\|\nabla a_h^*\|_{L^2(\Omega)}) \leq c\gamma^{-\frac12}\eta.
\end{align*}
Hence $ \varphi\in H_0^1(\Omega)$.
Now the approximation property of the operator $P_h$ implies
\begin{equation*}
  \| \varphi-P_h \varphi\|_{L^2(\Omega)} \le ch\| \nabla  \varphi\|_{L^2(\Omega)} \le ch(1+\|\nabla a_h^*\|_{L^2(\Omega)})\leq ch\gamma^{-\frac12}\eta.
\end{equation*}
Thus by Lemma \ref{lem:FEM_priori}, the term ${\rm I}_1 $ in \eqref{eqn:sp-00} can be bounded by
\begin{align}\label{eqn:I1-elliptic}
  |{\rm I}_1 | & \le  ch (1+\| \nabla a_h ^*\|_{L^2(\Omega)} )^2 \le  c    h(1+\gamma^{-1}\eta^2) \le  c    h \gamma^{-1}\eta^2.
\end{align}
For the term ${\rm I}_2$, by the triangle inequality, inverse inequality on the FEM space $X_h$, the $L^2(\Omega)$
stability of the operator $P_h$ and Lemma \ref{lem:FEM_priori}, we have
\begin{align*}
&\|  \nabla(u^\dag- u_h(a_h^*)) \|_{L^2(\Omega)} \\
 \leq &\|  \nabla(u^\dag - P_hu^\dag ) \|_{L^2(\Omega)} + ch^{-1}\|  P_h u^\dag  - u_h(a_h^*)  \|_{L^2(\Omega)}\\
   \leq &c(h + h^{-1}\| u^\dag - u_h(a_h^*)   \|_{L^2(\Omega)})\le  c(h + h^{-1}\eta).
\end{align*}
Meanwhile, there holds the \textit{a priori} estimate
$$\|  \nabla(u - u_h(a_h^*)) \|_{L^2(\Omega)} \le c.$$ Hence,
the Cauchy-Schwarz inequality and Lemma \ref{lem:FEM_priori} imply
\begin{align}
  {\rm I}_2  & \le    \|  \nabla(u^\dag - u_h(a_h^*)) \|_{L^2(\Omega)} \|  \nabla  \varphi\|_{L^2(\Omega)}\nonumber\\
   &\le c \min(h + h^{-1}\eta, 1)(1+\|\nabla a_h^* \|_{L^2(\Omega)})\nonumber\\
   &\le c \min(h^{-1}\eta, 1) \gamma^{-\frac12}\eta.\label{eqn:I2-elliptic}
\end{align}
The  estimates \eqref{eqn:I1-elliptic} and \eqref{eqn:I2-elliptic} together imply
\begin{align*}
 ((a^\dag-a_h^*)\nabla u^\dag,\nabla \varphi)
  \le  c(h \gamma^{-\frac12}\eta + \min(h + h^{-1} \eta,1)) \gamma^{-\frac12}\eta.
\end{align*}
By repeating the argument of Theorem \ref{thm:Bonito_stab}
for the term $((a^\dag-a_h^*)\nabla u^\dag,\nabla \varphi)$, we get the identity
\begin{equation*}
  ((a^\dag-a_h^*)\nabla u^\dag,\nabla \varphi)  = \frac12\int_\Omega \Big(\frac{a^\dag-a_h^*}{a^\dag}\Big)^2 \big(  a^\dag | \nabla u^\dag  |^2 + fu^\dag \big) \,\d x.
\end{equation*}
This identity and the preceding estimates on ${\rm I}_1$ and ${\rm I}_2$ lead directly to the first assertion.
The second assertion follows from the argument in Theorem \ref{thm:Bonito_stab}.
\end{proof}

\begin{remark}
Theorem \ref{thm:FEM_error} provides useful guidance for choosing the algorithmic parameters. With the choice {$\gamma\sim \delta^2$}
and {$h\sim\delta^\frac12$}, there holds
\begin{equation*}
    \|a^\dag-a_h^*\|_{L^2(\Omega)}  \le c\delta^{\frac1{4(1+\beta)}}.
\end{equation*}
The choice {$\gamma\sim\delta^2$} differs from the usual condition
for Tikhonov regularization, i.e., ${ \lim_{\delta\to0^+}\frac{\delta^2}{\gamma(\delta)}=0,}$
but it agrees  with conditional stability in Theorem \ref{thm:cond-stab in abstract form}. It is instructive to compare the rate with the conditional stability estimate in Theorem \ref{thm:Bonito_stab}. By the a prior regularity estimate $u(a),u(a^\dagger)\in H^2(\Omega)$ for $a,a^\dagger\in \mathcal{A}\cap W^{1,p}(\Omega)$ with $p>d$ and
Gagliardo-Nirenberg interpolation inequality \eqref{eqn:Gagliardo-Nirenberg},
we have
\begin{equation*}
\|a-a^\dagger\|_{L^2(\Omega)}\leq c\|u(a)-u(a^\dagger)\|_{L^2(\Omega)}^\frac{1}{4(1+\beta)}.
\end{equation*}

\end{remark}

\begin{remark}\label{rmk:FEM_otherformulation}
The discussions so far focus on the output least-squares formulation. There are other approaches to reconstruct the diffusion coefficient $a$ and establish relevant error bounds. One approach relies on reformulating the inverse problem as a transport equation for the diffusion coefficient $a$. Richter \cite{Richter:1981,Richter:1981numer} proved the uniqueness of problem \eqref{eqn:gov_Falk} under the assumption that $a$ is given on the inflow boundary, the portion of the boundary with $\partial_\nu a<0$, and
\begin{equation}\label{eqn:FEM_Richter_cond}
    \inf_{\Omega} \max\{|\nabla u^\dagger |, \Delta u^\dagger \}>0.
\end{equation}
In addition, Richter \cite{Richter:1981numer} proposed a modified upwind difference scheme for the transport problem and proved an $O(h)$ convergence rate under the assumption $u^\dagger\in C^3(\overline{\Omega})$ and $a^\dagger\in C^2(\overline{\Omega}) $. Kohn and Lowe \cite{Kohn:1988} proposed the following discrete scheme to solve  problem  \eqref{eqn:gov_Falk}
\begin{align}\label{eqn:FEM_Kohn}
    \min_{\sigma_h\in \mathcal{K}_h, a_h\in \mathcal{A}_h} \| \sigma_h- &a_h\nabla z^\delta \|_{L^2(\Omega)}^2+\| \nabla\cdot \sigma_h+f \|_{L^2(\Omega)}^2\nonumber\\
      &+\|  \sigma_h\cdot n-g \|_{L^2(\partial\Omega)}^2
       +\gamma\| \nabla a_h\|_{L^2(\Omega)}^2,
\end{align}
where $\mathcal{A}_h=\mathcal{A}\cap V_h$ and $\mathcal{K}_h=V_h^{\otimes d}$ is the $d$-fold product space of $V_h$. The objective function measures the equation error but with $u^\dagger$ replaced by noise measurement $z^\delta$.
Another widely adopted approach involves minimizing the equation error. Kohn and Lowe proved that under the assumptions $u^\dagger\in H^3(\Omega)$, $\Delta u^\dagger\in C(\overline{\Omega})$, $a^\dagger\in H^2(\Omega)$ and $\|u^\dagger-z^\delta\|_{H^1(\Omega)}\le \delta$, there holds \cite[Theorem 4]{Kohn:1988}
\begin{equation*}
    \|a_h^*-a^\dagger\|_{L^2(\Omega)}\le c(h+\delta+\gamma^{\frac{1}{2}})^{\frac{1}{2}},
\end{equation*}
under the condition
\begin{equation}\label{eqn:FEM_cond_Kohn}
    \begin{aligned}
            &\forall \psi \in H^1(\Omega)\text{, equation }\nabla u^\dagger\cdot\nabla v_\psi=\psi \\
            &\text{has a solution $v_\psi$ with }\|v_\psi\|_{H^1(\Omega)}\le c\|\psi\|_{H^1(\Omega)}.
    \end{aligned}
\end{equation}
Condition \eqref{eqn:FEM_cond_Kohn} is weaker than condition \eqref{eqn:Falk_cond} in Falk's result.
\end{remark}

The numerical analysis of finite element approximations have a rather long history, but deriving error bounds on the discrete approximations is relatively recent, especially for nonlinear inverse problems, due to the nonlinearity of the forward maps, and one cannot apply the stability results for the continuous inverse problems directly. The error analysis has been conducted for several other nonlinear inverse problems, including diffusion coefficient recovery in the standard parabolic and time-fractional diffusion problems \cite{Jin:2021Error, Jin:2021Numerical,Jin:2023recovery,Jin:2024numerical} from either space-time measurement or terminal data, inverse potential problem \cite{CenShinZhou:2024,Zhang:2022potential,Jin:2023potential,JinShinZhou:2024,jin:2024stochastic}, simultaneous recovery of the potential and diffusion coefficient from two internal measurements (i.e., quantitative photo-acoustic tomography in the diffusive regime) \cite{Cen:2024multi,ACZ:2025}. The backward problem for time-fractional diffusion was thoroughly analyzed in
\cite{WuYangZhou:2025,ZhangZhou:2020IP, ZhangZhou:2022SISC, ZhangZhou:2023IP,jin2023:Inverse}. There have also been significant recent advances on the numerical analysis of unique continuation / data assimilation problems for elliptic, parabolic and wave equations, especially using stabilized FEMs \cite{Burman:2013,Burman:2018control,Burman:2018heat,Burman:2020wave,Burman:2020convectionI,Burman:2022convectionII}.

\section{Hybrid schemes}\label{sec:Hybrid}

In this section, we describe a class of schemes known as hybrid schemes, which aim to combine the strengths of traditional numerical methods with modern data-driven techniques. Indeed, many traditional computational techniques enjoy rigorous mathematical guarantees but face challenges for some problem settings, whereas machine learning techniques based on DNNs exhibit strong empirical performance but often lack rigorous mathematical underpinnings. Hybrid schemes aim to combine the strengths of these two different classes of methods.

\subsection{Preliminary}\label{subsec:Hybrid_Pre}

First we describe notation and properties of the concerned class of deep neural networks (DNNs). Let $\{d_\ell\}_{\ell=0}^L \subset\mathbb{N}$ be fixed with $d_0=d$. The parameterization $\Theta=
\{(A^{(\ell)},b^{(\ell)})_{\ell=1}^L\}$ consists of weight matrices and bias vectors, with
$A^{(\ell)}=[W_{ij}^{(\ell)}]\in \mathbb{R}^{d_\ell\times d_{\ell-1}}$ and $b^{(\ell)}=
[b^{(\ell)}_i]\in\mathbb{R}^{d_{\ell}}$ the weight matrix and bias vector at the $\ell$-th layer, respectively.
Then a DNN $v_\theta:= v^{(L)}:\Omega\subset\mathbb{R}^d\to\mathbb{R}^{d_L}$
realized by a parameter vector $\theta\in\Theta$ is defined recursively by
\begin{equation}\label{eqn:NN_realization}
\left\{\begin{aligned}
v^{(0)}&=x,\quad x\in\Omega\subset\mathbb{R}^d,\\		
v^{(\ell)}&=\rho(A^{(\ell)}v^{(\ell-1)}+b^{(\ell)}),\quad \ell=1,2,\cdots,L-1,\\
		v^{(L)}&=A^{(L)}v^{(L-1)}+b^{(L)},
	\end{aligned}\right.
\end{equation}
where the nonlinear activation function $\rho:\mathbb{R}\to\mathbb{R}$ is applied componentwise to a vector. The DNN $v_\theta$ has a depth $L$ and width $W:=
\max_{\ell=0,\dots,L}(d_{\ell})$. We denote the set of DNN functions of depth $L$, with at most $N_\theta$ nonzero weights, and maximum weight bound $R$ by
\begin{equation*}
    \mathcal{N}(L,N_\theta,R) =: \{v_\theta \mbox{ is a DNN of depth }L: \|\theta\|_{\ell^0}\leq N_\theta, \|\theta\|_{\ell^\infty}\leq R\},
\end{equation*}
where $\|\cdot\|_{\ell^0}$ and $\|\cdot\|_{\ell^\infty}$ denote, respectively, the number of nonzero entries in and the maximum norm of a vector.

The properties of a neural network $v_\theta$ depend heavily on the choice of the activation function $\rho$. We discuss several common choices below.
\begin{itemize}
    \item The rectified linear unit, $\mathrm{ReLU}$, defined by
    \begin{equation*}
        \mathrm{ReLU}(x)=\max(x,0),
    \end{equation*}
    is a very popular choice. It is easy to compute, scale-invariant, and the reduces the vanishing gradient problem.  One common issue with the ReLU networks is that of "dying neurons", since the gradient of ReLU is zero for $x<0$. This issue can be overcome by using the so-called leaky ReLU function
    \begin{equation*}
        \mathrm{ReLU}_\alpha(x)=\max(x,\alpha x),
    \end{equation*}
    with $\alpha\in (0,1)$. Other smooth variants are the Gaussian error linear unit, sigmoid linear unit, exponential linear unit and softplus.
    \item The logistic function is defined by
    \begin{equation*}
        \rho(x)=\frac{1}{1+e^{-x}}.
    \end{equation*}
    It is a smooth approximation of the Heaviside function, and is sigmoidal (i.e., it is smooth, monotonic and has horizontal asymptotes as $x\to\pm\infty$).
    \item The hyperbolic tangent, or tanh, defined by
    \begin{equation*}
        \rho(x)=\frac{e^x-e^{-x}}{e^x+e^{-x}}
    \end{equation*}
    is also smooth sigmoidal function. One problem with sigmoidal functions is that the gradient vanishes away from zero, which might cause problems when using gradient type training algorithms.
\end{itemize}
In neural solvers for PDE direct and inverse problems, smooth activation functions are typically adopted and the most popular choices are tanh and logistics, due to the regularity requirements of the underlying PDE. In practice, the activation function $\rho$ can also be trainable.

The approximation theory of DNNs initially focuses on universal approximation properties \cite{Cybenko:1989,Hornik:1991}, e.g., every continuous function defined on a compact domain can be approximated arbitrarily well by a shallow neural network under certain mild conditions on $\rho$. Later, the works \cite{Barron:1994,Mhaskar:1996} systematically investigated  approximation capabilities of sigmoid-activated DNNs for (piecewise) smooth functions. The analysis is based on global polynomial approximation and the approximation capacity of  DNNs.  For the ReLU activation, Yarotsky \cite{Yarotsky:2017} established a rigorous approximation theory for DNNs approximating localized Taylor polynomials. The key lies in constructing  an exact partition of unity using the ReLU activation. The work \cite{Guhring:2021} extends the approximation results to a wide class of smooth activation functions using an approximate partition of unity. The following approximation result holds for the $\tanh$ activation \cite[Theorem 4.9]{Guhring:2021}.
\begin{lemma}\label{lem:NN_approx_tanh}
Let $\rho\equiv\tanh:x\mapsto\frac{e^x-e^{-x}}{e^x+e^{-x}}$. Let $s\in\mathbb{N}\cup\{0\}$ and $p\in[1,\infty]$ be fixed, and $v\in W^{k,p}
(\Omega)$ with $\mathbb{N}\ni k\geq s+1$. Then for any $\epsilon>0$, there exists one
\begin{equation*}
    v_\theta\in \mathcal{N}\left(c\log(d+k),c\epsilon^{-\frac{d}{k-s-\mu (s=2)}},c \epsilon^{-2-\frac{2(d/p+d+k+\mu(s=2))+d/p+d}{k-s-\mu(s=2)}}\right),
\end{equation*}
where $\mu>0$ is arbitrarily small, such that $\|v-v_\theta\|_{W^{s,p}(\Omega)} \leq \epsilon$.
\end{lemma}

On the domain $\Omega=(0,1)^d$, Guhring and Raslan \cite{Guhring:2021} prove Lemma \ref{lem:NN_approx_tanh} in three steps. They first divide $(0,1)^d$ into $(N+1)^d$ equal hypercubes with a grid size $1/N$ and construct an approximate partition of unity by DNNs \cite[Lemma 4.5]{Guhring:2021}. Next, they approximate a function $v\in W^{k,p}(\Omega)$ by a localized Taylor polynomial $v_{\rm poly}$:
$\|v-v_{\rm poly}\|_{W^{s,p}(\Omega)}\leq c_{\rm poly}N^{-(k-s-\mu(s=2))}$,
where the construction of $v_{\rm poly}$ relies on the approximate partition of unity and the constant $c_{\rm poly}=c_{\rm poly}(d, p, k,s)>0$. Finally, they show that there exists a DNN parameter $\theta$, satisfying the conditions in Lemma \ref{lem:NN_approx_tanh}, such that \cite[Lemma D.5]{Guhring:2021}:
     \begin{equation*}
        \|v_{\rm poly}-v_\theta\|_{W^{s,p}(\Omega)} \leq c_{\rm DNN}\|v\|_{W^{k,p}(\Omega)}\tilde{\epsilon},
     \end{equation*}
    where the constant $c_{\rm DNN}=c_{\rm DNN}(d,p,k,s)>0$ and  $\tilde{\epsilon}\in(0,\frac12)$.
    Now for small $\epsilon>0$, the desired estimate follows from the choice
    $N=(\frac{\epsilon}{2c_{\rm poly}})^{-\frac{1}{k-s-\mu(s=2)}}$ and $\tilde{\epsilon}=\frac{\epsilon}{2c_{\rm NN}\|v\|_{W^{k,p}(\Omega)}}$.

Below we take $k=1$ and $s=0$ in Lemma \ref{lem:NN_approx_tanh}. For any $\epsilon>0$ and $p\geq 1$, we denote the DNN parameter set by $\mathfrak{P}_{p,\epsilon}\subset \Theta$ for
\begin{equation*}
    \mathcal{N}\Big(c\log(d+1), c\epsilon^{-\frac{d}{1-\mu}}, {c \epsilon^{-\frac{4p+3d+3pd}{p(1-\mu)}} }\Big).
\end{equation*}
The next lemma gives the boundedness of derivatives of tanh-activated DNNs. This result is useful for bounding the quadrature error.
\begin{lemma}\label{lem:NN_bounds}
    Let $\theta \in \mathfrak{P}_{\infty,\epsilon}$, of depth $L$, width $W$ and maximum bound $R$, and $v_\theta$ be its DNN realization, with $RW>2$. Then the following estimates hold
    \begin{align*}
        \|  v_\theta\|_{L^\infty(\Omega)} \le  cR(W+1),&\quad  |v_\theta|_{W^{1,\infty}(\Omega)} \le  cR^L W^{L-1},\\
         |v_\theta|_{W^{2,\infty}(\Omega)} \le  cR^{2L} W^{2L-2},&\quad  |v_\theta|_{W^{3,\infty}(\Omega)} \le  cR^{3L} W^{3L-3},
    \end{align*}
    where $|\cdot|_{W^{k,\infty}(\Omega)}$ ($k=1,2,3$) denotes the $W^{k,\infty}(\Omega)$ seminorm.
\end{lemma}
\begin{proof}
For the activation $\rho\equiv \tanh$, there holds
\begin{align*}
        \rho'(x)&=1-\rho^2(x),\quad \rho''(x)=-2\rho(x)(1-\rho^2(x)),\\ \rho'''(x)&=(6\rho^2(x)-2)(1-\rho^2(x)).
    \end{align*}
Therefore, the following estimates hold:
\begin{equation}\label{eqn:NN_act_bound}
\begin{aligned}
&\|\rho \|_{L^\infty(\mathbb{R})}\le 1,\quad \|\rho' \|_{L^\infty(\mathbb{R})}\le 1,\\ &\|\rho'' \|_{L^\infty(\mathbb{R})}\le 1,\quad \|\rho''' \|_{L^\infty(\mathbb{R})}\le 2.
\end{aligned}
\end{equation}
    By the DNN realization \eqref{eqn:NN_realization}, we have for every layer $\ell=1,\dots,L-1$ and each $i=1,\cdots,d_\ell$,
    $v^{(\ell)}_i=\rho\big(\sum_{j=1}^{d_{\ell-1}}A^{(\ell)}_{ij}v^{(\ell-1)}_j+b^{(\ell)}_i\big)$. The estimate \eqref{eqn:NN_act_bound} direct yields   $\|v_i^{(L-1)}\|_{L^\infty(\Omega)}\le 1$. By the realization \eqref{eqn:NN_realization} and the structure of the DNN, we derive
    \begin{equation*}
        \|  v_\theta\|_{L^\infty(\Omega)}\le \|A^{(L)} v^{(L-1)}\|_{L^\infty(\Omega)}+ \|b^{(L)} \|_{L^\infty(\Omega)}\le RW+R.
    \end{equation*}
    Now we  bound the derivatives of $v_\theta$. For the $\ell$-th layer, fix $1\le m\le d$. Direct computation with the chain rule gives
    \begin{align*}
    	 \partial_{x_m}v^{(\ell)}_i=\rho'\Big(\sum_{j=1}^{d_{\ell-1}}A^{(\ell)}_{ij}v^{(\ell-1)}_j+b^{(\ell)}_i\Big)\Big(\sum_{j=1}^{d_{\ell-1}}A^{(\ell)}_{ij}\partial_{ x_m}v^{(\ell-1)}_j\Big).
    \end{align*}
    By \eqref{eqn:NN_act_bound}, there holds
    \begin{align*}
    	\|\partial_{x_m}v^{(\ell)}_i\|_{L^\infty(\Omega)}\le RW  \max_{1\le j\le d_\ell}\|\partial_{ x_m}v^{(\ell-1)}_j\|_{L^\infty(\Omega)}.
    \end{align*}
    Note also the trivial estimate
    \begin{align*}
        \|\partial_{ x_m}v^{(1)}_i\|_{L^\infty(\Omega)}\leq \bigg\|\rho'\Big(\sum_{j=1}^{d}A^{(1)}_{ij}x_j+b^{(1)}_i\Big)A^{(1)}_{im}\bigg\|_{L^\infty(\Omega)}\leq R.
    \end{align*}
    Taking maximum over $i=1,\dots,d_\ell$ and then applying the inequality recursively lead to
    \begin{align*}
        \max_{1\le i\le d_\ell }\|\partial_{x_m}v_i^{(\ell)}\|_{L^\infty(\Omega)} \leq  (RW)^{\ell-1}\max_{1\le i\le d_1 }\|\partial_{x_m}v_i^{(1)}\|_{L^\infty(\Omega)}\le R^{\ell}W^{\ell-1}.
    \end{align*}
    This completes the proof of the estimate of the first-order derivative of DNN functions. The estimates on high order derivatives follow similarly.
\end{proof}

\subsection{Hybrid DNN-FEM scheme}\label{subsec:Hybrid_err}
Now we describe a hybrid DNN-FEM scheme \cite{Cen:2024} for diffusion coefficient identification \eqref{eqn:inv_cond_gov} and give an error analysis. The scheme is based on the least-squares formulation with an $H^1(\Omega)$-seminorm penalty. We employ an alternative discretization strategy when compared with \eqref{eqn:FEM_loss_Jin_dis}-\eqref{eqn:FEM_vari_Jin_dis}. In the hybrid scheme, we parameterize the diffusion coefficient $a$ by one DNN and approximate the state $u$ using the Galerkin FEM. Let $X_h$ be the piecewise linear FEM space defined in Section \ref{sec:FEM} and $\mathfrak{P}_{p,\epsilon}$ be the DNN function class given in Section \ref{subsec:Hybrid_Pre}. Then the hybrid DNN-FEM scheme reads
\begin{equation}\label{eqn:hybrid_loss}
	\min_{\theta\in\mathfrak{P}_{p,\epsilon}} J_{\gamma,h}(a_\theta)=\frac{1}{2}\|u_h(a_\theta)-z^{\delta}\|^2_{L^2(\Omega)}+\frac{\gamma}{2}\|\nabla a_\theta\|_{L^2(\Omega)}^2,
\end{equation}
where the discrete state $u_h\equiv u_h(a_\theta)\in X_h$ satisfies the following discrete variational problem
\begin{equation}\label{eqn:hybrid_vari}	
	(a_\theta\nabla u_h,\nabla\varphi_h)=(f,\varphi_h), \quad\forall \varphi_h\in X_h.
\end{equation}

Compared with the piecewise polynomial FEM space $X_h$, DNN functions are globally defined with a nonlinear composition structure. This leads to the following two challenges of the scheme \eqref{eqn:hybrid_loss}-\eqref{eqn:hybrid_vari}. First,  unlike compactly supported FEM basis functions, it is challenging to impose the box constraint of the admissible set $\mathcal{A}$ directly. Second, the scheme \eqref{eqn:hybrid_loss}-\eqref{eqn:hybrid_vari} requires evaluating various integrals involving DNNs, which cannot be computed exactly as for FEM functions.

For the first issue, we employ a cutoff operator $P_{\!\!\mathcal{A}\!}:H^1(\Omega)\rightarrow \mathcal{A}$ defined by
\begin{equation}\label{eqn:hybrid_proj}
    P_{\!\!\mathcal{A}\!}(v) = \min(\max(c_0,v),c_1).
\end{equation}
The operator $P_{\!\!\mathcal{A}\!}$ is stable in the following sense \cite[Corollary 2.1.8]{Ziemer:1989}
\begin{equation}\label{eqn:hybrid_proj_stb}
\| \nabla P_{\!\!\mathcal{A}\!}(v) \|_{L^p(\Omega)} \le \|\nabla v \|_{L^p(\Omega)},\quad \forall v \in W^{1,p}(\Omega),p\in[1,\infty],
\end{equation}
and moreover, for all $v\in \mathcal{A}$, there holds
\begin{equation}\label{eqn:hybrid_proj_approx}
\|  P_{\!\!\mathcal{A}\!}(w) - v \|_{L^p(\Omega)} \le \| w - v \|_{L^p(\Omega)},\quad \forall w \in L^{p}(\Omega),~p\in[1,\infty].
\end{equation}

For the second issue, we employ a quadrature scheme. There are many possible quadrature rules \cite{Thomee:2006,Cialet:2002}. We describe one simple scheme. For each element $K\in \mathcal{T}_h$,  we uniformly divide it into $2^{dn}$ sub-simplexes,  denoted by $\{ K_i\}_{i=1}^{2^{dn}}$, with a diameter $h_K/2^n$. The division for $d=1,2$ is trivial, and for $d=3$, it is also feasible  \cite{Ong:1994}. Then consider the following quadrature rule over the element $K$ (with $P_{\!\!j}^i$ denoting the $j$th node of the $i$th sub-simplex $K_i$):
\begin{equation*}
    Q_K(v) = \sum_{i=1}^{2^{dn}} \frac{|K_i|}{d+1} \sum_{j=1}^{d+1} v(P_{\!\!j}^i),\quad \forall v\in C(\overline K).
\end{equation*}
The continuous embedding $H^2(\Omega) \hookrightarrow C(\overline{\Omega})$ (for $d=1,2,3$) and Bramble--Hilbert lemma \cite[Theorem 4.1.3]{Cialet:2002} imply
\begin{equation*}
    \Big|\int_{K} v \ \mathrm{d}x - Q_K(v)\Big| \leq
    c |K|^{\frac12} 2^{-2n} h_K^2  |v|_{H^2(K)}, \quad \forall v\in H^2(K).
\end{equation*}
Then we can define a global quadrature rule:
\begin{equation*}
    Q_h(v) =  \sum_{K\in \mathcal{T}_h} Q_K(v),\quad \forall v\in C(\overline \Omega),
\end{equation*}
which satisfies the following error estimate
\begin{equation}\label{eqn:hybrid_quad_err}
    \Big|\int_{\Omega} v \ \mathrm{d}x - Q_h(v)\Big| \leq c 2^{-2n} h^2 |v|_{H^{2}(\Omega)},\quad \forall v\in H^2(\Omega).
\end{equation}
Similarly, we define a discrete $L^2(\Omega)$ inner product $(\cdot,\cdot)_h$ by
\begin{equation*}
    (w,v)_h := Q_h(wv) = \sum_{K\in\mathcal{T}_h}Q_K(wv),\quad \forall w,v\in C(\overline \Omega).
\end{equation*}
Then the practical hybrid DNN-FEM approximation scheme reads
\begin{equation}\label{eqn:hybrid_loss_quad}
	\min_{\theta\in \mathfrak{P}_{\infty,\epsilon}} \widehat{J}_{\gamma,h}(a_\theta)=\frac{1}{2}\|\widehat{u}_h( P_{\!\!\mathcal{A}}(a_\theta))-z^{\delta}\|^2_{L^2(\Omega)}+\frac{\gamma}{2} Q_h(|\nabla a_\theta|^2),
\end{equation}
where the discrete state $\widehat{u}_h\equiv \widehat{u}_h(P_{\!\!\mathcal{A}}(a_\theta))\in X_h$ satisfies the following discrete variational problem
\begin{equation}\label{eqn:hybrid_vari_quad}	
    ( P_{\!\!\mathcal{A}}(a_\theta)\nabla \widehat{u}_h,\nabla\varphi_h)_h=(f,\varphi_h), \quad\forall \varphi_h\in X_h.
\end{equation}
Note that the projection operator $P_{\!\!\mathcal{A}}$ is applied only to the data discrepancy term, but not the penalty term due to the limited regularity of $P_{\!\!\mathcal{A}}(a_\theta)$: $P_{\!\!\mathcal{A}}(a_\theta)$ is generally only once differentiable, which precludes the use of gradient type algorithms to train the loss.
We focus on approximating the integrals involving DNNs only.
The variational problem \eqref{eqn:hybrid_loss_quad} additionally involves the quadrature approximation, which necessitates quantifying the associated error.
The presence of the operator $P_{\!\!\mathcal{A}}$ in the weak formulation ensures the $X_h$-ellipticity of the broken $L^2(\Omega)$ semi-inner product, and hence the well-posedness of problem \eqref{eqn:hybrid_loss_quad}-\eqref{eqn:hybrid_vari_quad}. We aim to derive (weighted) $L^2(\Omega)$ error estimates of the approximation $P_{\!\!\mathcal{A}}(a^*_\theta)$ based on the stability result in Theorem \ref{thm:Bonito_stab}.
The analysis of the quadrature error requires the following condition on the problem data, i.e., $a^\dag$ and $f$.
\begin{assumption}\label{assum:hybrid_reg_quad}
    $a^\dag\in W^{2,\infty}(\Omega)\cap \mathcal{A}$ and $f\in L^{\infty}(\Omega)$.
\end{assumption}
Next we state an analogue of Lemma \ref{lem:FEM_priori}, which provides an \textsl{a priori} bound on $\|u^\dag-\widehat{u}_h(P_{\!\!\mathcal{A}}(\widehat{a}^*_\theta))\|_{L^2(\Omega)}$ and $\nabla P_{\!\!\mathcal{A}}(\widehat{a}_\theta^*)$. The proof relies on the minimizing property of $\widehat{\theta}^\ast$ and the error estimate of the discrete state $\widehat{u}_h(P_{\!\!\mathcal{A}}(\widehat{a}^*_\theta))$.
\begin{lemma}\label{lem:hybrid_priori_quad}
    Let Assumption \ref{assum:hybrid_reg_quad} hold. Fix $\epsilon>0$, and let $\widehat{\theta}^*\in\mathfrak{P}_{\infty,\epsilon} $ be a minimizer to problem \eqref{eqn:hybrid_loss_quad}-\eqref{eqn:hybrid_vari_quad} and $\widehat{a}^*_\theta$ its DNN realization. Then the following estimate holds
    \begin{equation*}
        \|u^\dag -\widehat{u}_h( P_{\!\!\mathcal{A}} (\widehat{a}^*_\theta))\|^2_{L^2(\Omega)}+\gamma Q_h(|\nabla  P_{\!\!\mathcal{A}}(\widehat{a}_\theta^*)|^2)\leq c(h^4+\epsilon^2+\delta^2+\gamma).
    \end{equation*}
\end{lemma}
The next lemma gives the quadrature error for the bilinear form. The estimate indicates that choosing a large subdivision level $n$ can ensure a small quadrature error.
\begin{lemma}\label{lem:hybrid_quad_prod_err}
The following error estimate holds for any $v_h,w_h\in X_h$, $n\in\mathbb{N}$ and $k=1,2$:
    \begin{align*}
        &|(a\nabla v_h,\nabla w_h)-(a \nabla v_h,\nabla w_h)_h|\\
        \leq &c (2^{-n}h)^k \|a\|_{W^{k,\infty}(\Omega)}\|\nabla v_h\|_{L^2(\Omega)}\|\nabla w_h\|_{L^2(\Omega)}.
    \end{align*}
\end{lemma}
\begin{proof}
    Let $\Pi_{K_j}: C(K_j) \rightarrow P_1(K_j)$ be the Lagrange nodal interpolation operator on the sub-simplex $K_j$. Since the quadrature rule on $K_j$ is exact for $P_1(K_j)$, we have $$(a\nabla v_h,\nabla w_h)-(a \nabla v_h,\nabla w_h)_h =  \sum_{K\in \mathcal{T}_h} \sum_{j=1}^{2^{dn}}  \int_{K_j} (a - \Pi_{K_j}a) \nabla v_h\cdot\nabla w_h\,\d x.$$
    Then the local estimate for the Lagrange interpolation operator $\Pi_{K_j}$ leads to
        \begin{align*}
            &|(a\nabla v_h,\nabla w_h)-(a \nabla v_h,\nabla w_h)_h|\\
            \le& \sum_{K\in \mathcal{T}_h} \sum_{j=1}^{2^{dn}}
            \int_{K_j} \Big| (a - \Pi_{K_j}a) \nabla v_h\cdot\nabla w_h \Big| \,\d x\\
            \le& c \sum_{K\in \mathcal{T}_h} \sum_{j=1}^{2^{dn}} (2^{-n} h)^k \| a \|_{W^{k,\infty}(K_j)}
            \| \nabla v_h\|_{L^2(K_j)} \|\nabla w_h \|_{L^2(K_j)}\\
            \le& c  (2^{-n} h)^k \| a \|_{W^{k,\infty}(\Omega)} \| \nabla v_h\|_{L^2(\Omega)} \|\nabla w_h \|_{L^2(\Omega)}.
        \end{align*}
    This proves the desired estimate.
\end{proof}

Next we show an \textit{a priori} bound on the quadrature error of the penalty term, which follows directly from the estimate \eqref{eqn:hybrid_quad_err} and Lemma \ref{lem:NN_bounds}.
\begin{lemma}\label{lem:hybrid_quad_H1err}
Fix $\theta \in \mathfrak{P}_{\infty,\epsilon}$, of depth $L$, width $W$ and bound $R$, with $RW>2$, and let $v_\theta$ be its DNN realization. Then the following error estimate holds
    \begin{align*}
        \|\nabla v_\theta\|^2_{L^2(\Omega)}-Q_h(\|\nabla v_\theta\|_{\ell^2}^2)\leq c2^{-2n}h^2R^{4L}{W}^{4L-4}.
    \end{align*}
\end{lemma}

Now we can state an error estimate on the approximation $\widehat{a}_\theta^*$.
\begin{theorem}\label{thm:hybrid_error_quad}
Let Assumption \ref{assum:hybrid_reg_quad} hold. Fix $\epsilon>0$, and let $\widehat{\theta}^*\in \mathfrak{P}_{\infty,\epsilon}$ be a minimizer to problem \eqref{eqn:hybrid_loss_quad}-\eqref{eqn:hybrid_vari_quad} and $\widehat{a}_\theta^*$ its DNN realization. Then with
\begin{equation*}
 \eta^2:=h^4+\epsilon^2+\delta^2+\gamma \quad\mbox{and}\quad \zeta:=1+\gamma^{-1}\eta^2+2^{-2n} h^2 R^{4L}{W}^{4L-4},
\end{equation*}
there holds
\begin{align*}		
    &\int_{\Omega}\Big(\frac{a^\dag- P_{\!\!\mathcal{A}}(\widehat{a}_\theta^*)}{a^\dag}\Big)^2\big(a^\dag |\nabla u^\dag|^2+fu^\dag\big)\ \mathrm{d}x\\
    \leq& c\big(h\zeta^\frac12+ (\min(h^{-1}\eta+h,1) +  2^{-n}h R^{L}{W}^{L} \big)\zeta^\frac12.
\end{align*}
Moreover, if condition \eqref{eqn:Bonito_cond} holds, then
\begin{equation*}
    \|a^\dagger-P_{\!\!\mathcal{A}}(\widehat{a}_\theta^*)\|_{L^2(\Omega)}\leq c\big(h\zeta^\frac12+ (\min(h^{-1}\eta+h,1) +  2^{-n}h R^{L}{W}^{L} )\zeta^\frac12\big)^{\frac{1}{2(\beta+1)}}.
\end{equation*}
\end{theorem}
\begin{proof}
    By the weak formulations of $u^\dag$ and $\widehat{u}_h(P_{\!\!\mathcal{A}}(\widehat{a}_\theta^*))$, cf. \eqref{eqn:inv_cond_gov} and \eqref{eqn:hybrid_vari_quad}, respectively, we have for any  $\varphi\in H_0^1(\Omega)$,
    \begin{align*}
       & \big((a^\dag-  P_{\!\!\mathcal{A}}(\widehat{a}_\theta^*))\nabla u^\dagger,\nabla\varphi\big)\\
        = &\big ((a^\dag- P_{\!\!\mathcal{A}}( \widehat{a}_\theta^*))\nabla u^\dagger,\nabla(\varphi-P_h\varphi)\big)
        + \big( (a^\dag- P_{\!\!\mathcal{A}}(\widehat{a}_\theta^*))\nabla u^\dagger,\nabla P_h\varphi\big) \\
        =& -\big(\nabla\cdot((a^\dag- P_{\!\!\mathcal{A}}(\widehat{a}_\theta^*))\nabla u^\dagger),\varphi-P_h\varphi\big)\\
        &+\big( P_{\!\!\mathcal{A}}(\widehat{a}_\theta^*)\nabla(\widehat{u}_h(P_{\!\!\mathcal{A}}(\widehat{a}_\theta^*))-u^\dagger),\nabla P_h\varphi\big)  \\
        & + [\big( P_{\!\!\mathcal{A}}(\widehat{a}_\theta^*)\nabla \widehat{u}_h(P_{\!\!\mathcal{A}}(\widehat{a}_\theta^*)),\nabla P_h\varphi\big)_h - \big( P_{\!\!\mathcal{A}}(\widehat{a}_\theta^*)\nabla \widehat{u}_h(P_{\!\!\mathcal{A}}(\widehat{a}_\theta^*)),\nabla P_h\varphi\big)] \\ =:& {\rm I + II + III}.
    \end{align*}
    Next we take $\varphi\equiv\frac{a^\dag- P_{\!\!\mathcal{A}}(\widehat{a}_\theta^*)}{a^\dag}u^\dag$ in the identity and bound the three terms separately. By the stability estimate \eqref{eqn:hybrid_proj_stb} of the operator $P_{\!\!\mathcal{A}}$ and Lemmas \ref{lem:hybrid_priori_quad} and \ref{lem:hybrid_quad_H1err}, we obtain
    \begin{align*}
        &\|\nabla  P_{\!\!\mathcal{A}}(\widehat{a}_\theta^*)\|^2_{L^2(\Omega)}
        \le \| \nabla  \widehat{a}_\theta^*\|^2_{L^2(\Omega
        )} \\
        =&Q_h(|\nabla\widehat{a}_\theta^*|^2) +  [\| \nabla{\widehat{a}}_\theta^*\|^2_{L^2(\Omega)}-Q_h(|\nabla\widehat{a}_\theta^*|^2)]\\
        \le& c (\gamma^{-1}\eta^2 \, + 2^{-2n}h^2R^{4L}W^{4L-4}).
    \end{align*}
    Thus we can bound $\|\nabla \varphi\|_{L^2(\Omega)}$ by
    \begin{equation*}
        \|\nabla \varphi\|_{L^2(\Omega)}\leq c(1+\|\nabla  P_{\!\!\mathcal{A}}(\widehat{a}_\theta^*)\|_{L^2(\Omega)})
        \leq c\zeta^\frac12.
    \end{equation*}
    Repeating the argument of Theorem \ref{thm:FEM_error} and applying Lemma \ref{lem:hybrid_priori_quad} yield
    \begin{align*}
        |{\rm I}| &\leq ch(1+\|\nabla P_{\!\!\mathcal{A}} (\widehat{a}_\theta^*)\|^2_{L^2(\Omega)}) \leq ch\zeta,\\
        |{\rm II}| &\le  c\big(1+\|\nabla P_{\!\!\mathcal{A}}(\widehat{a}_\theta^*)\|_{L^2(\Omega)}\big)\|\nabla(\widehat{u}_h(P_{\!\!\mathcal{A}}(\widehat{a}_\theta^*))-u^\dagger)\|_{L^{2}(\Omega)} \\
        &\le c \min(h^{-1}\eta+h,1) \zeta^\frac12.
        \end{align*}
    From Lemma \ref{lem:hybrid_quad_prod_err} (with $p=1$), Lemma \ref{lem:NN_bounds} and the stability of the operator $P_{\!\!\mathcal{A}}$, we deduce
    \begin{align*}
        |{\rm III}|&\leq c2^{-n}h\|P_{\!\!\mathcal{A}}(\widehat{a}_\theta^*)\|_{W^{1,\infty}(\Omega)} \|\nabla\widehat{u}_h(P_{\!\!\mathcal{A}}(\widehat{a}_\theta^*))\|_{L^2(\Omega)}
        \|\nabla P_h \varphi\|_{L^2(\Omega)}\\
        &\leq  c 2^{-n}h  \zeta^{\frac12}
        (\|  P_{\!\!\mathcal{A}}( \widehat{a}_\theta^*) \|_{L^\infty(\Omega)} + \| \nabla P_{\!\!\mathcal{A}}(\widehat{a}_\theta^*) \|_{L^\infty(\Omega)})\\
        &\leq  c 2^{-n}h \zeta^{\frac12}
        (1 + \| \nabla \widehat{a}_\theta^* \|_{L^\infty(\Omega)})
        \leq c 2^{-n}h R^LW^{L} \zeta ^\frac12.
    \end{align*}
    Upon repeating the argument in Theorem \ref{thm:Bonito_stab}, we obtain
    \begin{equation*}
    	\big((a^\dag-P_{\!\!\mathcal{A}}(\widehat{a}_\theta^*))\nabla u^\dag,\nabla\varphi\big) = \frac12\int_{\Omega}\Big(\frac{a^\dag-P_{\!\!\mathcal{A}}(\widehat{a}_\theta^*)}{a^\dag}\Big)^2\big(a^\dag|\nabla u^\dagger|^2+fu^\dagger\big)\ \mathrm{d}x.
    \end{equation*}
    This yields the desired weighted $L^2(\Omega)$ estimate. The proof of the second assertion is identical with that of Theorem \ref{thm:Bonito_stab}.
\end{proof}
\begin{remark}\label{rmk:hybrid_err_quad}
Theorem \ref{thm:hybrid_error_quad} provides useful guideline for choosing the algorithmic parameters:
$\gamma=\mathcal{O}(\delta^2)$, $ h=\mathcal{O}(\delta^\frac12)$, and $\epsilon=\mathcal{O}(\delta)$.
    Then under condition \eqref{eqn:Bonito_cond}, we obtain $$\|a^\dag-P_{\!\!\mathcal{A}}(a_\theta^*)\|_{L^2(\Omega)}\leq c\delta^\frac{1}{4(1+\beta)},$$ provided that the quadrature error is sufficiently small.  This result is comparable with that for the purely FEM approximation in Theorem \ref{thm:FEM_error}.
    The quadrature error involves a factor $R^{4L}W^{4L-1}$, which can be large for DNNs, and hence it may require a large $n$ to compensate its influence on the reconstruction $P_{\!\!\mathcal{A}}(\widehat{a}_\theta^*)$. 
    However, in practice, the choice $n=0$ suffices the desired accuracy.
\end{remark}

Now we briefly review various hybrid approaches in the literature. To identify the diffusion coefficient $a$, Berg and Nystr{\"o}m \cite{Berg:2021} approximate the unknown coefficient $a$ and the state $u$ by DNNs and the FEM, respectively, and present extensive numerical experiments to show its robustness for treating noisy and incomplete data. However, the formulation does not treat the box constraint and impose the penalty on the unknown coefficient. Later, Mitusch et al \cite{Mitusch:2021} extend the hybrid formulation to both stationary and transient and nonlinear PDEs, e.g., a heat equation with nonlinear diffusion coefficient $a(x,t,u)$.
The diffusion coefficient $a(x,t,u)$ is approximated by one DNN, and the state $u(x,t)$ is approximated using the Galerkin FEM in space and backward difference in time. The loss is designed to minimize the error between the solution of the DNN-FEM model and the observational data.


Badia and Mart{\'i}n \cite{Badia:2024,Badia:2024compatible} presented a slightly different approach, by interpolating the DNN representation of the unknowns onto FEM spaces. For problem \eqref{eqn:inv_cond_gov}, the diffusion coefficient $a$ and the state $u$ are approximated by two separate DNNs $a_\theta$ and $u_\kappa$. Then the DNNs are interpolated onto FEM spaces and the loss is given as
\begin{equation*}
    J(a_\theta,u_\kappa)=\|u_\kappa-z^\delta\|_{L^2(\Omega)}^2+\gamma\|R_h(\Pi_h(a_\theta), \Pi_h(u_\kappa)) \|_{H^{-1}(\Omega)},
\end{equation*}
where $\gamma>0$ is the penalty parameter, and $R_h$ denotes the PDE residual in the weak form:
\begin{equation*}
    R_h(\Pi_h(a_\theta), \Pi_h(u_\kappa)) \varphi_h= ( f,\varphi_h)- ( \Pi_h(a_\theta)\nabla \Pi_h(u_\kappa)), \nabla \varphi_h),\quad \forall \varphi_h\in X_h.
\end{equation*}
The formulation is similar to physics informed neural networks (PINNs), which will be discussed in Section \ref{subsec:NN_PINN}. By interpolating DNNs into the FEM space $X_h$, the approach overcomes the challenges encountered in the PINN formulation, e.g., imposing the Dirichlet boundary condition, selecting collocation points, and controlling quadrature errors. The numerical results show that it performs well when either the mesh resolution or the polynomial order increases.

The hybrid DNN-FEM approach has been investigated extensively in solid mechanics. Huang et al \cite{Huang:2020} employ DNNs to represent unknown constitutive relation and to formulate a loss by the Galerkin FEM. Consider the following model problem
\begin{equation*}
    \left\{
        \begin{aligned}
            \sigma_{ij,j}&=f_i,&&\mbox{ in }\Omega,\\
            u_i&=\overline{u}_i,&&\mbox{ on }\Gamma_u,\\
            n_j\sigma_{ij}&=\overline{t}_i,&&\mbox{ on }\Gamma_t,
        \end{aligned}
    \right.
\end{equation*}
where $\Gamma_u$ and $\Gamma_t$ denote displacement boundary and traction boundary, respectively. The displacement $u_i$ and the stress tensor $\sigma_{ij}$ are related by
\begin{equation*}
    \sigma_{ij}= \sum_{m=1}^d\sum_{n=1}^d C_{ijmn}\epsilon_{mn}, \quad\mbox{with } \epsilon_{mn}=\frac{1}{2}\left(\frac{\partial u_n}{\partial x_m}+ \frac{\partial u_m}{\partial x_n} \right),
\end{equation*}
where $C_{ijmn}$ denotes the constitutive tensor components. Huang et al approximate the unknown tensor $C_{ijmn}$ by DNNs, and discretize the state equation via the FEM. See also the work \cite{ChenGu:2021} on learning linear elasticity relations from the measured displacement field, arising in elastography. It is worth noting that most of the existing schemes have not been theoretically analyzed.


\subsection{Hybrid DNN-LMM}
In this part, we describe the hybrid DNN-LNM for dynamics recovery.
Linear multistep methods (LMMs) are high-order discretization techniques for solving dynamical systems \cite{Hairer:1993SolvingI,Hairer:1993SolvingII}. Recently, LMMs are also employed for dynamics discovery, and integrated with DNNs. Consider the following dynamical system with the initial condition $y_0\in \mathbb{R}^d$:
\begin{equation}\label{eqn:hybrid_LMM_gov}
    \left\{
        \begin{aligned}
            \frac{\d}{\d t} y(t)&=f(y(t)),\quad 0<t<T,\\
            y(0)&=y_0 ,
        \end{aligned}
    \right.
\end{equation}
where $y\in C[0,T]^d$ is a vector-valued state function and $f:\mathbb{R}^d\to \mathbb{R}^d$ characterizes the dynamics. The forward model can be solved accurately using LMMs. Specifically, fix $N\in\mathbb{N}$, and let $h := T/N$ and $t_n = nh$ for $n=0,\dots,N$. Given the starting states $y_0,y_1,\dots,y_{M-1}$, we can compute $y_n$ for $n=M,M+1,\dots,N$ using the following $M$-step LMMs:
\begin{equation}\label{eqn:hybrid_LMM_forward_scheme}
    \sum_{m=0}^{M} \alpha_m y_{n-m}= h \sum_{m=0}^{M}\beta_m f(y_{n-m}),\quad n= M, M+1, \dots, N,
\end{equation}
where the coefficients $\alpha_m, \beta_m \in  \mathbb{R}$ for $m = 0, 1,\dots ,M$ are given and $\alpha_0$ is nonzero. The inverse problem is to recover the unknown dynamics $f$ using the knowledge of the state $y$ at the equidistant time steps $\{t_n\}_{n=1}^{N}$. One hybrid approach combining LMMs and DNNs for the problem was presented in \cite{Raissi:2018,Tipireddy:2019}, in which one DNN $f_\theta$ is employed to represent the unknown dynamics $f$ and the loss consists of the residual of the dynamics system discretized by the LMM scheme \eqref{eqn:hybrid_LMM_forward_scheme}, i.e.,
\begin{equation*}
     J(f_\theta)=\frac{1}{N-M+1}\sum_{n=M}^N\Big|\sum_{m=0}^{M} \alpha_m y_{n-m} - h\beta_m f_{\theta}(y_{n-m})\Big|^2.
\end{equation*}
The numerical experiments therein indicate that the method can discover complicated or high-dimensional systems.

Keller and Du \cite{DuKeller:2021} developed a first systematic framework for analyzing dynamics discovery using LMMs. Suppose that both the state $\{y_n=y(t_n)\}_{n=0}^{N}$ and the first $M$ approximate dynamics $\hat{f}_i\approx f(y_i)$, $i= 0, 1, \dots, M-1$ are given. The discovery of the dynamics $f$ is based on reformulating  the scheme \eqref{eqn:hybrid_LMM_forward_scheme} as a linear system for $\hat{f}$:
\begin{equation}\label{eqn:hybrid_LMM_inv_Du2021}
    \sum_{m=0}^{M} \beta_m \hat{f}_{n-m}= \frac{1}{h} \sum_{m=0}^{M}\alpha_m  y_{n-m},\quad n= M, M+1, \dots, N.
\end{equation}
The linear system admits a unique solution since the matrix is lower-triangular with nonzero diagonal entries. To analyze the reconstruction error, Keller and Du employed the first and the second characteristic polynomials \cite{Mayers:2003}:
\begin{equation}\label{eqn:hybrid_LMM_charpoly}
    \rho(z)=\sum_{m=0}^M \alpha_{M-m} z^m\quad \text{and}\quad \sigma(z)=\sum_{m=0}^M \beta_{M-m} z^m.
\end{equation}
The error analysis of the discovery of the dynamics $f$ using LMMs can then be formulated in terms of the characteristic polynomials $\rho(z)$ and $\sigma(z)$ \cite[Lemma 3.3, Theorem 3.10]{DuKeller:2021}. The analysis builds on the following concepts:
\begin{itemize}
    \item A LMM scheme for dynamics discovery is said to be consistent of degree $p$ if the local truncation error
    $$(\tau_h)_n:=\frac{1}{h}\sum_{m=0}^M \alpha_m y_{n-m}-\sum_{m=0}^M \beta_mf(y_{n-m})$$  satisfies
    $$N^{p-1}\|\tau_h\|_{\infty}\to 0\quad \mbox{as }h\to 0.$$
    A LMM is consistent of degree $p$ provided that
    $$\rho(e^z)-z\sigma(e^z)=O(z^{p+1})\quad \mbox{as }z\to 0.$$
    \item A LMM scheme for dynamics discovery is said to be stable if there exists a constant $c >0$ , not depending on $N$, such that, for any two grid functions $u, v$, we have
    \begin{equation*}
               \max_{0\le i\le N}|u_i-v_i|\le c\Big(\max_{0\le i\le M-1}|u_i-v_i|+\max_{0\le i\le N}\Big|\sum_{m=0}^M \beta_m (u_{i-m}-v_{i-m})\Big| \Big).
    \end{equation*}
    A LMM is stable provided that  the roots of $\sigma(z)$ have modulus less than $1$.
\end{itemize}

The characterization of the consistency error follows from the classical truncation error analysis for LMMs, while the characterization of stability is investigated using the recurrence of the error. By combining the consistency and stability, we obtain the following error estimate for the LMM scheme \cite[Theorem 3.14]{DuKeller:2021}.
\begin{theorem}\label{thm:hybrid_LMM_Du2021}
Let $f$ be the exact dynamics and $\hat{f}$ the solution of \eqref{eqn:hybrid_LMM_inv_Du2021}. If a LMM scheme for dynamics discovery is consistent of degree $p$ and stable, and the initialization satisfies $\max_{0\le i\le M-1}|f_i-\hat{f}_i|=O(h^p)$, then there holds $$\max_{0\le i\le N}|f_i-\hat{f}_i|=O(h^p).$$
\end{theorem}
\begin{proof}
    By \eqref{eqn:hybrid_LMM_inv_Du2021}, we have for $n=M,M+1,\dots,N$,
    \begin{equation*}
        \sum_{m=0}^{M} \beta_m (\hat{f}_{n-m}-f_{n-m})=\frac{1}{h}\sum_{m=0}^{M} \alpha_m y_{n-m}-\sum_{m=0}^{M} \beta_m f(y_{n-m})= (\tau_h)_n.
    \end{equation*}
    Since the LMM scheme for dynamics discovery is stable, there exists $c$ independent of $h$, such that
    \begin{equation*}
       \max_{0\le i\le N}|f_i-\hat{f}_i|\le c\big(\max_{0\le i\le M-1}|f_i-\hat{f}_i|+\max_{0\le i\le N}|(\tau_h)_i| \big).
    \end{equation*}
    By the consistency error and initialization error bound, we arrive at the desired error estimate.
\end{proof}

\begin{remark}\label{rmk:hybrid_LMM_Du2021}
The stability of LMMs for the direct problem is characterized by the polynomial $\rho(z)$. Thus, traditional stable LMMs for solving direct problems may be unstable for dynamics discovery. Indeed, the Adams-Bashforth and Adams-Moulton schemes are stable for all $M$ when solving direct problems \cite{Henrici:1962}. However, for dynamics discovery, the Adams-Bashforth scheme is stable only for $1\le M\le 6$ and the Adams-Moulton scheme is stable only for $M=0,1$ \cite[Theorem 4.2]{DuKeller:2021}.
\end{remark}

Later, Du et al \cite{DuGu:2022} combined  LMMs and DNNs for dynamics discovery. They employed the one-sided finite difference method (FDM) of order $p$ to generate the first $M$ unknown dynamics:
\begin{equation}\label{eqn:hybrid_LMM_aux_Du2022}
    \hat{f}_n=\frac{1}{h}\sum_{m=0}^{p} \gamma_m y_{n+m},\quad n=0,1,\dots,M-1,
\end{equation}
where $\gamma_m$ are the finite difference coefficients. The authors employed one DNN $f_\theta$ to represent the unknown dynamics and minimized the following loss $J_h(f_\theta)$ based on the residuals of \eqref{eqn:hybrid_LMM_inv_Du2021} and \eqref{eqn:hybrid_LMM_aux_Du2022}:
\begin{align}\label{eqn:hybrid_LMM_loss_Du2022}
        J_h(f_\theta)=&\frac{1}{N}\sum_{n=0}^{M-1}\left|f_\theta(y_n)-\frac{1}{h}\sum_{m=0}^{p} \gamma_m y_{n+m} \right|^2\\
        &+\frac1N\sum_{n=M}^{N}\left| \sum_{m=0}^{M}\beta_m f_\theta(y_{n-m})-\frac1h\sum_{m=0}^{M} \alpha_m y_{n-m} \right|^2.\nonumber
\end{align}
The error estimate of the discovery on the trajectory $\mathcal{T}:=\{y(t):0\le t\le T\}$ was analyzed in terms of the $\ell^2$ seminorm: $|f|_{2,h}:=((N+1)^{-1}\sum_{n=0}^{N}|f(y_n)|^2)^{1/2}$, for all $f\in C(\mathcal{T})$. Let $f$ be the exact dynamics and $f_{\mathcal{N},h}^*$ a global minimizer of \eqref{eqn:hybrid_LMM_loss_Du2022} using a LMM of order $p$ with an admissible set $\mathcal{N}$. Then the following error estimate holds \cite[Theorem 5.1]{DuGu:2022}.
\begin{theorem}\label{thm:hybrid_LMM_err_Du2022}
    Suppose $y\in C^\infty([0,T]^d)$ and $f$ is defined on $ \mathcal{T}'$, a small neighborhood of $\mathcal{T}$. Then we have
    \begin{equation*}
        \left| f_{\mathcal{N},h}^*-f   \right|_{2,h}< c\kappa(A_h)(h^p+e_{\mathcal{N}}),
    \end{equation*}
    where $e_{\mathcal{N}}$  satisfies $e_{\mathcal{N}}>\inf_{f_\theta\in \mathcal{N}}\sup_{y\in\mathcal{T}}|f_\theta(y)-f(y)| $, $\kappa(A_h)$ denotes the condition number of the matrix $A_h$ (characterized by the linear system \eqref{eqn:hybrid_LMM_inv_Du2021} and \eqref{eqn:hybrid_LMM_aux_Du2022}), and the constant $c>0$ is independent of $h$ and $\mathcal{N}$.
\end{theorem}
\begin{proof}
With $f_n:=f(y_n)$, we can rewrite the loss function $J_h(f_\theta)$ as
    \begin{align*}
        J_h(f_\theta)=&\frac{1}{N}\sum_{n=0}^{M-1}\left|(f_\theta-f)_n+f_n-\frac{1}{h}\sum_{m=0}^{p} \gamma_m y_{n+m} \right|^2\\
        &+\frac1N\sum_{n=M}^{N}\left| \sum_{m=0}^{M}\beta_m (f_\theta -f)_{n-m}
        +\sum_{m=0}^{M}\beta_m f_{n-m}-\frac1h\sum_{m=0}^{M} \alpha_m y_{n-m} \right|^2\\
        =:&\frac{1}{N}\left\| A_h (f_\theta-f)- \begin{bmatrix} e_h \\ \tau_h \end{bmatrix} \right\|_2^2,
    \end{align*}
where $e_h$ and $ \tau_h$ are the errors arising from one-side FDM and LMM, respectively. There exists $c>0$ independent of $h$ such that
    \begin{equation*}
        \|\tau_h\|_2+\|e_h\|_2 \le c h^p.
    \end{equation*}
    Since $f_{\mathcal{N},h}^*$ is a global minimizer of $J_h$, there holds $J_h(f_{\mathcal{N},h}^*)\le J_h(f_\theta)$ for all $f_\theta\in \mathcal{N}$. That is,
    \begin{align*}
        \frac{1}{N}\left\| A_h (f_{\mathcal{N},h}^*-f)- \begin{bmatrix} e_h \\ \tau_h \end{bmatrix} \right\|_2^2\le \frac{1}{N}\left\| A_h (f_\theta-f)- \begin{bmatrix} e_h \\ \tau_h \end{bmatrix} \right\|_2^2.
    \end{align*}
Consequently,
    \begin{align*}
        &\|f_{\mathcal{N},h}^*-f \|_2\le \|A_h^{-1}\|_2
        \left\| A_h (f_{\mathcal{N},h}^*-f) \right\|_2\\
        \le & \|A_h^{-1}\|_2 \big(  \| A_h\|_2 \| (f_\theta-f)\|_2 + 2(\|  e_h \|_2+  \|  \tau_h\|_2 )^2\big)\\
        \le &c \kappa_2(A_h)(e_{\mathcal{N}}+h^p).
    \end{align*}
This completes the proof of the theorem.
\end{proof}

The term $e_{\mathcal{N}}$ represents the DNN approximation error. By Lemma \ref{lem:NN_approx_tanh} or similar approximation results \cite{Shen:2020,Lu:2021approx}, we can bound $e_{\mathcal{N}}$ in terms of DNN parameters, e.g., width and depth. The condition number $\kappa(A_h)$ can be uniformly bounded in $h$, if all roots of the second characteristic polynomial $\sigma(z)$ have modulus smaller than $1$ \cite[Theorem 5.4]{DuGu:2022}. Therefore, the $\ell^2$ error decays to zero as $h \to  0^+$ and the DNN size approaches infinity.

The \textit{a priori} analysis in Du et al \cite{DuGu:2022} only quantifies the error evaluated at the given sample locations. Zhu et al \cite{ZhuWuTang:2024} investigated the error estimate on out of sample locations. Consider the training data
\begin{equation*}
    \mathcal{T}_{train}=\{ ( x_{k}, \phi_{h,f}(x_{k}), \dots, \phi_{Mh,f}(x_{k}) ) \}_{k=1}^{K},
\end{equation*}
where $( x_{k})_{k=1}^{K}$ denotes different locations of initial states, $\phi_{t,f}(x_{k}):=x_{k}+\int_0^t f(y(\tau ))\d \tau$ denotes the flow map of the dynamics $f$ starting at the state $x_{k}$. The training loss is formulated as
\begin{equation*}
    J_{h,K}(f_{\theta})=\frac{1}{K} \sum_{k=1}^{K} \left\| \sum_{m=0}^{M}h^{-1}\alpha_m \phi_{mh,f}(x_{k})-\sum_{m=0}^{M}\beta_m f_{\theta}( \phi_{mh,f}(x_{k}) )\right\|_2^2.
\end{equation*}
In addition, let $\mathcal{K}\subset \mathbb{R}^d$ be a bounded domain and $\mathcal{B}(\mathcal{K},R)=\bigcup_{x\in \mathcal{K}}\mathcal{B}(x,R)$ be a neighborhood of $\mathcal{K}$. We define the expected loss by
\begin{equation*}
    J(f_{\theta})=\int_{\mathcal{B}(\mathcal{K},R)}  \left\| \sum_{m=0}^{M}h^{-1}\alpha_m \phi_{mh,f}(x)-\sum_{m=0}^{M}\beta_m f_{\theta}( \phi_{mh,f}(x) )\right\|_2^2 \d x.
\end{equation*}

Then the following error estimate holds \cite[Theorem 4.1]{ZhuWuTang:2024}. For a minimizer $f_\theta^*$ of the empirical training loss $J_{h,K}$, the estimate \eqref{eqn:hybrid_LMM_err_ZhuWuTang2024} can be expressed explicitly in terms of the time step $h$, DNN architectural parameters $W$ and $L$ and the number of sampling points $K$  \cite[Theorem 4.2]{ZhuWuTang:2024}.
\begin{theorem}\label{thm:hybrid_LMM_err_ZhuWuTang2024}
Consider a stable and consistent LMM of order $p$. Let $Q, R, r>0$, and $\mathcal{K} \subset \mathbb{R}^d$ be a bounded domain. Suppose that the target function $f$ is analytic on $\mathcal{B}(\mathcal{K}, R+r)$ with $\|f\|_{\mathcal{B}(\mathcal{K}, R+r)}\le Q$.
Let $f_{\theta^*}$ be a learned vector field which satisfies
$$\| f_{\theta^*}^{(i)}\|_{\mathcal{B}(\mathcal{K}, R)}\le i!\cdot Q\cdot r^{-i}, \quad \mbox{for }i=0,\dots,I,$$
with $I \geq \gamma / h^{2 / 5}+1$. Then for any sufficiently small time step $0<h<h_0$, there holds
\begin{align}
\label{eqn:hybrid_LMM_err_ZhuWuTang2024}
\int_{\mathcal{K}}\left\|f_{\theta^*}(x)-f(x)\right\| \d x \leq c(h^p+ J(f_{\theta^*})^{\frac{1}{2}})
\end{align}
where  constants $\gamma, h_0, c$ depend only on $r / Q, R / Q$, the dimension $d$, the volume of the domain $\mathcal{B}(\mathcal{K}, R)$, and the LMM coefficients $\alpha_m$ and $\beta_m$.
\end{theorem}
Theorem \ref{thm:hybrid_LMM_err_ZhuWuTang2024} quantities the error in the whole domain $\mathcal{K}$. The proof of Theorem \ref{thm:hybrid_LMM_err_ZhuWuTang2024} utilizes the inverse modified differential equation (IMDE) \cite{ZhuJinZhuTang:2022}. The IMDE is a perturbed differential equation, which satisfies the LMM scheme \eqref{eqn:hybrid_LMM_forward_scheme} formally.
Zhu et al \cite{ZhuWuTang:2024} establish the estimate in two steps. The first step involves truncating the IMDE and obtaining an $O(h^p)$ estimation of the difference between the target function and the truncated IMDE.
In the second step, the error between the neural network function and the truncated IMDE on domain $\mathcal{K}$ is estimated using interpolation. The regularity properties of IMDE and $f_\theta$ control the interpolation residual, while the error on interpolation points is bounded by the expected loss $J(f_\theta)$. Thus, the difference between the learned network and the truncated IMDE is bounded by the trained loss $J(f_\theta)$ and an additional discrepancy term that can be made exponentially small.


\section{DNN based approaches}\label{sec:NN}

In this section, we discuss the unsupervised approaches using DNNs to approximate both the parameter and the state. The development of neural solvers for direct and inverse problems associated with PDEs has attracted much attention recently, and many neural solvers have been proposed, including physics informed neural networks \cite{Raissi:2019}, deep Galerkin method \cite{SirignanoSpiliopoulos:2018}, deep Ritz method \cite{EYu:2018} and weak adversarial network \cite{ZangBaoYe:2020}, deep least-squares method \cite{CaiChoiLiu:2024}.

The error analysis of neural network approximations for direct problems associated with PDEs have received much attention; see \cite{Lu:2021priori,muller:2022error,de:2024numerical,Jiao:2021DRM} and the references therein. The error bound is generally derived in two steps. First, the underlying loss enjoys certain coercivity that enables bounding the error of the numerical approximation by the loss. Second, one often decomposes the error in the loss into three parts. The next lemma gives a decomposition of the error for the loss.

\begin{lemma}\label{lem:error-decomp}
Let $\mathcal{J}(a)$ be the population loss defined over a hypothesis space $\mathcal{H}$, and the $\widehat{\mathcal{J}}(a)$ be the empirical loss. Let $\theta_\mathcal{A}$ be an approximate minimizer to $\widehat{\mathcal{J}}(a_\theta)$ by suitable algorithm over the DNN parameter set $\Theta$. Then the following error decomposition holds
\begin{align*}
    & \mathcal{J}(a_{\theta_\mathcal{A}})-\inf_{a\in \mathcal{H}} \mathcal{J}(a) \\
   \leq &2\sup_{\theta\in\Theta}|\mathcal{J}(a_{\theta})- \widehat{\mathcal{J}}(a_{\theta})| + [\widehat{\mathcal{J}}(a_{\theta_\mathcal{A}}) - \inf_{\theta\in\Theta}\widehat{\mathcal{J}}(a_{\theta})] \\
   &+ [\inf_{\theta\in\Theta}\mathcal{J}(a_{\theta})-\inf_{a\in\mathcal{H}}\mathcal{J}(a)].
\end{align*}
\end{lemma}
\begin{proof}
Let $\widehat{\theta}^*\in \Theta$ and $\theta^*\in \Theta$ be a global minimizer to the losses $\widehat{\mathcal{J}}$ and $\mathcal{J}$, respectively. Then we have
\begin{align*}
\mathcal{J}(a_{\theta_\mathcal{A}})-\inf_{a\in \mathcal{H}} \mathcal{J}(a) = & [\mathcal{J}(a_{\theta_\mathcal{A}})- \widehat{\mathcal{J}}(a_{\theta_\mathcal{A}})] + [\widehat{\mathcal{J}}(a_{\theta_\mathcal{A}}) - \widehat{\mathcal{J}}(a_{\widehat{\theta}^*})] \\
 &+ [\widehat{\mathcal{J}}(a_{\widehat{\theta}^*})-\widehat{\mathcal{J}}(a_{{\theta}^*})] + [\widehat{\mathcal{J}}(a_{{\theta}^*})-\mathcal{J}(a_{\theta^*})] \\
 &+ [\mathcal{J}(a_{\theta^*})-\inf_{a\in\mathcal{H}}\mathcal{J}(a)].
\end{align*}
By the minimizing property of $\widehat{\theta}^*$ to $\widehat{\mathcal{J}}$ over $\Theta$, we have $\widehat{\mathcal{J}}(a_{\widehat{\theta}^*})\leq \widehat{\mathcal{J}}(a_{{\theta}^*})$. Then the desired assertion follows.
\end{proof}

The decomposition in Lemma \ref{lem:error-decomp} splits the error into three parts: approximation error, statistical error and optimization error.
\begin{itemize}
\item[(i)] The \textit{approximation error} $\inf_{\theta\in\Theta}\mathcal{J}(a_{\theta})-\inf_{a\in\mathcal{H}}\mathcal{J}(a)$ is related to the approximation capacity of the DNN ansatz space to the space $\mathcal{H}$. Often a larger function class yields a smaller approximation error, and the faster is the decay of the approximation error for a smoother target function.
\item[(ii)] The \textit{statistical error} $2\sup_{\theta\in\Theta}|\mathcal{J}(a_{\theta})- \widehat{\mathcal{J}}(a_{\theta})|$ is due to the use of quadrature rules for numerical integration. In the high-dimensional case, this is commonly achieved with (quasi) Monte Carlo methods, and the error can be bounded via Rademacher complexity \cite[Definition 2]{BartlettMendelson:2002} and covering number  of the DNN function class.
\item[(iii)] The \textit{optimization error} $\widehat{\mathcal{J}}(a_{\theta_\mathcal{A}}) - \widehat{\mathcal{J}}(a_{\widehat{\theta}^*})$ arises from the nonconvexity of the loss due to the nonlinearity of the DNN functions, and thus the optimizer may not find a global minimum $\widehat{\theta}^*$ of the empirical loss $\widehat{\mathcal{J}}$.
\end{itemize}

This decomposition holds for many machine learning tasks. It was adapted to analyze neural PDE solvers recently. The works \cite{Lu:2021priori,Jiao:2021DRM,muller:2022error} investigate the error analysis of the deep Ritz method \cite{EYu:2018} for elliptic equations; see also the recent survey \cite{de:2024numerical} for relevant studies with physics informed neural networks.

So far there are only a few studies on the error analysis of DNN approximations for inverse problems.
Jin et al \cite{Jin:2022CDII} proposed a deep Ritz method for current density impedance imaging based on a relaxed weighted least gradient formulation. Recently, Cen et al \cite{Cen:2025} employed the mixed least-square DNN to recover the anisotropic diffusion coefficient. These works utilize the approximation results and Rademacher complexity argument to derive error bounds for the empirical loss, which may be insufficient for deriving bounds on the reconstruction error. For the latter, it is essential to design and analyze loss formulations based on the conditional stability results.

\subsection{Preliminary}\label{subsec:NN_pre}
We use the DNN architecture in Section \ref{subsec:Hybrid_Pre}, and set $k=2$ and $s=1$ in Lemma \ref{lem:NN_approx_tanh}. Let $\mathfrak{P}_{p,\epsilon}$ be the DNN parameter set
\begin{equation*}
    \mathcal{N}\Big(c\log(d+2), c\epsilon^{-\frac{d}{1-\mu}}, c \epsilon^{-\frac{6p+3d+3pd}{p(1-\mu)}}\Big).
\end{equation*}
We use the parallelization operation to assemble multiple DNNs into one larger DNN. It is useful for approximating the components of a vector-valued function. Moreover, the new DNN does not change the depth, and its width equals to the sum of that of the subnetworks.
\begin{lemma}\label{lem:NN_paral}
Let $\bar \theta=\{(\bar A^{(\ell)},\bar b^{(\ell)})\}_{\ell=1}^L,\tilde\theta = \{(\tilde A^{(\ell)},\tilde b^{(\ell)})\}_{\ell=1}^L\in\Theta$,  let $\bar v$ and $\tilde v$ be their DNN realizations, and define $\theta=\{( A^{(\ell)}, b^{(\ell)})\}_{\ell=1}^L$ by
	\begin{align*}
		 A^{(1)}&=
		\begin{bmatrix}
		\bar A^{(1)}  \\
		\tilde A^{(1)}
		\end{bmatrix}\in\mathbb{R}^{2d_1\times d_0}, \
	     A^{(\ell)}=
	    \begin{bmatrix}
	   \bar A^{(\ell)} & 0 \\
	    	0            & \tilde A^{(\ell)}
	    \end{bmatrix}\in\mathbb{R}^{2d_\ell\times 2d_{\ell-1}},\quad  \ell=2,\dots,L; \\
		 b^{(\ell)} &=
		\begin{bmatrix}
	\bar b^{(\ell)} \\
		\tilde b^{(\ell)}
		\end{bmatrix}\in \mathbb{R}^{2d_\ell}, \quad \ell=1,\dots,L.
	\end{align*}
Then {$ v_{\theta}=(\bar v_{\bar \theta},\tilde v_{\tilde \theta})^\top $} is the DNN realization of $\theta$,
of depth $L$ and width $2{W}$.
\end{lemma}

For a Banach space $B$, the notation $B^{\otimes d}$ denotes the $d$-fold product space, and $ \mathfrak{P}_{p,\epsilon}^{\otimes d}$ denotes the parallelized DNN of $d$ identical DNN parameter classes $\mathfrak{P}_{p,\epsilon}$ with depth $L$ and width $W$. The next lemma gives the Lipschitz continuity of the tanh-DNN function class.
\begin{lemma}\label{lem:NN_Lip}
    Let $\Theta$ be a parametrization with depth $L$ and width $W$, and $\theta=\{(A^{(\ell)},b^{(\ell)})_{\ell=1}^{L} \} ,\,\tilde{\theta}=\{(\tilde A^{(\ell)},\tilde b^{(\ell)})_{\ell=1}^{L} \}\in\Theta$ satisfying $\|\theta\|_{\ell^\infty},
    \|\tilde{\theta}\|_{\ell^\infty}\leq R$. Then for the DNN realizations $v,\tilde v:\Omega\rightarrow \mathbb{R}$ of $\theta,\tilde \theta$, the following estimates hold
    \begin{align*}
        \|v-\tilde{v}\|_{L^\infty(\Omega)}\leq & 2LR^{L-1}W^{L}\|\theta-\tilde\theta\|_{\ell^\infty},\\
        \|\nabla (v-\tilde{v})\|_{L^\infty(\Omega; \mathbb{R}^d)}\leq &  \sqrt{d}L^2R^{2L-2}W^{2L-2}\|\theta-\tilde\theta\|_{\ell^\infty}.
    \end{align*}
\end{lemma}
\begin{proof}
    By the DNN realization \eqref{eqn:NN_realization}, the structure of the network, the bounds of $\rho=\mathrm{tanh}$  in  \eqref{eqn:NN_act_bound} and Lemma \ref{lem:NN_bounds}, we derive that for every layer $\ell=1,\dots,L$ and each component $i=1,\cdots,d_\ell$,
    \begin{align*}
        \|v_i^{(\ell)}-\tilde{v}_i^{(\ell)}\|_{L^\infty(\Omega)}
        \le  & \left\| \left(A^{(\ell)} v^{(\ell-1)}+b^{(\ell)}\right) - \left(\tilde{A}^{(\ell)} \tilde{v}^{(\ell-1)}+\tilde{b}^{(\ell)}\right)  \right\|_{L^\infty(\Omega)}\\
        \le & |b_i^{(\ell)}- \tilde{b}_i^{(\ell)} |+ \sum_{j=1}^{d_{\ell-1}} | A_{ij}^{(\ell)}- \tilde{A}_{ij}^{(\ell)}| \|v_j^{(\ell-1)}\|_{L^\infty(\Omega)}\\
        &+\sum_{j=1}^{d_{\ell-1}} |  \tilde{A}_{ij}^{(\ell)}| \|v_j^{(\ell-1)}-\tilde{v}_j^{(\ell-1)}\|_{L^\infty(\Omega)}    \\
        \le & \|\theta-\tilde\theta\|_{\ell^\infty}+W\|\theta-\tilde\theta\|_{\ell^\infty}\\
         &+RW \max_{1\le j\le d_{\ell-1}}\|v_j^{(\ell-1)}-\tilde{v}_j^{(\ell-1)}\|_{L^\infty(\Omega)}.
    \end{align*}
    Note also the trivial estimate
    \begin{align*}
        \|v_i^{(1)}-\tilde{v}_i^{(1)}\|_{L^\infty(\Omega)}
        \le& \left\| \sum_{j=1}^{d} A_{ij}^{(1)} x_j +b_i^{(1)}-\sum_{j=1}^{d} \tilde{A}_{ij}^{(1)} x_j -\tilde{b}_i^{(1)} \right\|_{L^\infty(\Omega)}\\
        \le& (W+1)\|\theta-\tilde\theta\|_{\ell^\infty}.
    \end{align*}
By applying the inequality recursively, we arrive at
    \begin{align*}
        \|v_i^{(\ell)}-\tilde{v}_i^{(\ell)}\|_{L^\infty(\Omega)}\le (1+W)\sum_{k=0}^{\ell-1} (RW)^k \|\theta-\tilde\theta\|_{\ell^\infty}\le 2\ell W^{(\ell)}R^{(\ell-1)}\|\theta-\tilde\theta\|_{\ell^\infty}.
    \end{align*}
    This completes the proof of the first assertion. The estimate on $\|\nabla(v-\tilde v)\|_{L^\infty(\Omega;\mathbb{R}^d)}$ follows similarly and can be found in \cite[Lemma 3.4 and Remark 3.3]{Jin:2022CDII}.
\end{proof}

\subsection{Mixed least-squares approach}\label{ssec:MLSNN}
Now we describe a DNN approach for diffusion coefficient identification \eqref{eqn:inv_cond_gov} based on a mixed reformulation of the governing equation \cite{Jin:2024conductivity}, following Kohn and Lowe \cite{Kohn:1988} (cf. Remark \ref{rmk:FEM_otherformulation}), and provide an error estimate of the DNN approximations that is explicit in terms of the noise level, penalty parameters, and DNN architectural parameters (depth, width, and parameter bounds). The discussion follows the recent work \cite{Jin:2024conductivity}, and the approach has been extend to recovering anisotropic conductivity by Cen et al \cite{Cen:2025}.

By setting $\sigma= a\nabla u$, we can recast problem \eqref{eqn:inv_cond_gov} into the following first-order system
\begin{equation}\label{eqn:MLSNN_gov}
    \left\{
        \begin{aligned}
            \sigma=&a\nabla u,&&\mbox{ in }\Omega,\\
            -\nabla\cdot \sigma=&f,&&\mbox{ in }\Omega,\\
            u=&0,&&\mbox{ on }\partial\Omega.
        \end{aligned}
    \right.
\end{equation}
The measurement $z^\delta$ is noisy with an accuracy
\begin{equation}\label{eqn:MLSNN_noise}
    \delta:=\|u(a^\dagger)-z^\delta\|_{H^{\frac{3}{2}}(\Omega)}.
\end{equation}
Note that we have imposed a mild regularity condition on $z^\delta$. The mixed least-squares formulation is based on the system \eqref{eqn:MLSNN_gov} with the state $u$ replaced by $z^\delta$. We employ two DNNs to approximate the diffusion coefficient $a$ and the flux $\sigma$. For the diffusion coefficient $a$, we use a DNN function class $\mathfrak{P}_{p,\epsilon_a}$ (with $p\ge 2$ and tolerance $\epsilon_a$)  of depth $L_a$, width $W_a$ and maximum bound $R_a$, and for $\sigma:\Omega\rightarrow \mathbb{R}^d$, we employ $d$ identical DNN function classes $\mathfrak{P}_{2,\epsilon_\sigma}$  of depth $L_\sigma$, width $W_\sigma$ and maximum bound $R_\sigma$ and stack them into one DNN, cf. Lemma \ref{lem:NN_paral}. Throughout,  $\theta$ and $\kappa$ denote the parameters of DNN approximations to $a$ and $\sigma$, respectively.

The least-squares formulation \eqref{eqn:MLSNN_gov} leads to the following loss
\begin{align}\label{eqn:MLSNN_loss}
        J_\gamma(\theta,\kappa)=&\|\sigma_\kappa-P_{\!\!\mathcal{A}}(a_\theta)\nabla z^\delta\|_{L^2(\Omega)}^2+\gamma_\sigma\|\nabla\cdot \sigma_\kappa + f\|_{L^2(\Omega)}^2\\
        &+\gamma_a \|\nabla a_\theta\|_{L^2(\Omega)}^2+ \gamma_b\|\sigma_\kappa-a^\dagger \nabla z^\delta \|_{L^2(\partial\Omega)}^2. \nonumber
\end{align}
where $\gamma=(\gamma_\sigma,\gamma_a,\gamma_b)\in \mathbb{R}^3_+$ denotes the vector of penalty parameters, and $P_{\!\!\mathcal{A}}$ is the projection operator, cf.  \eqref{eqn:hybrid_proj}.  The $H^1(\Omega )$ penalty on $a_\theta$  is to stabilize the minimization process. The boundary penalty $\gamma_b\|\sigma_\kappa-a^\dagger \nabla z^\delta \|_{L^2(\partial\Omega)}^2$ incorporates the \textit{a priori} information of the flux $\sigma$  on the boundary
$\partial \Omega $. The assumption on the knowledge $a^\dagger|_{\partial\Omega}$ is often made in existing studies \cite{Richter:1981,Alessandrini:1986}.

Note that the boundary condition is enforced through an additional penalty term in the loss. So the approximation only approximately satisfies the boundary conditions. This is generally referred to as having soft boundary conditions. In practice, for special domain geometries, one may construct the ansatz spaces such that the boundary conditions are satisfied exactly, which leads to the so-called hard boundary conditions \cite{LuJohnson:2021,SukumarSrivastava:2022}.

Then the DNN reconstruction scheme reads
\begin{equation}\label{eqn:MLSNN_scheme}
\min_{(\theta,\kappa)\in(\mathfrak{P}_{p,\epsilon_a},\mathfrak{P}_{2,\epsilon_\sigma}^{\otimes d}) } J_\gamma(\theta,\kappa).
\end{equation}

These loss functions $J_\gamma$ are formulated using integrals in terms of neural networks. Very often it is not feasible to evaluate these integrals exactly, as in the Galerkin FEM. Instead, one can approximate the integral with a sum using suitable quadrature rules. The choice of the quadrature rule depends crucially on the dimension $d$. In the low-dimensional case, one can use the tensor form of the composite mid-point rule, trapezoidal rule, and Gauss quadrature etc. Then the convergence rate depends on the regularity of the underlying integrand. However, these grid based quadrature rules are unsuitable for domains in high dimensions $(d\geq4)$. For moderately high-dimensions (i.e., $4\le d\leq 20$), one can use low-discrepancy sequences, e.g., Sobol and Halton sequences, as the quadrature points. So long as the integrand is of bounded Hardy-Krause variation, the quadrature error decays at a rate $n^{-1}$ (up to logarithmic factors), with $n$ being the number of quadrature points. For problems in high-dimensions ($d\geq20$), Monte-Carlo quadrature is the method of choice for numerical integration. Then the quadrature points are randomly chosen, independent and identically distributed (i.i.d.) with respect to a uniform distribution.

Next we describe the standard Monte Carlo method to discretize the integrals. Let $\mathcal{U}(\Omega)$ and $\mathcal{U}(\partial\Omega)$ be the uniform distributions
over $\Omega$ and  $\partial\Omega$, respectively, and
$(a_\theta,\sigma_\kappa)$ the DNN realization of $(\theta,\kappa)$.
The notation $\mathbb{E}_{\mathcal{U}(\Omega)}[\cdot]$ denotes taking the expectation with respect to $\mathcal{U}(\Omega)$ and likewise $\mathbb{E}_{\mathcal{U}(\partial\Omega)}$. Then we
can rewrite the loss $J_{\gamma}$ as
\begin{align*}
    J_{\gamma}(\theta,\kappa)
    &=|\Omega|\mathbb{E}_{X\sim\mathcal{U}(\Omega)}\Big[|\sigma_{\kappa}(X)- P_{\!\!\mathcal{A}}(a_\theta(X))\nabla z^\delta(X)|^2\Big] \\ &\quad+\gamma_\sigma|\Omega|\mathbb{E}_{X\sim\mathcal{U}(\Omega)}\Big[\big(\nabla\cdot\sigma_\kappa(X)+f(X)\big)^2\Big]\\
    &\quad+\gamma_a|\Omega|\mathbb{E}_{X\sim\mathcal{U}(\Omega)}\Big[|\nabla a_\theta(X)|^2 \Big]\\
    &\quad+\gamma_b|\partial\Omega|\mathbb{E}_{Y\sim\mathcal{U}(\partial\Omega)}\Big[|\sigma_\kappa(Y) - a^\dag(Y) \nabla z^\delta(Y) |^2\Big]   \\
    &=: \mathcal{E}_d(\sigma_{\kappa},a_\theta) + \gamma_\sigma \mathcal{E}_\sigma(\sigma_\kappa)+\gamma_a\mathcal{E}_a(a_\theta)+\gamma_b \mathcal{E}_b(\sigma_\kappa),
\end{align*}
where $|\Omega|$ and $|\partial\Omega|$ denote the Lebesgue measure of $\Omega$ and $\partial\Omega$, respectively, and $|\cdot|$ denotes the Euclidean norm on $\mathbb{R}^d$. Thus, $J_{\gamma}(\theta,\kappa)$ is also known as the population loss. Let
$X=(X_j)_{j=1}^{n_d}$ and $Y=(Y_j)_{j=1}^{n_b}$ be i.i.d. samples drawn from $\mathcal{U}
(\Omega)$ and $\mathcal{U}(\partial\Omega)$, with $n_d,n_b\in\mathbb{N}$. Then the empirical loss $\widehat{J}_{\gamma}(\theta,\kappa)$ is given by
\begin{equation}\label{eqn:MLSNN_loss_quad}
	\widehat{J}_{\gamma}(\theta,\kappa)=:\mathcal{\widehat{E}}_d(\sigma_{\kappa},a_\theta)      +\gamma_\sigma \mathcal{\widehat{E}}_\sigma(\sigma_\kappa) +\gamma_a\mathcal{\widehat{E}}_a(a_\theta)+\gamma_b\mathcal{\widehat{E}}_b(\sigma_\kappa),
\end{equation}
where $\mathcal{\widehat{E}}_d(\sigma_{\kappa},a_\theta)$, $\mathcal{\widehat{E}}_\sigma(\sigma_\kappa)$, $\mathcal{\widehat{E}}_a(a_\theta)$
 and $\mathcal{\widehat{E}}_b(\sigma_{\kappa})$ are Monte Carlo approximations of $\mathcal{E}_d(\sigma_{\kappa},a_\theta)$, $\mathcal{E}_\sigma(\sigma_\kappa)$, $\mathcal{E}_a(a_\theta)$ and $\mathcal{E}_b(\sigma_{\kappa})$, obtained by replacing the expectation with a sample mean:
\begin{align*}
    \mathcal{\widehat{E}}_d(\sigma_{\kappa},a_\theta)&:=n_d^{-1}|\Omega|\sum_{j=1}^{n_d} |\sigma_{\kappa}(X_{j})
    -P_{\!\!\mathcal{A}}(a_\theta(X_{j}))\nabla z^\delta(X_{j})|^2,\\
    \mathcal{\widehat{E}}_\sigma(\sigma_\kappa)&:=n_d^{-1}|\Omega|\sum_{j=1}^{n_d}\big(\nabla\cdot\sigma_\kappa(X_j)+f(X_j)\big)^2, \\
    \mathcal{\widehat{E}}_a(a_\theta)&:=n_d^{-1}|\Omega|\sum_{j=1}^{n_d}|\nabla a_\theta(X_j)|^2.\\
    \mathcal{\widehat{E}}_b(\sigma_\kappa)&:=n_b^{-1}|\partial\Omega|\sum_{j=1}^{n_b}|\sigma_\kappa(Y_j) - a^\dag(Y_j)\nabla z^\delta(Y_j)|^2.
\end{align*}
The practical reconstruction scheme reads
\begin{equation}\label{eqn:MLSNN_scheme_quad}
    \min_{(\theta,\kappa)\in (\mathfrak{P}_{p,\epsilon_a},\mathfrak{P}_{2,\epsilon_\sigma}^{\otimes d}) } \widehat{J}_\gamma(\theta,\kappa).
\end{equation}
The well-posedness of problem \eqref{eqn:MLSNN_scheme_quad} follows by a standard argument in calculus of variation. Indeed, the compactness of the parameterizations $\mathfrak{P}_{p,\epsilon_a}$  and $\mathfrak{P}_{2,\epsilon_\sigma}$ holds due to the
uniform boundedness of the parameter vectors and the finite-dimensionality of the spaces.
The smoothness of the $\tanh$ activation implies the continuity of the loss $\widehat{J}_\gamma$ in the DNN parameters $(\theta,\kappa)$. These two properties ensure the existence of a global minimizer $(\widehat{\theta}^*, \widehat{\kappa}^*)$.

In practice, the empirical loss $\widehat{J}_{\bm \gamma}$ is minimized using gradient type methods, e.g., gradient descent, stochastic gradient descent or its variants. In the context of neural PDE solvers, the most popular methods include  Adam \cite{Kingma:2017} and L-BFGS \cite{ByrdLu:1995}. These algorithms have been implemented in standard machine learning software platforms, and require computing the gradient of the loss with respect to the DNN parameters, which can be computed using automatic differentiation \cite{Baydin:2018}. Theoretically, due to the nonlinearity of DNN functions, the empirical loss $\widehat{J}_\gamma$ is potentially highly nonconvex, and thus generally there is no theoretical guarantee that the optimization algorithms can achieve a global minimum. Nonetheless, standard optimization algorithms often work reasonably well. So far it remains very challenging to minimize the loss down to machine precision even in the over-parameterized regime. Indeed, most existing works on neural PDE solvers for direct problems only report $O(10^{-3})$ -- $O(10^{-2})$ accuracy for the relative error. It is widely believed that the presence of optimization errors has precluded numerically verifying the error estimate. Nonetheless, due to the inevitable presence of data noise, this does not seem to pose serious issues for inverse problems.

Below we estimate the  error $\|a^\dagger-P_{\!\!\mathcal{A}}(\widehat{a}_\theta^*)\|_{L^2(\Omega)}$of the approximation $P_{\!\!\mathcal{A}}(\widehat{a}_\theta^*)$, where $\widehat{a}_\theta^*$ denotes the realization of the DNN parameterization $\widehat{\theta}^*$. We need the following assumption on the data regularity.
\begin{assumption}\label{assum:MLSNN_reg}
$a^\dagger  \in  W^{2,p}(\Omega ) \cap  \mathcal{A}$, $p >\max(2, d  )$, $f \in  H^1(\Omega ) \cap  L^\infty (\Omega )$ and $\nabla z^\delta\in L^{\infty}(\Omega)\cap L^\infty(\partial\Omega)$.
\end{assumption}
The $L^\infty$ bound is needed in order to apply the standard Rademacher complexity argument to bound the quadrature errors arising from the Monte Carlo method. Under Assumption \ref{assum:MLSNN_reg}, the elliptic regularity theory implies $u^\dag:=u(a^\dag)\in H^3(\Omega)\cap W^{2,p}(\Omega)\cap H_0^1(\Omega)$ and the flux $\sigma^\dag:=a^\dag \nabla u^\dag\in (H^2(\Omega))^d$.

The next result gives an \textit{a priori} estimate on the population loss $J_{\gamma}(\theta^*,\kappa^*)$.
\begin{lemma}\label{lem:MLSNN_min}
Let Assumption \ref{assum:MLSNN_reg} hold. Fix small $\epsilon_a$, $\epsilon_\sigma>0$, and let
$(\theta^*,\kappa^*)\in (\mathfrak{P}_{p,\epsilon_a}, \mathfrak{P}_{2,\epsilon_\sigma
}^{\otimes d})$ be a minimizer of the loss $J_{\gamma}(\theta,\kappa)$ in \eqref{eqn:MLSNN_loss}. Then there holds
    \begin{equation*}
    	J_{\gamma}(\theta^*,\kappa^*)\leq c\big(\epsilon_a^2+(1+\gamma_\sigma+\gamma_b)\epsilon_\sigma^2+(1+\gamma_\sigma)\delta^2+\gamma_a\big).
    \end{equation*}
\end{lemma}
\begin{proof}
    By the regularity of $u^\dag$ and $\sigma^\dag$ and Lemma \ref{lem:NN_approx_tanh}, we deduce that there exists at least one $(\theta_\epsilon,\kappa_\epsilon)\in
    (\mathfrak{P}_{p,\epsilon_a}, \mathfrak{P}_{2,\epsilon_\sigma
    }^{\otimes d})$ such that its DNN realization $(a_{\theta_\epsilon},\sigma_{\kappa_\epsilon})$ satisfies
    \begin{equation}\label{eqn:MLSNN_NNapprox-1}
        \|a^\dag-a_{\theta_\epsilon}\|_{W^{1,p}(\Omega)}\leq\epsilon_a\quad \mbox{and}\quad  \|\sigma^\dag-\sigma_{\kappa_\epsilon}\|_{H^1(\Omega)}\leq \epsilon_\sigma.
    \end{equation}
    The Sobolev embedding $W^{1,p}(\Omega)\hookrightarrow L^\infty(\Omega)$ \cite[Theorem 4.12, p. 85]{AdamsFournier:2003} implies
    \begin{equation}\label{eqn:MLSNN_NNapprox-2}
        \|a^\dag-a_{\theta_\epsilon}\|_{L^\infty(\Omega)}\leq c\epsilon_a.
    \end{equation}
By the minimizing property of $(\theta^*,\kappa^*)$ and the triangle inequality, we have
    \begin{align*}
    	& J_{\gamma}(\theta^*,\kappa^*)\leq J_{\gamma}(\theta_\epsilon,\kappa_\epsilon) \\
    	=& \|\sigma_{\kappa_\epsilon}-P_{\!\!\mathcal{A}}(a_{\theta_\epsilon})\nabla z^\delta\|_{L^2(\Omega)}^2+\gamma_\sigma\|\nabla\cdot\sigma_{\kappa_\epsilon}+f\|_{L^2(\Omega)}^2\\
        &+\gamma_a\|\nabla a_{\theta_\epsilon}\|_{L^2(\Omega)}^2 +\gamma_b\|\sigma_{\kappa_\epsilon}-a^\dag\nabla z^\delta\|^2_{L^2(\partial\Omega)}  \\
        \leq& c\big(\|\sigma_{\kappa_\epsilon}-\sigma^\dag\|^2_{L^2(\Omega)}+ \|(a^\dag-P_{\!\!\mathcal{A}}(a_{\theta_\epsilon}))\nabla u^\dag\|_{L^2(\Omega)}^2\\
        &+ \|P_{\!\!\mathcal{A}}(a_{\theta_\epsilon})\nabla (u^\dag-z^\delta)\|_{L^2(\Omega)}^2 \\
        &+\gamma_\sigma\|\nabla\cdot(\sigma_{\kappa_\epsilon}- \sigma^\dag)\|_{L^2(\Omega)}^2+\gamma_a\|\nabla a_{\theta_\epsilon}\|_{L^2(\Omega)}^2\\
        &+\gamma_b\|\sigma_{\kappa_\epsilon}-\sigma^\dag\|^2_{H^1(\Omega)} +\gamma_b\|a^\dag\nabla(u^\dag-z^\delta)\|^2_{L^2(\partial\Omega)}\big).
    \end{align*}
    The stability estimates \eqref{eqn:hybrid_proj_stb}-\eqref{eqn:hybrid_proj_approx} of $P_{\!\!\mathcal{A}}$ lead to
    \begin{equation*}
    \begin{split}
      \|(a^\dag-P_{\!\!\mathcal{A}}(a_{\theta_\epsilon}))\nabla u^\dag\|_{L^2(\Omega)} & \leq \|a^\dag-P_{\!\!\mathcal{A}}(a_{\theta_\epsilon})\|_{L^\infty(\Omega)}\|\nabla u^\dag\|_{L^2(\Omega)} \\
       &\leq \|a^\dag-a_{\theta_\epsilon}\|_{L^\infty(\Omega)}\|\nabla u^\dag\|_{L^2(\Omega)},\\
      \|P_{\!\!\mathcal{A}}(a_{\theta_\epsilon})\nabla (u^\dag-z^\delta)\|_{L^2(\Omega)} & \leq \|P_{\!\!\mathcal{A}}(a_{\theta_\epsilon})\|_{L^{\infty}(\Omega)}\|\nabla(u^\dag-z^\delta)\|_{L^2(\Omega)}\\
       &\le \overline{c}_a \| \nabla(u^\dag-z^\delta)\|_{L^2(\Omega)}.
    \end{split}
    \end{equation*}
    The regularity of $a^\dag$ and the error estimate \eqref{eqn:MLSNN_NNapprox-1} imply $\|\nabla a_{\theta_\epsilon}\|_{L^2(\Omega)}\le c$.
    The \textit{a priori} estimates \eqref{eqn:MLSNN_NNapprox-1}-\eqref{eqn:MLSNN_NNapprox-2} and the trace theorem \cite[Theorem 5.36, p. 164]{AdamsFournier:2003} yield
    \begin{align*}
        J_{\gamma}(\theta^*,\kappa^*)&\leq c\big[\epsilon_\sigma^2 + \|a^\dag-a_{\theta_\epsilon}\|^2_{L^\infty(\Omega)} +  \|\nabla (u^\dag-z^\delta)\|^{2}_{L^2(\Omega)}\\
         &\quad + (\gamma_\sigma+\gamma_b)\|\sigma^\dag-\sigma_{\kappa_\epsilon}\|^2_{H^1(\Omega)}+\gamma_a  \| \nabla a_{\theta_\epsilon} \|^2_{L^2(\Omega)}\\ &\quad +\gamma_b\|u^\dag-z^\delta\|_{H^{\frac{3}{2}}(\Omega) } \big]\\
        &\leq c\big(\epsilon_q^2+(1+\gamma_\sigma+\gamma_b)\epsilon_\sigma^2+(1+\gamma_b)\delta^2+\gamma_a\big).
    \end{align*}
    This completes the proof of the lemma.
\end{proof}
Next we investigate the errors arising from the Monte Carlo approximation:
\begin{equation*}
    \sup_{a_\theta\in\mathcal{N}_{\theta},\sigma_\kappa\in\mathcal{N}_{\kappa}} \big|J_{\gamma}(a_\theta,\sigma_\kappa)-\widehat{J}_{\gamma}
    (a_\theta,\sigma_\kappa)\big|\leq \Delta\mathcal{E}_d + \gamma_\sigma \Delta\mathcal{E}_\sigma + \gamma_a \Delta\mathcal{E}_a + \gamma_b\Delta\mathcal{E}_b,
\end{equation*}
where $\mathcal{N}_{\theta}$ and $\mathcal{N}_{\kappa}$ are DNN function classes (for the sets $\mathfrak{P}_{p,\epsilon_q}$ and $\mathfrak{P}_{2,\epsilon_\sigma}^{\otimes d}$, respectively), and the error components are given by
\begin{align*}
    &\Delta\mathcal{E}_d: = \sup_{\sigma_{\kappa}\in\mathcal{N}_{\kappa},a_\theta\in\mathcal{N}_{\theta}}\big| \mathcal{E}_d(\sigma_{\kappa},a_\theta)-\mathcal{\widehat{E}}_d(\sigma_{\kappa},a_\theta)\big|,\\
    &\Delta\mathcal{E}_\sigma:= \sup_{\sigma_{\kappa}\in\mathcal{N}_{\kappa}}\big|\mathcal{E}_\sigma(\sigma_\kappa)-\mathcal{\widehat{E}}_\sigma(\sigma_\kappa)\big|, \\ &\Delta\mathcal{E}_b:=\sup_{\sigma_{\kappa}\in\mathcal{N}_{\kappa}}\big|\mathcal{E}_b (\sigma_\kappa)-\mathcal{\widehat{E}}_b(\sigma_\kappa)\big|, \\ &\Delta\mathcal{E}_a:=\sup_{a_\theta\in\mathcal{N}_{\theta}}\big|\mathcal{E}_a(a_\theta)-\mathcal{\widehat{E}}_a(a_\theta)\big|.
\end{align*}
These errors are known as statistical errors in statistical learning theory \cite{Anthony:1999,Shalev:20214}.

 The next result provides bounds on the statistical errors in terms of the number of sampling points and DNN architecture parameters. The proof uses Rademacher complexity \cite[Definition 2]{BartlettMendelson:2002}, which measures the complexity of a collection of functions by the correlation between function values with Rademacher random variables.
\begin{definition}
Let $\xi=(\xi_j)_{j=1}^n$ be i.i.d. samples from the distribution $\mathcal{U}(\Omega)$, and $\omega=( \omega_j)_{j=1}^n$ be i.i.d. Rademacher random variables with probability $P(\omega_j=1)=P(\omega_j=-1)=\frac12$. The Rademacher complexity $\mathfrak{R}_n(\mathcal{H}_d)$ of the class $\mathcal{H}_d$ is given by
	\begin{equation*}
        \mathfrak{R}_n(\mathcal{H}_d)= \mathbb{E}_{\xi,\omega}\bigg{[}\sup_{h\in\mathcal{H}_d}\ n^{-1}\bigg{\lvert}\ \sum_{j=1}^{n}\omega_j h(\xi_j)\ \bigg{\rvert} \bigg{]}.
	\end{equation*}
\end{definition}

\begin{theorem}\label{thm:MLSNN_err_stat}
Let Assumption \ref{assum:MLSNN_reg} hold, and $a_\theta\in \mathcal{N}(L_a,N_a,R_a)$ and $\sigma_\kappa \in \mathcal{N}(L_\sigma,N_\sigma,R_\sigma)$. Fix $\tau\in(0,\frac18)$, and define the following bounds
\begin{align*}
    e_d&:=c\frac{R_\sigma^{2}N_\sigma^{2}(N_\sigma+N_a)^{\frac12}(\log^\frac12 R_\sigma +\log^\frac12 N_\sigma + \log^\frac12 R_a+\log^\frac12 N_a+\log^\frac12 n_d)}{\sqrt{n_d}}\\
       &\quad + \tilde c R^2_\sigma N_\sigma^2\sqrt{\frac{\log\frac{1}{\tau}}{n_d}},\\
    e_\sigma&:=c\frac{R_\sigma^{2L_\sigma}N_\sigma^{2L_\sigma-\frac32}
    \big(\log^\frac12R_\sigma+\log^\frac12N_\sigma+\log^\frac12n_d\big)}{\sqrt{n_d}}+ \tilde cR_\sigma^{2L_\sigma}N_\sigma^{2L_\sigma-2}\sqrt{\frac{\log\frac{1}{\tau}}{n_d}},\\
    e_{b}&:=c\frac{R_\sigma^2N_\sigma^{\frac52}\big(\log^\frac12 R_\sigma+\log^\frac12  N_\sigma+\log^\frac12 n_b\big)}{\sqrt{n_b}}+\tilde cR_\sigma^{2}N_\sigma^{2}\sqrt{\frac{\log\frac{1}{\tau}}{n_b}},\\
    e_q&:=c\frac{R_a^{2L_a}N_a^{2L_a-\frac32}
    \big(\log^\frac12R_a+\log^\frac12N_a+\log^\frac12n_d\big)}{\sqrt{n_d}}+\tilde cR_a^{2L_a}N_a^{2L_a-2}\sqrt{\frac{\log\frac{1}{\tau}}{n_d}},
\end{align*}
where $c$ and $\tilde c$ depend on  $|\Omega|$, $|\partial\Omega|$, $d$,  $\|f\|_{L^\infty(\Omega)}$, $\|z^\delta\|_{W^{1,\infty}(\Omega)}$ and $\|z^\delta\|_{W^{1,\infty}(\partial\Omega)}$ at most polynomially. Then, with probability
at least $1-\tau$, each of the following statements holds
\begin{equation*}
    \Delta \mathcal{E}_i \leq e_{i}, \quad i\in\{d,\sigma,b,q\}.
\end{equation*}
\end{theorem}
\begin{proof}
The proof employs the Rademacher complexity of the function classes $\mathcal{N}_\theta:=\mathcal{N}(L_a,N_a,R_a)$ and $\mathcal{N}_\kappa:=\mathcal{N}(L_\sigma,N_\sigma,R_\sigma)$. To estimate $e_d$, we define an induced function class
    \begin{align*}
      \mathcal{H}_d  =& \{h: \Omega\to \mathbb{R}|\,\, h(x) = \|\sigma_\kappa(x)-P_{\!\!\mathcal{A}}(a_\theta(x))\nabla z^\delta(x)\|_{\ell^2}^2,\\ &\quad \sigma_\kappa\in \mathcal{N}_\kappa, a_\theta\in \mathcal{N}_\theta\}.
\end{align*}
The Rademacher complexity $\mathfrak{R}_n(\mathcal{H}_d)$ can be bounded by Dudley's formula (cf. \cite[Theorem 9]{Lu:2021priori} and \cite[Theorem 1.19]{Wolf:2018}):
    \begin{equation*}
    \mathfrak{R}_n(\mathcal{H}_d)\leq\inf_{0<s< M_{\mathcal{H}_d}}\bigg(4s\ +\ 12n^{-\frac12}\int^{M_{\mathcal{H}_d}}_{s}\big(\log\mathcal{C}(\mathcal{H}_d,\|\cdot\|_{L^{\infty}(\Omega)},\epsilon)\big)^{\frac12}\ {\rm d}\epsilon\bigg),
    \end{equation*}
    where $\mathcal{M}_{\mathcal{H}_d}:=\sup_{h\in \mathcal{H}_d} \|h\|_{L^{\infty}(\Omega)}$, and $\mathcal{C}(\mathcal{H}_d,\|\cdot\|_{L^{\infty}(\Omega)},\epsilon)$ is the covering number of the set $\mathcal{H}_d$ \cite[Definition 5.4]{Jiao:2021DRM}. The supremum $\|h\|_{L^{\infty}(\Omega)}$  and the covering number can be  estimated by Lemmas \ref{lem:NN_bounds} and \ref{lem:NN_Lip}:
    \begin{align*}
      \|h\|_{L^\infty(\Omega)} &\leq
         2(dR_\sigma^{2}(W_\sigma+1)^{2}+c_1^2c_z^2), \quad h\in \mathcal{H}_d,\\
      \|h-\tilde h\|_{L^\infty(\Omega)} & \le  \Lambda_d (\|\theta -\tilde\theta\|_{\ell^\infty} + \|\kappa-\tilde\kappa\|_{\ell^\infty}),\quad \forall h,\tilde h\in \mathcal{H}_d,
    \end{align*}
    with $c_z=\|\nabla z^\delta\|_{L^\infty(\Omega)}$ and
    \begin{align*}
      \Lambda_d =2\big(\sqrt{d}R_\sigma(W_\sigma+1)+c_1c_z\big)\max(2\sqrt{d}L_\sigma R_\sigma^{L_\sigma-1}W_\sigma^{L_\sigma},c_zL_aR_a^{L_a-1}W_a^{L_a}).
    \end{align*}
    The Lipschitz continuity of DNN functions with respect to the DNN parameters and \cite[Proposition 5]{CuckerSmale:2002} yield the bound
    \begin{align*}
         \log\mathcal{C}(\mathcal{H}_d,\|\cdot\|_{L^{\infty}(\Omega)},\epsilon)&\leq \log \mathcal{C}(\Theta_\sigma\otimes \Theta_a,\|\cdot\|_{\ell^\infty},\Lambda_d^{-1}\epsilon)\\
        &\leq (N_\sigma+N_a)\log(4\max(R_a,R_\sigma)\Lambda_d\epsilon^{-1}),
    \end{align*}
with $\Theta_\sigma$ and $\Theta_a$ being the DNN parameter sets for $\sigma$ and $a$, respectively. Finally, we employ the following PAC-type generalization bound \cite[Theorem 3.1]{Mohri:2018} to bound the statistical error:  for any $\tau\in(0,1)$, with probability at least $1-\tau$
    \begin{equation*}
      \sup_{h\in \mathcal{H}_d}\bigg|n^{-1}\sum_{j=1}^n h(X_j)-\mathbb{E}[h(X)]\bigg| \leq 2\mathfrak{R}_n(\mathcal{H}_d) + 2M_{\mathcal{H}_d}\sqrt{\frac{\log\frac{1}{\tau}}{2n}}.
    \end{equation*}
    This completes the proof of the theorem.
\end{proof}

Now we can state an error estimate on the approximation $\widehat{a}_\theta^*$ constructed via the empirical loss $\widehat{J}_{\gamma}(\theta,\kappa)$ in \eqref{eqn:MLSNN_loss_quad}.
\begin{theorem}\label{thm:MLSNN_err_quad}
    Let Assumption \ref{assum:MLSNN_reg} hold. Fix small $\epsilon_a$, $\epsilon_\sigma>0$, and let the tuple $(\widehat{\theta}^*,\widehat{\kappa}^*)\in (\mathfrak{P}_{p,\epsilon_a}, \mathfrak{P}_{2,\epsilon_\sigma}^{\otimes d})$ be a minimizer of the loss $\widehat J_{\gamma}(\theta,\kappa)$ in \eqref{eqn:MLSNN_loss_quad} and $\widehat{a}_\theta^*$ the
    DNN realization of $\widehat{\theta}^*$. Fix $\tau\in(0,\frac{1}{8})$, let the quantities $e_d$, $e_\sigma$, $e_{b}$ and $e_a$ be defined in Theorem \ref{thm:MLSNN_err_stat}, and let
    $$\eta:=e_d+\gamma_\sigma e_\sigma+\gamma_b e_{b}+\gamma_a e_a
    +\epsilon_a^2+ (1+\gamma_\sigma+\gamma_b)\epsilon_\sigma^2+(1+\gamma_b) \delta^2+\gamma_a.$$
    Then with probability at least $1-4\tau$, there holds
    \begin{align*}
        &\int_{\Omega}\Big(\frac{a^\dag-P_{\!\!\mathcal{A}}(\widehat{a}_\theta^*)}{a^\dag}\Big)^2\big(a^\dag |\nabla u^\dag|^2+fu^\dag\big)\ \d x\\
        \leq& c\big((e_d+\eta)^{\frac{1}{2}}+(e_\sigma+\gamma_\sigma^{-1}\eta)^{\frac{1}{2}}+\delta\big)\big(1+ e_a+\gamma_a^{-1}\eta\big)^{\frac{1}{2}},
    \end{align*}
    where $c>0$ depends on  $|\Omega|$, $|\partial\Omega|$, $d$,  $\|z^\delta\|_{W^{1,\infty}(\partial\Omega)}$, $\|f\|_{L^\infty(\Omega)}$ and $\|a^\dagger\|_{L^\infty(\partial\Omega)}$ at most polynomially.
    Moreover, if condition \eqref{eqn:Bonito_cond} holds, then with probability at least $1-4\tau$,
    \begin{equation*}
        \|a^\dagger-P_{\!\!\mathcal{A}}(\widehat{a}_\theta^*)\|_{L^2(\Omega)}\leq c\big(\big((e_d+\eta)^{\frac{1}{2}}+(e_\sigma+\gamma_\sigma^{-1}\eta)^{\frac{1}{2}}+\delta\big)\big(1+ e_a+\gamma_a^{-1}\eta\big)^{\frac{1}{2}}\big)^{\frac{1}{2(\beta+1)}}.
    \end{equation*}
\end{theorem}
 \begin{proof}
Let $(\theta^*,\kappa^*)\in (\mathfrak{P}_{p,\epsilon_a}, \mathfrak{P}_{2,\epsilon_\sigma}^{\otimes d})$ be a minimizer of the loss $J_\gamma(\theta,\kappa)$. The minimizing property of $(\widehat{\theta}^*,\widehat{\kappa}^*)$ to the empirical loss $\widehat J_{\gamma}(\theta,\kappa)$ implies
    \begin{align*}
      \widehat J_{\gamma}(\widehat\theta^*,\widehat\kappa^*)
       & \leq [\widehat J_{\gamma}(\theta^*,\kappa^*) - J_{\gamma}(\theta^*,\kappa^*)] + J_{\gamma}(\theta^*,\kappa^*)\\
       &\leq | \widehat J_{\gamma}(\theta^*,\kappa^*) - J_{\gamma}(\theta^*,\kappa^*)| +  J_{\gamma}(\theta^*,\kappa^*).
    \end{align*}
    Consequently,
    \begin{equation*}
       \widehat{J}_{\gamma}(\widehat\theta^*,\widehat\kappa^*)  \leq J_{\gamma}(\theta^*,\kappa^*) + \sup_{(\theta,\kappa)\in (\mathfrak{P}_{p,\epsilon_a}, \mathfrak{P}_{2,\epsilon_\sigma}^{\otimes d})} |J_{\gamma}(\theta,\kappa) - \widehat J_{\gamma}(\theta,\kappa)|.
    \end{equation*}
The term $J_{\gamma}(\theta^*,\kappa^*)$ can be bounded by the approximation error, cf. Lemma \ref{lem:MLSNN_min}, and the term $|J_{\gamma}(\theta,\kappa) - \widehat J_{\gamma}(\theta,\kappa)|$ represents the statistical error, which has been estimated in Theorem \ref{thm:MLSNN_err_stat}. Therefore,
    with probability at least $1-4\tau$, there holds
    $$\widehat{J}_{\gamma}(\widehat{\theta}^*,\widehat{\kappa}^*)\leq c\eta.$$
    This and the triangle inequality imply that the following estimates hold simultaneously with probability at least $1-4\tau$,
        \begin{align}
            \|\widehat{\sigma}_\kappa^*-P_{\!\!\mathcal{A}}(\widehat{a}_\theta^*)\nabla z^\delta\|_{L^2(\Omega)}^2\leq & \big|\mathcal{E}_d(\widehat{\sigma}^*_{\kappa},\widehat{a}^*_\theta) -\mathcal{\widehat{E}}_d(\widehat{\sigma}^*_{\kappa},\widehat{a}^*_\theta)\big|+  \mathcal{\widehat{E}}_d(\widehat{\sigma}^*_{\kappa},\widehat{a}^*_\theta)\label{eqn:MLSNN_priori_err}\\
            \leq& c(e_d+\eta),\nonumber\\
        \|\nabla\cdot\widehat{\sigma}_\kappa^*+f\|_{L^2(\Omega)}^2\leq& \big|\mathcal{E}_\sigma(\widehat{\sigma}^*_{\kappa}) -\mathcal{\widehat{E}}_\sigma(\widehat{\sigma}^*_{\kappa})\big|+  \mathcal{\widehat{E}}_\sigma(\widehat{\sigma}^*_{\kappa})\nonumber \\ \leq& c (e_\sigma+\gamma_\sigma^{-1}\eta),\nonumber\\
        \|\nabla \widehat{a}_\theta^*\|^2_{L^2(\Omega)}\leq& \big|\mathcal{E}_a( \widehat{a}^*_\theta) -\mathcal{\widehat{E}}_a( \widehat{a}^*_\theta)\big|+  \mathcal{\widehat{E}}_a( \widehat{a}^*_\theta)\nonumber\\ \leq& c(e_a+\gamma_a^{-1}\eta),\nonumber\\
        \|\widehat{\sigma}_{\kappa}^*-a^\dag \nabla z^\delta\|_{L^2(\partial\Omega)}^2\leq& \big|\mathcal{E}_b(\widehat{\sigma}^*_{\kappa}) -\mathcal{\widehat{E}}_b(\widehat{\sigma}^*_{\kappa})\big|+  \mathcal{\widehat{E}}_b(\widehat{\sigma}^*_{\kappa})\nonumber\\ \leq& c(e_b+\gamma_b^{-1}\eta).\nonumber
\end{align}
Now we follow the strategy of the conditional stability in Theorem \ref{thm:Bonito_stab}.  For any $\varphi\in H_0^1(\Omega)$, by integration by parts, we have
\begin{align*}
    &\big((a^\dag-P_{\!\!\mathcal{A}}(\widehat{a}_\theta^*))\nabla u^\dag,\nabla\varphi\big) \\
    =&\big (\sigma^\dag-\widehat{\sigma}_\kappa^*,\nabla\varphi\big)+\big(P_{\!\!\mathcal{A}}(\widehat{a}_\theta^*)\nabla(z^\delta-u^\dag),\nabla\varphi\big)+\big(\widehat{\sigma}_\kappa^*-P_{\!\!\mathcal{A}}(\widehat{a}_\theta^*)\nabla z^\delta,\nabla\varphi\big)  \\
    		=& -\big (\nabla\cdot(\sigma^\dag-\widehat{\sigma}_\kappa^*),\varphi\big)+\big(P_{\!\!\mathcal{A}}(\widehat{a}_\theta^*)\nabla(z^\delta-u^\dag), \nabla\varphi\big)
           +\big(\widehat{\sigma}_\kappa^*-P_{\!\!\mathcal{A}}(\widehat{a}_\theta^*)\nabla z^\delta,\nabla\varphi\big)\\
           =&:\mathrm{I}+\mathrm{II}+\mathrm{III}
    \end{align*}
    Let $\varphi\equiv\frac{a^\dag-P_{\!\!\mathcal{A}}(\widehat{a}_\theta^*)}{a^\dag}u^\dag$. Then by direct computation, we have
    \begin{equation*}
    	\nabla\varphi = \frac{\nabla(a^\dag-P_{\!\!\mathcal{A}}(\widehat{a}_\theta^*))}{a^\dag}u^\dag+\frac{q^\dag-P_{\mathcal{A}}(\widehat{a}_\theta^*)}{a^\dag}\nabla u^\dag- \frac{a^\dag-P_{\!\!\mathcal{A}}(\widehat{a}_\theta^*)}{(a^\dag)^2}(\nabla a^\dag) u^\dag.
    \end{equation*}
    Using Assumption \ref{assum:MLSNN_reg}, the box constraint on $a^\dagger$ and $P_{\!\!\mathcal{A}}(\widehat{a}_\theta^*)$, and the stability estimate \eqref{eqn:hybrid_proj_stb} of the  operator $P_{\!\!\mathcal{A}}$, we arrive at
    \begin{align*}
        &\Big\|\frac{\nabla(a^\dag-P_{\!\!\mathcal{A}}(\widehat{a}_\theta^*))}{a^\dag}u^\dag\Big\|_{L^2(\Omega)}
        \leq c\big(1+\|\nabla \widehat{a}_\theta^*\|_{L^2(\Omega)}\big)
        \leq  c(1+e_a+\gamma_a^{-1}\eta)^{\frac{1}{2}},\\
        &\Big\|\frac{a^\dag-P_{\!\!\mathcal{A}}(\widehat{a}_\theta^*)}{a^\dag}\nabla u^\dag\Big\|_{L^2(\Omega)}+\Big\|\frac{a^\dag-P_{\!\!\mathcal{A}}(\widehat{a}_\theta^*)}{(a^\dag)^2}(\nabla a^\dag) u^\dag\Big\|_{L^2(\Omega)}\leq c .
    \end{align*}
    This implies $\varphi\in H_0^1(\Omega)$ with the following \textit{a priori} bounds
    \begin{equation*}
    	\|\varphi\|_{L^2(\Omega)} \leq c \quad \text{and}\quad  \|\nabla\varphi\|_{L^2(\Omega)}\leq c(1+e_a+\gamma_a^{-1}\eta)^{\frac{1}{2}}.
    \end{equation*}
    By the estimate \eqref{eqn:MLSNN_priori_err} and the Cauchy-Schwarz inequality, we have
    \begin{align*}
        |\mathrm{I}|&\leq  \|\nabla\cdot(\sigma^\dagger-\widehat{\sigma}_\kappa^*)\|_{L^2(\Omega)}\|\varphi\|_{L^2(\Omega)}
        \leq c(e_\sigma+\gamma_\sigma^{-1}\eta)^{\frac{1}{2}}.
    \end{align*}
    Similarly, we deduce
    \begin{equation*}
        \begin{aligned}
    		|\mathrm{III}|&\leq  \|\widehat{\sigma}_{\kappa}^*-P_{\!\!\mathcal{A}}(\widehat{a}_\theta^*)\nabla z^\delta\|_{L^2(\Omega)}\|\nabla\varphi\|_{L^2(\Omega)}\\
            &\leq  c (1+e_a+\gamma_a^{-1}\eta)^{\frac{1}{2}}(e_d+\eta)^{\frac{1}{2}}.
        \end{aligned}
    \end{equation*}
    Meanwhile, by the Cauchy--Schwarz inequality,
    \begin{align*}
     |\mathrm{II}|
        \leq &c\|P_{\!\!\mathcal{A}}(\widehat{a}_\theta^*)\|_{L^{\infty}(\Omega)}\|\nabla(z^\delta-u^\dagger)\|_{L^2(\Omega)}\|\nabla\varphi\|_{L^2(\Omega)} \\
        \leq &c(1+e_a+\gamma_a^{-1}\eta)^{\frac{1}{2}}\delta.
    \end{align*}
    Upon repeating the argument in Theorem \ref{thm:Bonito_stab}, we obtain
    \begin{equation*}
    	\big((a^\dag-P_{\!\!\mathcal{A}}(\widehat{a}_\theta^*))\nabla u^\dag,\nabla\varphi\big) = \frac12\int_{\Omega}\Big(\frac{a^\dag-P_{\!\!\mathcal{A}}(\widehat{a}_\theta^*)}{a^\dag}\Big)^2\big(a^\dag|\nabla u^\dagger|^2+fu^\dagger\big)\ \mathrm{d}x.
    \end{equation*}
This yields the desired weighted $L^2(\Omega)$ estimate. The proof of the second assertion is identical with that of Theorems \ref{thm:Bonito_stab} and \ref{thm:FEM_error}.
 \end{proof}

\subsection{PINN formulation}\label{subsec:NN_PINN}
In recent years, physics informed neural networks (PINNs) have garnered much attention for solving PDE forward and inverse problem. The idea behind PINNs is very simple: since DNNs are universal approximators in many function spaces, based on the principle of PDE residual minimization, one can naturally represent the strong form of PDE residual with the ansatz space of neural networks and minimize the residual via (stochastic) gradient descent to obtain a neural network that approximates the PDE solution. This framework was already considered in the 1990s in several works \cite{Dissanayake:1994,lagaris:1998artificial,Lagaris:2000}. The modern version of PINNs was introduced and named as such in \cite{Raissi:2019}. Since then there has been an explosive growth in the literature on PINNs. We discuss the PINN formulation for diffusion coefficient identification \eqref{eqn:inv_cond_gov} and several other related problems, including inverse source problems, unique continuation and elliptic optimal control.

\subsubsection{Diffusion coefficient identification}
First we describe the PINN formulation for diffusion coefficient identification. In the discussion, the accuracy of the noisy data  $z^\delta$ is given by the noise level:
\begin{equation}\label{eqn:PINN_noise}
    \delta:=\|u(a^\dag)-z^\delta\|_{H^1(\Omega)}.
\end{equation}
We employ two DNNs $a_\theta $ and $u_\kappa $ to represent the diffusion coefficient $a$ and the state $u$, respectively. Based on the stability result in Theorem \ref{thm:Bonito_stab}, we formulate the following population PINN loss
\begin{align}\label{eqn:inv_cond_loss_PINN}
     J_{\bm\gamma}(\theta,\kappa)=  \mathcal{E}_u+\gamma_d\mathcal{E}_d +\gamma_b\mathcal{E}_b +\gamma_a\mathcal{E}_a,
\end{align}
with the tuning parameters $\bm\gamma=(\gamma_d,\gamma_b,\gamma_a)\in \mathbb{R}^3_+$. The fitting term $\mathcal{E}_u$ to the observational data $z^\delta$, the PDE residual $\mathcal{E}_d$ and the boundary condition residual $\mathcal{E}_b$ are given  by
\begin{align*}
\mathcal{E}_u &= \| u_\kappa-z^\delta\|_{H^1(\Omega)}^2,\quad
\mathcal{E}_d = \|\nabla\cdot (P_{\!\!\mathcal{A}}(a_\theta) \nabla u_\kappa)+f\|^2_{L^2(\Omega)},\\
\mathcal{E}_b &= \|u_\kappa\|_{ H^1(\partial\Omega)}^2,\quad ~~~~~~ \mathcal{E}_a = \|\nabla P_{\!\!\mathcal{A}}(a_\theta)\|^2_{L^2(\Omega)}.
\end{align*}
In the loss, we impose the Dirichlet boundary condition in $H^1(\partial\Omega)$ instead of $L^2(\partial\Omega)$. This is to ensure that we can apply the standard elliptic regularity estimate.
When compared with the mixed least-squares approach, the
direct application of the PINN scheme to the governing PDE imposes higher regularity on the DNN function $u_\kappa$.
Let $X=(X_j)_{j=1}^{n_d}$ and $Y= (Y_j)_{j=1}^{n_b}$ be  i.i.d. sample points from $\mathcal{U}(\Omega)$ and $\mathcal{U}(\partial \Omega)$, respectively. Then we define Monte Carlo approximations of  $\mathcal{E}_i$, $i\in\{u,d,b,a\}$, by
\begin{align*}
    \widehat{\mathcal{E}}_u & =n_d^{-1}|\Omega| \sum_{j=1}^{n_d}\big[(u_\kappa(X_j)-z^\delta(X_j))^2 + \|\nabla (u_\kappa(X_j)-z^\delta(X_j))\|_{\ell^2}^2\big], \\
    \widehat{\mathcal{E}}_d & = n_d^{-1}|\Omega| \sum_{j=1}^{n_d} (\nabla\cdot (P_{\!\!\mathcal{A}}(a_\theta) \nabla u_\kappa)(X_j)+f(X_j))^2, \\
    \widehat{\mathcal{E}}_b & = n_b^{-1}|\partial\Omega| \sum_{j=1}^{n_b} \big[|u_\kappa(Y_j)|^2 + |\nabla u_\kappa(Y_j)|^2\big], \\
    \widehat{\mathcal{E}}_a & = n_d^{-1}|\Omega|\sum_{j=1}^{n_d}\|\nabla P_{\!\!\mathcal{A}}(a_\theta)(X_j)\|_{\ell^2}^2.
\end{align*}
Then the empirical loss $\widehat{J}_{\bm\gamma}(\theta, \kappa)$ reads
\begin{equation*}
    \widehat{J}_{\bm\gamma}(\theta,\kappa) = \widehat{\mathcal{E}}_u + \gamma_d\widehat{\mathcal{E}}_d + \gamma_b\widehat{\mathcal{E}}_b + \gamma_a\widehat{\mathcal{E}}_a.
\end{equation*}
Now we investigate the error $\|P_{\!\!\mathcal{A}}(\widehat{a}_{\theta}^\ast) -a^\dagger\|_{L^2(\Omega)}$ of the reconstruction $P_{\!\!\mathcal{A}}(\widehat{a}_{\theta}^\ast)$. To this end, we split the state approximation $\widehat{u}_{\kappa}^\ast$ into two parts:  $\widehat{u}_{\kappa}^\ast=\widehat{u}_{\kappa,1}^\ast+\widehat{u}_{\kappa,2}^\ast$, with the functions $\widehat{u}_{\kappa,i}^\ast$, $i=1,2$, satisfying
\begin{align*}
       & \left\{
            \begin{aligned}
                -\nabla \!\cdot\!( P_{\!\!\mathcal{A}}(\widehat{a}_{\theta}^\ast )\nabla \widehat{u}_{\kappa,1}^\ast)&=f ,\quad \mbox{in }\Omega ,\\
                \widehat{u}_{\kappa,1}^\ast&=0,\quad \mbox{on }\partial\Omega,
            \end{aligned}
        \right.\\
        &\left\{
            \begin{aligned}
                -\nabla \!\cdot\!( P_{\!\!\mathcal{A}}(\widehat{a}_{\theta}^\ast) \nabla \widehat{u}_{\kappa,2}^\ast)&=-\nabla \!\cdot\!(P_{\!\!\mathcal{A}} (\widehat{a}_{\theta}^\ast) \nabla \widehat{u}_{\kappa}^\ast) - f,\quad \mbox{in }\Omega,\\
            \widehat{u}_{\kappa,2}^\ast&=\widehat{u}_{\kappa}^\ast,\quad \mbox{on }\partial\Omega.
            \end{aligned}
        \right.
\end{align*}
By Theorem \ref{thm:Bonito_stab} (for $u^\dag=u(a^\dag)$ and $\widehat{u}_{\kappa,1}^\ast$), we deduce
\begin{align}\label{eqn:inv_cond_err_PINN}
     & \int_{\Omega}\Big|\frac{P_{\!\!\mathcal{A}}(\widehat{a}_{\theta}^\ast)-a^\dag}{a^\dag}\Big|^2\Big(fu^\dag \! + \!a^\dag|\nabla u^\dag|^2\Big)\ {\rm d}x  \\ \leq& \!  c\max\Big(\|\nabla a^\dag\|_{L^2(\Omega)}\!,\! \|\nabla P_{\!\!\mathcal{A}}(\widehat{a}_{\theta}^\ast)\|_{L^2(\Omega)}\Big)\|  \widehat{u}_{\kappa,1}^\ast\!-\!u^\dag \|_{H^1(\Omega)}^\frac{1}{2}.\nonumber
\end{align}
In view of the identity $ \widehat{u}_{\kappa,1}^\ast= (\widehat{u}_{\kappa}^\ast- z^\delta) + (z^\delta-\widehat{u}_{\kappa,2}^\ast)$ and the triangle inequality, we deduce
\begin{align*}
 \int_{\Omega}\Big|\frac{P_{\!\!\mathcal{A}}(\widehat{a}_{\theta}^\ast)-a^\dag}{a^\dag}&\Big|^2\Big(fu^\dag \! + \!a^\dag|\nabla u^\dag|^2\Big){\rm d}x  \leq  c\max\Big(\|\nabla a^\dag\|_{L^2(\Omega)}\!,\! \|\nabla P_{\!\!\mathcal{A}}(\widehat{a}_{\theta}^\ast)\|_{L^2(\Omega)})\Big) \\& \quad \times\left(\|\widehat{u}_{\kappa,2}^\ast \|_{H^1(\Omega)}  \!
 +\! \|\widehat{u}_{\kappa}^\ast\!- \!z^\delta \|_{H^1(\Omega)} \!+\! \|  z^\delta-u^\dag \|_{H^1(\Omega)}  \right)^\frac{1}{2}.
\end{align*}
Meanwhile, the standard elliptic regularity theory \cite{Gilbarg:1977} implies that there exists $c>0$ depending on the lower and upper bounds on $P_{\!\!\mathcal{A}}(\widehat{a}_\theta^\ast)$ such that
\begin{equation}\label{eqn:inv_cond_err_PINNres}
    \|\widehat{u}_{\kappa,2}^\ast\|_{H^1(\Omega)}\le c\Big(\|\nabla \!\cdot\!( P_{\!\!\mathcal{A}}(\widehat{a}_{\theta}^\ast) \nabla \widehat{u}_{\kappa}^\ast) + f\|_{H^{-1}(\Omega)}+\|\widehat{u}_{\kappa}^\ast\|_{H^{\frac{1}{2}}(\partial\Omega)} \Big).
\end{equation}
This implies the estimate below
\begin{align*}  &\|\widehat{u}_{\kappa,2}^\ast\|_{H^1(\Omega)}  \!
 +\! \|\widehat{u}_{\kappa}^\ast\!- \!z^\delta \|_{H^1(\Omega)} \leq cJ_{\bm\gamma}(\widehat{\theta}^*,\widehat{\kappa}^*)\\ = &c\Big([J_{\bm\gamma}(\widehat{\theta}^*,\widehat{\kappa}^*) - \widehat{J}_{\bm\gamma}(\widehat{\theta}^*,\widehat{\kappa}^*)] + \widehat{J}_{\bm\gamma}(\widehat{\theta}^*,\widehat{\kappa}^*)\Big) &  \\ \leq & c\Big(\sup_{(\theta,\kappa)\in\mathfrak{P}_{\infty,\epsilon_a}\times \mathfrak{P}_{2,\epsilon_u}}|J_{\bm\gamma}(\theta,\kappa) - \widehat{J}_{\bm\gamma}(\theta,\kappa)| + \widehat{J}_{\bm\gamma}(\widehat{\theta}^*,\widehat{\kappa}^*)\Big).
\end{align*}
Similarly, the following estimate holds for the $L^2(\Omega)$ bound of $\nabla P_{\!\!\mathcal{A}}(\widehat{a}_{\theta}^\ast)$
\begin{align*}
&\|\nabla P_{\!\!\mathcal{A}}(\widehat{a}_{\theta}^\ast)\|_{L^2(\Omega)} \leq \gamma_a^{-1}J_{\bm\gamma}(\widehat{\theta}^*,\widehat{\kappa}^*)\\ \leq &\gamma_a^{-1}\sup_{(\theta,\kappa)\in\mathfrak{P}_{\infty,\epsilon_a}\times \mathfrak{P}_{2,\epsilon_u}}|J_{\bm\gamma}(\theta,\kappa) - \widehat{J}_{\bm\gamma}(\theta,\kappa)| + \gamma_a^{-1}\widehat{J}_{\bm\gamma}(\widehat{\theta}^*,\widehat{\kappa}^*).
\end{align*}
By combining the last two estimates, we arrive at
\begin{align*}
&\int_{\Omega}\Big|\frac{P_{\!\!\mathcal{A}}(\widehat{a}_{\theta}^\ast)-a^\dag}{a^\dag}\Big|^2\Big(fu^\dag \! + \!a^\dag|\nabla u^\dag|^2\Big){\rm d}x
\\ \leq &  c\Big(1+\gamma_a^{-1}\sup_{(\theta,\kappa)\in\mathfrak{P}_{\infty,\epsilon_a}\times \mathfrak{P}_{2,\epsilon_u}}|J_{\bm\gamma}(\theta,\kappa) - \widehat{J}_{\bm\gamma}(\theta,\kappa)| + \gamma_a^{-1}\widehat{J}_{\bm\gamma}(\widehat{\theta}^*,\widehat{\kappa}^*)\Big)\\&\times
    \Big( \sup_{(\theta,\kappa)\in\mathfrak{P}_{\infty,\epsilon_a}\times \mathfrak{P}_{2,\epsilon_u}}|J_{\bm\gamma}(\theta,\kappa) - \widehat{J}_{\bm\gamma}(\theta,\kappa)| + \widehat{J}_{\bm\gamma}(\widehat{\theta}^*,\widehat{\kappa}^*) + \|  z^\delta-u^\dag \|^2_{H^1(\Omega)} \Big)^\frac{1}{2}.
\end{align*}
By the definition of the noise level $\delta:=\|z^\delta - u^\dag\|_{H^1(\Omega)}$, it suffices to derive an \textit{a priori} bound on $\widehat{J}_{\bm\gamma}(\widehat\theta^*, \widehat\kappa^*)$. This can be achieved by decomposing the error into the statistical error and approximation error: Let $(\theta^*,\kappa^*) $ be a minimizer of the loss $J_{\bm\gamma}$. Then there holds
\begin{align*}
\widehat{J}_{\bm\gamma}(\widehat\theta^*, \widehat\kappa^*)\leq \widehat{J}_{\bm\gamma}(\theta^*,\kappa^*)\le \sup_{(\theta,\kappa)\in\mathfrak{P}_{\infty,\epsilon_a}\times \mathfrak{P}_{2,\epsilon_u}}|\widehat{J}_{\bm\gamma}(\theta,\kappa)-J_{\bm\gamma}(\theta,\kappa)|+J_{\bm\gamma}(\theta^*,\kappa^*).
\end{align*}

Next we provide an \textit{a priori} bound on the approximation error under suitable regularity assumptions on the problem data.
\begin{theorem}
If $a^\dag\in W^{2,p}(\Omega)$, $p>\max(2,d)$ and $f\in H^1(\Omega)\cap L^\infty(\Omega)$, then there holds
    \begin{equation*}
        J_{\bm\gamma}(\theta^*,\kappa^*)\le c(\epsilon^2+\delta^2).
    \end{equation*}
\end{theorem}
\begin{proof}
By the standard elliptic regularity theory, we have $u^\dag:=u(a^\dag)\in H^3(\Omega)\cap  W^{2,p}(\Omega)\cap H_0^1(\Omega)$. By Lemma \ref{lem:NN_approx_tanh}, there exists $(\theta_\epsilon,\kappa_\epsilon)$ such that
    \begin{equation*}
        \|a^\dag-a_{\theta_\epsilon}\|_{W^{1,p}(\Omega)}\leq\epsilon \quad \mbox{and}\quad  \|u^\dag-u_{\kappa_\epsilon}\|_{H^2(\Omega)}\leq \epsilon .
    \end{equation*}
By the minimizing property of $(\theta^*,\kappa^*)$ and the triangle inequality, we have
\begin{align*}
& J_{\bm\gamma}(\theta^*,\kappa^*)\leq J_{\bm\gamma}(\theta_\epsilon,\kappa_\epsilon) \\
    	=& \| u_{\kappa_\epsilon}-z^\delta\|_{H^1(\Omega)}^2 +\|\nabla\cdot(P_{\!\!\mathcal{A}}(a_{\theta_\epsilon})\nabla u_{\kappa_\epsilon})+f \|_{L^2(\Omega)}^2+\| u_{\kappa_\epsilon}\|^2_{H^1(\partial\Omega)}  \\
        \le& c\Big( \| u_{\kappa_\epsilon}-u^\dag\|_{H^1(\Omega)}^2 +\| u^\dag-z^\delta\|_{H^1(\Omega)}^2
        +\|\nabla P_{\!\!\mathcal{A}}(a_{\theta_\epsilon})  \cdot \nabla u_{\kappa_\epsilon}   - \nabla a^\dag  \cdot \nabla u^\dag \|_{L^2(\Omega)}^2 \\& +\| P_{\!\!\mathcal{A}}(a_{\theta_\epsilon})\Delta u_{\kappa_\epsilon} - a^\dag  \Delta u^\dag \|_{L^2(\Omega)}^2+\|u_{\kappa_\epsilon}-u^\dag \|^2_{L^2(\partial\Omega)}\Big).
    \end{align*}
    By the triangle inequality, we have
    \begin{align*}
       & \|\nabla P_{\!\!\mathcal{A}}(a_{\theta_\epsilon})  \cdot \nabla u_{\kappa_\epsilon}   - \nabla a^\dag  \cdot \nabla u^\dag \|_{L^2(\Omega)}\\
        \le & \|\nabla (P_{\!\!\mathcal{A}}(a_{\theta_\epsilon}) -a^\dag)   \|_{L^2(\Omega)}\|\nabla u^\dag\|_{L^\infty(\Omega)}\\
        &+\|\nabla a^\dag   \|_{L^\infty(\Omega)}\|\nabla (u_{\kappa_\epsilon}- u^\dag)\|_{L^2(\Omega)}\le c \epsilon.
    \end{align*}
    Likewise we have
    \begin{align*}
        \|P_{\!\!\mathcal{A}}(a_{\theta_\epsilon})\Delta u_{\kappa_\epsilon} - a^\dag  \Delta u^\dag  \|_{L^2(\Omega)}
        \le & \|P_{\!\!\mathcal{A}}(a_{\theta_\epsilon})    \|_{L^\infty(\Omega)}\|\Delta (u_{\kappa_\epsilon}-u^\dag)\|_{L^2(\Omega)}\\
        &+\|  a^\dag-a_{\theta_\epsilon}   \|_{L^\infty(\Omega)}\|\nabla  u^\dag\|_{L^2(\Omega)}\le c \epsilon.
    \end{align*}
    By the trace theorem, we obtain $J_{\bm\gamma}(\theta^*,\kappa^*)\le c(\epsilon^2+\delta^2)$.
\end{proof}

\begin{remark}\label{rmk:PINN_const}
The estimate \eqref{eqn:inv_cond_err_PINN} is direct from Theorem \ref{thm:Bonito_stab}, where $c>0$ depends on $\Omega$, $d$, $\underline{c}_a$, $\overline{c}_a$, $\|f\|_{L^\infty(\Omega)}$, $\|\nabla a^\dag\|_{L^2(\Omega)}$ and $\|\nabla P_{\!\!\mathcal{A}}(\widehat{a}_{\theta}^\ast)\|_{L^2(\Omega)}$; see Remark \ref{rmk:Bonito_const_dep}. 
The estimate \eqref{eqn:inv_cond_err_PINNres} is direct from the elliptic regularity. The equation for $\widehat{u}_{\kappa,2}^\ast$ involves the recovered diffusion coefficient $\widehat{a}_{\theta}^\ast$. Thus, $c$ in \eqref{eqn:inv_cond_err_PINNres} depends on the lower and upper bounds of $\widehat{a}_{\theta}^\ast$.
\end{remark}
\begin{remark}\label{rmk:PINN_comments}
    Compared with the mixed least-squares DNN approach in Section \ref{ssec:MLSNN}, the PINN formulation allows deriving an error bound from the conditional stability result directly. 
    However, the PINN formulation involves the second-order derivatives of $u_\kappa$, resulting in stronger regularity on the DNN architecture.
    Numerically, the training of PINN often takes longer than the mixed least-squares DNN, due to the computation of second-order derivatives.
\end{remark}

\subsubsection{Inverse source problem}

Consider the following parabolic equation
\begin{equation}\label{eqn:ISP_eq}
	\left\{\begin{aligned}
		 \partial_t u - \Delta u &= f(x)g(t), \ &\mbox{in}&\ \Omega\times(0,T), \\
		u&=0, \ &\mbox{on}&\ \partial\Omega\times(0,T), \\
		u(0)&=0, \ &\mbox{in}&\ \Omega.
	\end{aligned}
	\right.
\end{equation}
The inverse problem aims to recover the spatial component $f(x)$ of the source $f(x)g(t)$ from the terminal measurement $u(x,T)$, given the temporal component $g(t)$. The following stability result holds \cite{Isakov:1991,PrilepkoKostin:1992,PrilepkoVasin:2000}. Note that the stability result in Theorem \ref{thm:inv_source_stab} only requires positive lower and upper bounds on $g$.  It is unconditional in the sense of Definition \ref{def:cond_stab}, since no additional condition on $f$ is imposed.

\begin{theorem}\label{thm:inv_source_stab}
Let the function $g\in L^\infty(0,T)$ with $0<\underline{c}_g\le g\le \bar{c}_g$. Then there exists a unique solution $f\in L^2(\Omega)$ and $c=c(\underline{c}_g,\Omega,T)>0$ such that
\begin{equation*}
    \|f\|_{L^2(\Omega)}\le c\|\Delta u(\cdot,T)\|_{L^2(\Omega)}.
\end{equation*}
\end{theorem}

Inspired by the condition stability in Theorem \ref{thm:inv_source_stab}, Liu et al. \cite{ZhangLiLiu:2023} used PINN to solve the inverse source problem with the measured data $z^\delta$ in $H^2(\Omega)$, with a noise level $\delta = \|z^\delta - u(f^\dag)(T)\|_{H^2(\Omega)}$. The approach utilizes two separate DNNs to discretize the unknown source $f(x)$ and the state $u(x,t)$. The residuals of the underlying PDE system are defined, respectively, by
\begin{align*}
    &\mathcal{E}_d=\|\partial_t u_\kappa-\Delta u_\kappa-f_\theta  g\|_{H^1(0,T;L^2(\Omega))}^2, \quad  \mathcal{E}_i=  \|u_\kappa(0)\|_{H^2(\Omega)}^2,\\
    &\mathcal{E}_u=\|\Delta(u_\kappa(T)-z^\delta)\|_{L^2(\Omega)}^2,  \quad \mathcal{E}_b= \|u_\kappa\|_{H^2(0,T;L^2(\partial\Omega))}^2 .
\end{align*}
Note the formulation requires the strong regularity of the data $z^\delta$.
The stability result in Theorem \ref{thm:inv_source_stab} motivates the PINN loss
\begin{align}\label{eqn:inv_source_loss}
     J_\gamma(\kappa,\theta)=& \mathcal{E}_u + \gamma_d\mathcal{E}_d + \gamma_b\mathcal{E}_b+ \gamma_i\mathcal{E}_i,
\end{align}
with $\gamma = (\gamma_d,\gamma_b,\gamma_i)\in\mathbb{R}_+^3$ being the vector of penalty strengths. The loss $J_\gamma(\kappa,\theta)$ involves the derivatives of the residuals, which requires suitable regularity of the data $z^\delta$ and the state approximation  $u_\kappa$.  Like before, we denote by $\widehat{J}_\gamma$ the empirical loss, and by $(\widehat{\kappa}^\ast,\widehat{\theta}^\ast)$ a minimizer of $\widehat{J}_\gamma$ and $(\widehat{u}_{\kappa}^\ast$, $\widehat{f}_{\theta}^\ast)$ the DNN realizations. Let $\widetilde{u}=\widehat{u}_{\kappa}^\ast-u^\dagger$ be the state error. Then $\widetilde{u}$ satisfies
\begin{equation*}
        \left\{
            \begin{aligned}
                \partial_t \widetilde{u}-\Delta \widetilde{u}&=\partial_t \widehat{u}_{\kappa}^\ast - \Delta \widehat{u}_{\kappa}^\ast-\widehat{f}_{\theta}^\ast g  +(\widehat{f}_{\theta}^\ast-f^\dagger)g ,\quad &&\mbox{in }\Omega\times (0,T],\\
                \widetilde{u}&=\widehat{u}_{\kappa}^\ast,\quad &&\mbox{on }\partial\Omega\times (0,T],\\
                \widetilde{u}(0)&=\widehat{u}_{\kappa}^\ast(0),\quad &&\mbox{in }\Omega.
            \end{aligned}
        \right.
\end{equation*}
Now we decompose $\widetilde{u}$ into $\widetilde{u}=u_1+u_2$, with $u_1$ and $u_2$ respectively satisfying
\begin{align*}
        &\left\{
            \begin{aligned}
                \partial_t u_1-\Delta u_1&=(\widehat{f}_{\theta}^\ast-f^\dagger)g,\quad &&\mbox{ in }\Omega\times (0,T],\\
                u_1&=0,\quad &&\mbox{ on }\partial\Omega\times (0,T],\\
                u_1(0)&=0,\quad &&\mbox{ in }\Omega,\\
            \end{aligned}
        \right.\\
&\left\{\begin{aligned}        \partial_t u_2-\Delta u_2&=\partial_t \widehat{u}_{\kappa}^\ast - \Delta \widehat{u}_{\kappa}^\ast-\widehat{f}_{\theta}^\ast g,\quad &&\mbox{ in }\Omega\times (0,T],\\
                u_2&=\widehat{u}_{\kappa}^\ast,\quad &&\mbox{ on }\partial\Omega\times (0,T],\\
                u_2(0)&=\widehat{u}_{\kappa}^\ast(0),\quad &&\mbox{ in }\Omega.
            \end{aligned}
        \right.
\end{align*}
By Theorem \ref{thm:inv_source_stab}, we have
\begin{align*}
    \|\widehat{f}_{\theta}^\ast-f^\dagger \|_{L^2(\Omega)}\le  c\|\Delta u_1(T)\|_{L^2(\Omega)} .
\end{align*}
The splitting $u_1(T)=(\widehat{u}_\kappa^*-z^\delta)+(z^\delta-u^\dagger(T))+(-u_2(T))$ and the triangle inequality imply
\begin{align} \label{eqn:inv_source_err}
   \|\Delta u_1(T)\|_{L^2(\Omega)}
    \le&  \|\Delta (\widehat{u}_\kappa^*-z^\delta) \|_{L^2(\Omega)}+\|\Delta (z^\delta-u^\dagger(T)) \|_{L^2(\Omega)}\\
    &+\| \Delta u_2(T) \|_{L^2(\Omega)}.\nonumber
\end{align}
The three terms in the estimate \eqref{eqn:inv_source_err} denote the data residual, the data discrepancy, and {\textit{a priori}} bound on $u_2$ (which is determined by the problem data for $u_2$, i.e., the residuals $\partial_t \widehat{u}_{\kappa}^\ast - \Delta \widehat{u}_{\kappa}^\ast-\widehat{f}_{\theta}^\ast g$, $\widehat{u}_{\kappa}^\ast|_{\partial\Omega\times(0,T]}$ and $\widehat{u}_{\kappa}^\ast(0)$). In turn, these factors can be bounded by the training errors and quadrature errors. In this way, we obtain an error bound on the DNN approximation $\widehat{f}_{\theta}^\ast$. Due to the unconditional stability in Theorem \ref{thm:inv_source_stab},  the constant $c$ in \eqref{eqn:inv_source_err} is independent of the DNNs. In the analysis, the stability estimate in Theorem \ref{thm:inv_source_stab} plays a vital role.

Zhang and Liu \cite{ZhangLiu:2023} extend this approach to the inverse source problem for an elliptic system with incomplete internal measurement data. The authors derive the error bound in terms of the number of sampling points, noise level and regularization parameters, if the training error is small and the unknown source is analytic. See also \cite{Zhang:2024source} for an inverse dynamic source problem that arises in time-domain fluorescence diffuse optical tomography and \cite{Zhang:2023potential,Cao:2025Potential} for inverse potential problems in the parabolic equation.

\subsubsection{Point source identification}
Hu et al \cite{hu2024pointsourceidentificationusing} developed a PINN scheme for point source identification for the Poisson equation. The problem formulation is as follows. Consider the following Poisson equation
\begin{align}\label{eqn:Poisson}
    \left\{\begin{aligned}
    -\Delta u  &= F, \quad \mbox{in }\Omega, \\
    u & = f, \quad \mbox{on }\partial\Omega,
  \end{aligned}\right.
\end{align}
with the source $F$ given by
\begin{align*}
F = \sum_{j=1}^Mc_j\delta_{x_j},  \quad\mbox{with}\ c_j\in\mathbb{R},\ x_j\in\Omega,\ j = 1,\dots,M
\end{align*}
Given the number $M$ of point sources, the inverse problem is to determine the unknown intensity $\mathbf c = (c_j)_{j=1}^M$ and locations $\mathbf x = (x_j)_{j=1}^M\subset\Omega$ from the given Cauchy data $(f,g):=(u|_{\partial\Omega},\partial_\nu u|_{\partial\Omega})$. By Holmgren's uniqueness theorem, the Cauchy data uniquely determines the number, locations and strengths of the point sources. To state the stability result, we need several notation. Let $\alpha=\min_{j=1,\ldots,M}\mathrm{dist}(x_j,\partial\Omega)$, $\Omega_\alpha = \{x\in \Omega: \mathrm{dist}(x,\partial\Omega)\geq \alpha\}$, $\beta= \mathrm{diam}(\Omega)-\alpha,$ with $\mathrm{diam}(\Omega)$ being the diameter of $\Omega$. For any point configuration $\mathbf X=(\mathbf{x}_1, \cdots, \mathbf{x}_M)$, with $\mathbf{x}_j=(x_{1,j},x_{2,j},\cdots,x_{d,j})^T$, we denote by $P_{k}(\mathbf{x}_j)=x_{k,j}+{\rm i}x_{k+1,j}$, $k=1, \ldots, d-1$, its projection onto the $x_kOx_{k+1}$-complex plane and let $\mathbf{P}_{k}(\mathbf X)=(P_{k}(\mathbf{x}_1),P_{k}(\mathbf{x}_2),\cdots,P_{k}(\mathbf{x}_M))$. Then we define the separability coefficient $\rho_k(\mathbf X)$ of the projected sources $\mathbf{P}_k(\mathbf X)$ and the minimal distance $\rho(\mathbf X)$ by
\begin{equation}\label{def:sepa}
    \rho_k(\mathbf X)\!=\!\min_{1\leq i<j\leq M}\!{\rm dist}(P_{k}(\mathbf{x}_i),P_{k}(\mathbf{x}_j))\ \mbox{and}\ \rho(\mathbf X)\!=\!\min_{1\leq k<d}\rho_k(\mathbf X).
\end{equation}
For two point configurations $\widehat{\mathbf X}^*=(\widehat{\mathbf{x}}_1^*, \widehat{\mathbf{x}}_2^*,\cdots, \widehat{\mathbf{x}}_{M}^*)$ and $\mathbf X^*=\left(\mathbf{x}_1^*, \mathbf{x}_2^*,\cdots, \mathbf{x}_{M}^*\right)$, we make the following separability assumption in the analysis below.
\begin{assumption}\label{assump:separation}
There exists $\rho_*>0$ such that $\min\{\rho(\widehat{\mathbf X}^*), \rho(\mathbf X^*)\}\geq\rho_*$.
\end{assumption}

Moreover, under suitable conditions, we have the following stability result in the three-dimensional case (i.e., $d=3$) (see \cite[Theorem 3.2]{ElBadiaElHajj:2013} or \cite{ElBadiaElHajj:2012}). The stability is proved using specialized test functions.
\begin{theorem}\label{thm:stability-point-source}
Let $u^\ell$ for $\ell=1,2$, be the solutions of problem \eqref{eqn:Poisson} for the sources $F^\ell = \sum_{j=1}^M c_j^\ell \delta_{x_j,\ell}$, characterized by the configurations $(c_j^\ell,x_j^\ell)_{j=1}^M$, with the two point configurations $\{x_j^\ell\}_{j=1}^M$, $\ell=1,2$, satisfying Assumption \ref{assump:separation} and $c_j^\ell\neq 0$.
Then, there exists a permutation
$\pi$ of the set $\{1,\ldots,m\}$ such that the following estimate holds:
\begin{equation*}
    \max_{1\leq j\leq m}\|S_j^2-S^1_{\pi(j)}\|_2 \leq c\frac{\beta^2}{\rho}\left[\frac{\rho}{\beta}\|\partial_\nu u^2-\partial_\nu u^1\|_{L^2(\partial\Omega)} + \|u^2-u^1\|_{L^2(\partial\Omega)} \right]^{1/M}.
\end{equation*}
\end{theorem}

Numerically, one  challenge lies in the limited regularity of the state $u$ (due to the presence of point sources), which renders standard neural PDE solvers not directly applicable. To overcome the challenge, Hu et al. \cite{hu2024pointsourceidentificationusing} employed a singularity enrichment strategy \cite{HuJinZhou:2023,HuJinZhou:2024}. Let $\Phi(x)$ be the fundamental solution of the Laplacian. Then the state $u$ can be decomposed into singular and regular parts as
\begin{equation*}
    u=\sum_{j=1}^Mc_j\Phi_{x_j}+v : = w(\mathbf{c},\mathbf x) + v,
\end{equation*}
where the regular part $v$ satisfies
\begin{equation}\label{eqn:reg-part-equation}
    \left\{
    \begin{aligned}
    -\Delta v&=0,&& \mbox{in }\Omega,\\
    v&=f-w(\mathbf c, \mathbf x),&& \mbox{on }\partial\Omega,\\
    \partial_\nu v&=g-\partial_\nu w(\mathbf c, \mathbf x),&& \mbox{on }\partial\Omega.
    \end{aligned}
    \right.
\end{equation}
In practice, only noisy Cauchy data $(f^\delta,g^\delta) \approx (f,g)$ is available.
By minimizing the residuals and approximating the regular part $v$ using a DNN and discretizing the integrals using the Monte Carlo method, we obtain the following empirical loss
\begin{align}\label{eqn:loss-emp}
\widehat{J}(v_\theta,\mathbf{c},\mathbf x)=&\dfrac{|\Omega|}{n_r}\sum_{i=1}^{n_r}|\Delta v_\theta(X_{i})|^2\\& +\gamma_d\dfrac{|\partial\Omega|}{n_{b}}\sum_{i=1}^{n_b}\big(v_\theta(Y_{i})-f^\delta(Y_{i})+w(\mathbf{c},\mathbf x)(Y_{i})\big)^2\nonumber\\
&+\gamma_b\dfrac{|\partial\Omega|}{n_{b}}\sum_{i=1}^{n_b}\big(\partial_\nu v_\theta(Y_{i})-g^\delta(Y_{i})+\partial_\nu w(\mathbf{c},\mathbf x)(Y_{i})\big)^2, \nonumber
\end{align}
with $(X_i)_{i=1}^{n_r}\subset\Omega$ and $(Y_i)_{i=1}^{n_b}\subset\partial\Omega$ being i.i.d. sampling points.
The resulting training problem reads
\begin{equation}\label{eqn:min-emp}
    (\widehat{\theta}^*,\widehat{\mathbf{c}}^*,\widehat{\mathbf x}^*) = \mathop{{\rm argmin}}\limits_{\theta,\mathbf{c},X}\widehat{J}(v_\theta,\mathbf{c},\mathbf x),
\end{equation}
with $v_{\widehat{\theta}^*}$ being the DNN approximation of the regular part $v$.

First, we give some assumptions on the problem data below.
\begin{assumption}\label{hu2024:reg}
    $f\in H^{\frac52}(\partial\Omega)$, $g\in H^{\frac32}(\partial\Omega)$ and $f^\delta,g^\delta\in L^\infty(\partial\Omega)$ hold.
\end{assumption}

The next result gives an \textit{a priori} bound on the recovered locations, intensities, regular part and state \cite[Theorem 3.7]{hu2024pointsourceidentificationusing}. The proof relies crucially on a version of the conditional stability result similar to Theorem \ref{thm:stability-point-source}; see \cite[Theorem 3.1]{hu2024pointsourceidentificationusing} for the precise statement.
\begin{theorem}\label{thm:main-pointsource}
Let $p<\frac{d}{d-1}$, and Assumptions \ref{hu2024:reg} and \ref{assump:separation} hold.
For any small $\epsilon,\mu,\tau>0$, let the numbers $n_r$ and $n_b$ of sampling points satisfy
$n_r,n_b\sim\mathcal{O}(\epsilon^{-\frac{4}{1-\mu}}\log^{\frac{1}{1-\mu}}\tfrac{1}{\tau})$.
Then there exist a DNN function
$v_\theta \in \mathcal{N}(c, c\epsilon^{-\frac{d}{1-\mu}},c\epsilon^{-2-\frac{9d/2+4+2\mu}{1-\mu}})$,
and a permutation $\pi$ of the set $\{1,\ldots,M\}$ such that the following estimate holds with probability at least $1- 3 \tau$,
\begin{align*}	    	
&\max_{1\leq j\leq M}{\rm dist}(\widehat{\mathbf{x}}_{j}^*,\mathbf{x}_{\pi(j)}^*)+
\max_{1\leq j\leq M}\!\!|\widehat{c}_j^*-c_{\pi(j)}^*|\\&\quad +
\|v^*-v_{\widehat{\theta}^*}\|_{L^2(\Omega)}+\|u^*-u_{\widehat{\theta}^*}\|_{L^p(\Omega)}\leq c(\epsilon+\delta)^{\tfrac{1}{M}}.
\end{align*}
\end{theorem}

\subsubsection{Unique continuation problem}

Now consider the unique continuation problem for Poisson's equation: given an open and nonempty subdomain $\Omega'\subset \Omega$, the observational data $u_d\in L^2(\Omega')$ and the source $f\in L^2(\Omega)$, find $u\in H^1(\Omega)$ such that
\begin{equation}\label{eqn:FEM_ucp_obs}
    u=u_d\quad \mbox{in }\Omega',
\end{equation}
and $u\in H^1(\Omega)$ is a weak solution to
\begin{equation}\label{eqn:FEM_ucp_eq}
    -\Delta u=f\quad \mbox{in }\Omega.
\end{equation}
Note that we do not impose any boundary condition on the function $u$.  Hence the problem is ill-posed. The unique continuation problem can achieve H\"older or logarithm stability \cite{Alessandrini:2009}. We follow the description in \cite{BurmanOksanen:2018,Burman:2020convectionI}.
\begin{theorem}\cite[Theorem 7.1]{BurmanOksanen:2018}\label{thm:ucp_stab}
Let $f\in H^{-1}(\Omega)$ and $u\in H^1(\Omega)$ satisfy  \eqref{eqn:FEM_ucp_obs}-\eqref{eqn:FEM_ucp_eq} with $u_d\in L^2(\Omega^\prime)$. Then for any open simply connected set $E\subset \Omega$ such that $\dist(E,\partial \Omega)>0$, there holds
    \begin{equation*}
        \|u\|_{L^2(E)}\le c\|u\|_{L^2(\Omega)}^{1-\tau } \left(\|f\|_{H^{-1}(\Omega)} + \|u_d\|_{L^2(\Omega^\prime)}\right)^\tau,
    \end{equation*}
    where the constant $c>0$ and the exponent $\tau\in (0,1)$ depend on the set $E$. Moreover, the following global logarithmic stability holds
    \begin{equation*}
        \|u\|_{L^2(\Omega)}\le c\|u\|_{H^1(\Omega)}^{1-\tau } \left|\log\left(\|f\|_{H^{-1}(\Omega)} + \|u_d\|_{L^2(\Omega^\prime)}\right)\right|^\tau.
    \end{equation*}
\end{theorem}

\begin{remark}\label{rmk:UCP_Cauchy}
It can also be viewed as the Cauchy problem for Poisson's equation, i.e., find $u\in H^1(\Omega)$ such that
    \begin{equation*}
	\left\{
        \begin{aligned}
		  - \Delta u &= f, \ &\mbox{in}&\ \Omega , \\
		  u&=g_1, \ &\mbox{on}&\ \Gamma, \\
            \partial_{\nu} u&=g_2, \ &\mbox{on}&\ \Gamma, \\
	   \end{aligned}
	\right.
    \end{equation*}
    with the partial boundary $\Gamma\subset\partial\Omega$. The unique continuation problem has also been studied for heat equation \cite{Imanuilov:1995}, wave equation \cite{Burman:2020wave} and Stokes system \cite{Lin:2010} etc.  These data assimilation problems admit similar conditional stability results as Theorem \ref{thm:ucp_stab}, i.e.,  H\"older stability in the interior of the domain and logarithmic stability in the whole domain $\Omega$. See the review \cite{Alessandrini:2009} for various stability results for elliptic Cauchy problems.
\end{remark}
Mishra and Molinaro \cite{Mishra:2021inverse} proposed an abstract framework for PINNs. The study focuses on the unique continuation problem of Poisson's equation \eqref{eqn:FEM_ucp_obs}-\eqref{eqn:FEM_ucp_eq}, and approximates the state $u$ by a DNN $u_\theta$. They derived bounds on the generalization error using conditional stability directly.
 The stability result in Theorem \ref{thm:ucp_stab} motivates the following PINN loss
\begin{equation}\label{eqn:PINN_ucp_loss}
    J_\gamma(\theta)= \|u_d-u_{\theta}\|_{L^2(\Omega^\prime)}^2 + \gamma\|\Delta u_{\theta}+f\|_{L^2(\Omega)}^2,
\end{equation}
where $\gamma>0$ is the penalty parameter. In practice, one needs to approximate the integrals using quadrature rules (e.g., Monte Carlo method). Then the empirical loss $\widehat{J}_\gamma(\theta)$ is given by
\begin{equation}\label{eqn:PINN_ucp_loss_dis}
    \widehat{J}_\gamma(\theta)=\sum_{j=1}^{n_i} w_j^i|(u_d-u_{\theta})(X_j)|^2+ \gamma\sum_{j=1}^{n_d} w_j^d|(\Delta u_{\theta}+f)(Y_j)|^2,
\end{equation}
with $( X_j)_{j=1}^{n_i}\subset\Omega'$ and $( Y_j)_{j=1}^{n_d}\subset\Omega$ being quadrature points. The relevant quadrature weights $( w_j^i)_{j=1}^{n_i}$ and $(w_j^d)_{j=1}^{n_d}$ satisfy the following estimates
\begin{equation}\label{eqn:ucp_quad_err}
    \begin{aligned}
       \left|\sum_{j=1}^{n_i} w_j^i|(u_d-u_{\theta})(X_j)|^2- \|u_d-u_{\theta}\|_{L^2(\Omega')}^2\right|\le & cn_i^{-\frac{\alpha_i}{2}},\\
         \left|\sum_{j=1}^{n_d} w_j^d|(\Delta u_{\theta}+f)(Y_j)|^2- \|\Delta u_{\theta}+f\|_{L^2(\Omega)}^2\right|\le& cn_d^{-\frac{\alpha_d}{2}},
    \end{aligned}
\end{equation}
for some exponents $\alpha_i,\alpha_d>0$.
The constant $c$ in the estimate depends on the \textit{a priori} regularity of $u_\theta$, $u_d$ and $f$ (and thus on the DNN architecture). Let  $\widehat{u}^*_\theta$ be a minimizer to the functional $J_\gamma$, and $\widehat{u}^*_\theta$ its DNN realization.
We define the following training error:
\begin{equation*}
\begin{split}
    \widehat{\mathcal{E}}_i &= \left(\sum_{j=1}^{n_i} w_j^i|(u_d-\widehat{u}^*_\theta)(X_j)|^2\right)^{\frac{1}{2}} \quad\mbox{and}\quad    \widehat{\mathcal{E}}_d = \left(\sum_{j=1}^{n_d} w_j^d|(\Delta \widehat{u}^*_\theta+f)(Y_j)|^2\right)^{\frac{1}{2}}.
\end{split}
\end{equation*}
Now we derive an error bound using conditional stability in Theorem \ref{thm:ucp_stab}, which depends explicitly on the training error and the number of sampling points.
\begin{proposition}\label{prop:ucp_error}
Fix $f \in C(\overline{\Omega})$ and $u_d \in C^2(\overline{\Omega'})$. Let $u^\dagger \in H^1(\Omega)$ be the solution to problem \eqref{eqn:FEM_ucp_obs}-\eqref{eqn:FEM_ucp_eq}, and let $\widehat{\theta}^\ast$ be a minimizer of the loss $\widehat{J}_\gamma(\theta)$ defined in \eqref{eqn:PINN_ucp_loss_dis}, and $\widehat{u}_{\theta}^{\ast} \in C^2(\overline{\Omega})$ its DNN realization. Then for any $\Omega'\subset E \subset \Omega$, with $E$ being open, simply connected and such that $dist(E,\partial \Omega) > 0$, there holds
    \begin{equation}\label{eqn:ucp_error}
        \|u^\dagger-\widehat{u}_{\theta}^{\ast}\|_{L^2(E)} \leq c\Big(\widehat{\mathcal{E}}_i + \widehat{\mathcal{E}}_d + n_i^{-\frac{\alpha_i}{2}} + n_d^{-\frac{\alpha_d}{2}}\Big)^\tau,
    \end{equation}
    for some $\tau \in (0,1)$ and  $c=c(\|u^\dagger\|_{L^2(\Omega)}, \|\widehat{u}_{\theta}^{\ast}\|_{L^2(\Omega)})$.
\end{proposition}
\begin{proof}
Let $\widetilde{u}=u^\dagger-\widehat{u}_{\theta}^{\ast}\in H^1(\Omega)$. Then $\widetilde{u}$ satisfies
    \begin{equation*}
        \left\{
            \begin{aligned}
                -\Delta \widetilde{u}&=f+\Delta \widehat{u}_{\theta}^{\ast},\quad \mbox{ in }\Omega,\\
                  \widetilde{u}&=u_d - \widehat{u}_{\theta}^{\ast},\quad \mbox{ in }\Omega^\prime.
            \end{aligned}
        \right.
    \end{equation*}
    By Theorem \ref{thm:ucp_stab}, we obtain
    \begin{equation}\label{eqn:ucp_err_res}
        \|\widetilde{u}\|_{L^2(E)}\le c ( \|f+\Delta \widehat{u}_{\theta}^{\ast}\|_{L^2(\Omega)} + \|u_d - \widehat{u}_{\theta}^{\ast}\|_{L^2(\Omega^\prime)} )^\tau,
    \end{equation}
for some $\tau\in (0,1)$ and $c=c(\|\widetilde{u}\|_{L^2(\Omega)})$. By the estimates in \eqref{eqn:ucp_quad_err}, we obtain
\begin{align*}
\|f+\Delta \widehat{u}_{\theta}^{\ast}\|_{L^2(\Omega)}&\le \widehat{\mathcal{E}}_{d}+c n_d^{-\frac{\alpha_d}{2}}, \\ \|u_d - \widehat{u}_{\theta}^{\ast}\|_{L^2(\Omega^\prime)}&\le \widehat{\mathcal{E}}_{i}+cn_i^{-\frac{\alpha_i}{2}}.
\end{align*}
Combining these estimates with \eqref{eqn:ucp_err_res} yields the desired assertion.
\end{proof}

\subsubsection{Elliptic optimal control}
PDE constrained optimal control is a very important class of applied problems and has found many applications. This class of problems are formally similar to variational regularization for PDE inverse problems. Numerically solving the optimal control problems using neural networks have received much attention; see \cite{Borovykh:2022,Meng:2022,Nakamura：2021,Zhao:2024} for an incomplete list and the references therein. Lai et al. \cite{Lai:2025Hard} developed a PINN scheme to solve optimal control problems governed by PDEs with interfaces and control constraints.

Dai et al \cite{dai2024solvingellipticoptimalcontrol} developed a PINN scheme to solve elliptic optimal control problems. Let $\Omega \subset(-1,1)^d \subset \mathbb{R}^d$ $(d \geq 1)$ be a bounded domain with a smooth boundary $\partial \Omega$. Given the target function $u_d \in L^2(\Omega)$, consider the following distributed optimal control problem
\begin{equation}\label{eqn:OCP-obj}
\min J(u, z):=\frac{1}{2}\left\|u-u_d\right\|_{L^2(\Omega)}^2+\frac{\lambda}{2}\|z\|_{L^2(\Omega)}^2,
\end{equation}
where the state-control pair $(u, z)$ satisfies
\begin{equation}\label{eqn:OCP-gov}
    \left\{\begin{aligned}-\Delta u &= f+z,\quad \mbox{in}\ \Omega,\\
    u &= 0, \quad \mbox{on}\ \partial\Omega,
    \end{aligned}\right.
\end{equation}
with the given source $f \in L^2(\Omega)$, and $\lambda>0$. Problem \eqref{eqn:OCP-obj}-\eqref{eqn:OCP-gov} admits a unique solution $(\bar{u}, \bar{z})$. The objective is to construct a DNN approximation scheme for the minimizing pair $(\bar{u}, \bar{z})$ and to derive an \textit{a priori} $L^2(\Omega)$ error bound on the DNN approximation. Dai et al employed an optimize-then-discretize strategy. The coupled first-order optimality system of $(\bar{u}, \bar{p}, \bar{z})$ is given by \cite[Chapter 2]{troltzsch2010optimal}
\begin{equation}\label{eqn:OCP-opt-system}
    \left\{\begin{aligned}
    -\Delta u &= f+z,\quad \mbox{in}\ \Omega,\\
    u &= 0, \quad \mbox{on}\ \partial\Omega, \\
    -\Delta p &= u-u_d,\quad \mbox{in}\ \Omega,\\
    p &= 0, \quad \mbox{on}\ \partial\Omega,\\
    z &= -\lambda^{-1} p, \quad \mbox{in}\ \Omega.
    \end{aligned}\right.
\end{equation}
By eliminating the control $z$ in the first line of \eqref{eqn:OCP-opt-system} using the relation $z=-\lambda^{-1}p$ and applying the PINN strategy, we obtain the following population loss
\begin{equation*}
\begin{split}
 \mathcal{L}(u, p)&=\left\|\Delta u+f-\lambda^{-1} p\right\|_{L^2(\Omega)}^2+\gamma_i\left\|\Delta p+u-u_d\right\|_{L^2(\Omega)}^2\\&\quad +\gamma_b^u\|u\|_{L^2(\partial \Omega)}^2+\gamma_b^p\|p\|_{L^2(\partial \Omega)}^2,
\end{split}
\end{equation*}
where $\gamma=(\gamma_i, \gamma_b^u, \gamma_b^p)\in\mathbb{R}_+^3$ contains the penalty parameters to approximately enforce the PDE residual (of $p$) and zero Dirichlet boundary conditions of $u$ and $p$. The authors discretize the state-costate pair $(u,p)$ with two separate DNNs $(u_\theta,p_\kappa)\in\mathcal{Y}\times\mathcal{P}$ and  approximate the integrals using the Monte Carlo method.
Let $(X_i)_{i=1}^{n_d}$ and $(Y_i)_{i=1}^{n_b}$ be i.i.d. sampling points from $\mathcal{U}(\Omega)$ and $\mathcal{U}(\partial\Omega)$, respectively. Then the empirical loss $\widehat{\mathcal{L}}\left(u_\theta, p_\kappa\right)$ is given by
\begin{equation*}
\begin{aligned}
\min_{\mathcal{Y}\times\mathcal{P}} \widehat{\mathcal{L}}\left(u_\theta, p_\kappa\right)& =  \frac{|\Omega|}{n_d} \sum_{i=1}^{n_d}\left(\Delta u_\theta\left(X_i\right)+f\left(X_i\right)-\lambda^{-1} p_\kappa\left(X_i\right)\right)^2\\&\quad +\gamma_i \frac{|\Omega|}{n_d} \sum_{i=1}^{n_d}\left(\Delta p_\kappa\left(X_i\right)+u_\theta\left(X_i\right)-u_d\left(X_i\right)\right)^2 \\
& \quad +\gamma_b^u \frac{|\partial \Omega|}{n_b} \sum_{i=1}^{n_b}\left(u_\theta\left(Y_i\right)\right)^2+\gamma_b^p \frac{|\partial \Omega|}{n_b} \sum_{i=1}^{n_b}\left(p_\kappa\left(Y_i\right)\right)^2 .
\end{aligned}
\end{equation*}
The next result states an error bound on the DNN approximation.
\begin{theorem}
    Let the tuple $(\bar{u}, \bar{p}, \bar{z})$ solve the optimality system such that $\bar{u} \in H^s(\Omega) \cap$ $H_0^1(\Omega)$ and $\bar{p} \in H^s(\Omega) \cap H_0^1(\Omega)$ with $s \geq 3$, and $\left(u_{\theta^*}, p_{\kappa^*}, z_{\kappa^*}\right)$ be the DNN approximation. Fix the tolerance $\epsilon>0$, and take
    \begin{equation*}
      \mathcal{Y}=\mathcal{P}=\mathcal{N} \left(c \log (d+s), c(d, s) \epsilon^{-\frac{d}{-2-\mu}}, c(d, s) \epsilon^{-\frac{9+4 s}{2(s-2-\mu)}}\right).
    \end{equation*}
   Then by choosing $n_d=c(d, s) \epsilon^{-\frac{c(d+s) \log (d+s)}{s-2-\mu}}$ and $n_b=c(d, s) \epsilon^{-\frac{4(s d+s)}{s-2-\mu}-2}$ sampling points in $\Omega$ and on $\partial \Omega$, there holds
   \begin{equation*}
       \mathbb{E}_{\!X,\! Y\!}\left[\left\|\bar{u}-u_{\theta^*}\right\|_{L^2(\Omega)}^2+\left\|\bar{p}-p_{\kappa^*}\right\|_{L^2(\Omega)}^2+\left\|\bar{z}-z_{\kappa^*}\right\|_{L^2(\Omega)}^2\right] \leq c \epsilon^2,
   \end{equation*}
where $c$ depends on $\gamma, \lambda, d, s,\|f\|_{L^{\infty}(\Omega)},\left\|u_d\right\|_{L^{\infty}(\Omega)},\|\bar{u}\|_{H^s(\Omega)}$, and $\|\bar{p}\|_{H^s(\Omega)}$.
\end{theorem}

The error analysis involves two parts. First, let $(\bar{u}, \bar{p})$ be the solution tuple to the system $(2.4)$ with $\bar{z}=-\lambda^{-1} \bar{p}$. Then for any $\left(u_\theta, p_\kappa\right) \in$ $\mathcal{Y} \times \mathcal{P}$, with $z_\kappa=-\lambda^{-1} p_\kappa$, the following coercivity estimate holds
\begin{align*}
\left\|\bar{u}-u_\theta\right\|_{L^2(\Omega)}^2+\left\|\bar{p}-p_\kappa\right\|_{L^2(\Omega)}^2+\left\|\bar{z}-z_\kappa\right\|_{L^2(\Omega)}^2
\leq c(\boldsymbol{\gamma}, \lambda) \mathcal{L}\left(u_\theta, p_\kappa\right).
\end{align*}
Second, for any minimizing pair $(\hat{u}, \hat{p}) \in \mathcal{Y} \times \mathcal{P}$ of the empirical loss $\widehat{\mathcal{L}}(u, p)$, one can decompose the error of the loss $\mathbb{E}_{\mathrm{X}, \mathrm{Y}}[\mathcal{L}(\hat{u}, \hat{p})]$ by the approximation error $\inf _{(u, p) \in \mathcal{Y} \times \mathcal{P}} \mathcal{L}(u, p)$ and the quadrature error for the Monte Carlo approximation of the integrals. Dai et al employed the so-called offset Rademacher complexity \cite{Liang:2015,Duan:2023} to obtain a faster convergence rate of the statistical error.

\section{Summary and future directions}\label{sec:concl}

In this review, we have systematically discussed unsupervised learning approaches to solve (nonlinear) PDE  inverse problems, and have reviewed  three different approaches (including the FEM approach, the hybrid DNN-FEM approach and the DNN approach) on the model problem, i.e., diffusion coefficient identification. The focus is on numerical schemes with rigorous error analysis aligned with conditional stability.

First, the FEM discretization is typically based on reformulating the reconstruction task as regularized output least-squares fitting. The analysis relies on an energy argument using suitable test functions (e.g., exponential type test function), bounds of the objective  and positivity conditions on the weight. The analysis can deliver error bounds that are explicit in terms of the mesh size $h$, the noise level $\delta$ and the regularization parameter $\gamma$, which provide guidelines for choosing the algorithmic parameters. The obtained results are consistent with conditional stability results, under various assumptions on the problem data. Conditional stability directly motivates the error analysis.

Second, the hybrid DNN-FEM scheme approximates the unknown coefficient by the DNN and discretizes the state by the Galerkin FEM. The approach inherits the excellent expressivity of the DNN and  rigorous mathematical foundation of the FEM. However, the global support of DNN functions poses challenge in imposing the box constraint, and necessitates the use of quadrature schemes. The error analysis is similar to the FEM case. The convergence rate agrees with the stability result, if the quadrature error is made sufficiently small.

Third, we discuss two purely DNN-based approaches: the mixed least-squares DNN and PINNs. The error analysis takes into account both the approximation error of the DNN function class and the statistical error arising from the Monte Carlo approximation of the integrals. The approximation error is bounded using results from approximation theory (e.g., \cite{Guhring:2021}). The statistical error is analyzed using a PAC-type argument, using the covering number to bound the Rademacher complexity of the DNN function class. Together, the approximation error and the statistical error provide an \textit{a priori} estimate for the loss. The error of the reconstructed coefficient is further estimated using a weighted energy argument, inspired by the conditional stability result. For the PINN framework, the error analysis directly employs conditional stability results at the continuous level. However, the associated constants depend on the DNN architecture and require separate analysis. For different inverse problems, it is crucial to design the DNN loss function carefully according to the available stability estimates, and bound the error from a numerical perspective.

In sum, the use of DNNs to solve forward and inverse problems for PDEs has been explored only recently, and there has been important progress towards establishing theory-grounded numerical schemes. Despite the impressive  progress, there are still many interesting theoretical and practical issues that deserve further investigation.

First, the error analysis so far focuses on parameters with Sobolev regularity as in the FEM setting, and the error bounds suffers from the notorious curse of dimensionality.
One notable potential is to solve PDE inverse problems in the high-dimensional setting, for which it is imperative to derive realistic error bounds. Also it is of interest to investigate the optimality of the error bounds. The latter requires revisiting the approximation theory of DNNs, e.g., precisely characterizing the regime in which DNNs are superior to classical polynomial approximations.


Second, the error analysis focuses on the approximation error and statistical error in the context of the DNN based approaches, and ignores the optimization error. Due to the high nonlinearity of DNNs and non-convexity of the (empirical) loss, in practice one cannot guarantee obtaining a global minimizer. In practice, standard gradient type methods (e.g., Adam \cite{Kingma:2017}) work reasonably well but often fails to achieve very high precision. The presence of optimization errors has greatly complicated numerically verifying the error estimates as for conventional numerical schemes. It is of great importance to analyze the optimization error and its impact on the error bound, the connection with the empirically observed robustness with respect to data noise. The current convergence theory often builds on neural tangent kernel in an over-parameterized regime \cite{Jacot:2018,WangYuPerdikaris:2022}.

Third, the key of the error analysis is conditional stability of PDE inverse problems, e.g., Theorem \ref{thm:Bonito_stab} for diffusion coefficient identification. So far, the stability result utilizes the weak formulations of PDEs and a weighted energy argument. The error analysis mimics stability analysis and the convergence rate aligns with condition stability. Thus it is crucial to establish the stability at the continuous level that can then be applied to numerical analysis. This is a critical issue since often conditional stability results in theoretical PDE inverse problems require high regularity on the state or parameters \cite{Isakov:2006}. Indeed, Carleman estimates are often based on an exponential type weighted function and integration by parts, which is challenging to carry out at the discrete level  \cite{Klibanov:1991,Boyer:2010}, and discrete Carleman estimates often require restrictive conditions for the mesh size $h$.

Fourth, it is crucial to investigate inverse problems with poor stability, e.g. logarithmic stability or local H\"{o}lder stability. In practice, the measured data are often available only on a subdomain or on the boundary, which greatly deteriorates the stability of the inverse problem. Thus, it is of interest to establish  conditional stability of severely ill-posed problems in the low regularity regime and to derive error bounds.

Fifth and last, the discussions barely touch practical aspects of unsupervised learning approaches. Currently one notable challenge is the low computational efficiency: DNN-based methods often take much longer time to reach satisfactory reconstruction results, compared with the Galerkin FEM. This issue is especially pronounced for high-order PDEs. The training dynamics depends crucially on both the initialization scheme (e.g.,  Glorot (or Xavier) Initialization and He Initialization)  and the training algorithm. From a practical point of view, reducing the training error represents one major bottleneck in using DNNs to solve PDE inverse problems. So far the accuracy of DNN approximations is not very high (although more robust than the Galerkin FEM). The latter is often attributed to the presence of significant optimization errors: the optimizer fails to find a global minimizer. So there is an imperative need to develop more efficient and accurate implementations of the neural solvers for PDE inverse problems.

\section*{Acknowledgements}

The work of B. Jin is supported by Hong Kong RGC General Research Fund
(Project 14306423 and Project 14306824), and a
 start-up fund from The Chinese
University of Hong Kong.
The work of Z. Zhou is supported by by National Natural Science Foundation of China (Project 12422117) and Hong Kong Research Grants Council (15302323), and an internal grant of The Hong Kong Polytechnic University (Project ID: P0038888, Work Programme: ZVX3).

\section*{Appendix: Sobolev spaces}\label{sec:app}
Let $d\in\mathbb{N}$, $k\in\mathbb{N}_0:=\mathbb{N}\cup\{0\}$ and $1\leq p\leq \infty$. Let $\Omega\subset\mathbb{R}^d$ be an open bounded domain with a boundary $\partial\Omega$. Given a function $v:\Omega \to \mathbb{R}$ and a multi-index $\alpha\in\mathbb{N}_0^d$ (with $|\alpha|=\sum_{i=1}^d\alpha_i$), we denote by
\begin{equation*}
    D^\alpha v = \frac{\partial^{|\alpha|}v}{\partial x_1^{\alpha_1}\cdots\partial x_d^{\alpha_d}}
\end{equation*}
the distributional (weak) derivative of $v$. We denote by $L^p(\Omega)$ the usual Lebesgue space, and for $k\geq 1$, we define the Sobolev space $W^{k,p}(\Omega)$ by
\begin{equation*}
    W^{k,p}(\Omega) = \{v\in L^p(\Omega): D^\alpha v\in L^p(\Omega), \forall \alpha\in \mathbb{N}_0^d\mbox{ with }
|\alpha|\leq k\}.
\end{equation*}
On the space $W^{p,p}(\Omega)$, we define the following semi-norms for $m=0,\ldots,k$,
\begin{equation*}
    |v|_{W^{m,p}(\Omega)}=\left\{\begin{aligned}
 &\left(\sum_{|\alpha|=m}\|D^\alpha v\|_{L^p(\Omega)}^p\right)^{1/p},\quad &&1\leq p<\infty,\\
 &\max_{|\alpha|=m}\|D^\alpha v\|_{L^\infty(\Omega)}, \quad &&p=\infty,
 \end{aligned}\right.
\end{equation*}
and the following norm
\begin{equation*}
    \|v\|_{W^{k,p}(\Omega)} =\left\{\begin{aligned}
    &\left(\sum_{m=0}^k|v|_{W^{m,p}(\Omega)}^p\right)^{1/p},\quad &&1\leq p<\infty,\\
   &\max_{0\leq m\leq k}|f|_{W^{m,\infty}(\Omega)}, \quad &&p=\infty.
    \end{aligned}\right.
\end{equation*}
The space $W^{k,p}(\Omega)$ equipped with the norm $\|\cdot\|_{W^{k,p}(\Omega)}$ is a Banach space, and when $p=2$, it is a Hilbert space.

We denote by $C^k(\overline{\Omega})$ the space of functions that are continuously differentiable (up to the boundary $\partial\Omega$), and equip the space with the norm $\|v\|_{C^k(\overline\Omega)}=\|v\|_{W^{k,p}(\Omega)}$.

\bibliographystyle{abbrv}

\end{document}